\documentclass[12pt,reqno]{article}
\usepackage{fullpage}
\usepackage{comment}
\usepackage{enumerate}
\usepackage{thmtools}
\usepackage{mathtools}
\usepackage[normalem]{ulem}
\usepackage{appendix}
\usepackage{amssymb}
\usepackage{amsthm}
\usepackage[normalem]{ulem}
\usepackage[mathscr]{euscript}
\usepackage{pgf,tikz}
\usetikzlibrary{arrows}
\usepackage{float}
\usepackage{hyperref}
\usepackage{cite}
\usepackage{color}
\usepackage{bbm}
\usepackage{xcolor}
\usepackage{transparent}
\usepackage{subfigure}
\usepackage{upgreek} 
\usepackage{amsfonts,txfonts}
\usepackage{bm}
\usepackage{graphicx}                       
\usepackage{epstopdf}
\usepackage{psfrag}
\usepackage{thm-restate}                   
\usepackage{cleveref}

\definecolor{azure(colorwheel)}{rgb}{0.0, 0.5, 1.0}
\definecolor{hanpurple}{rgb}{0.32, 0.09, 0.98}
\definecolor{iris}{rgb}{0.35, 0.31, 0.81}
\definecolor{byzantine}{rgb}{0.74, 0.2, 0.64}
\definecolor{ashgrey}{rgb}{0.7, 0.75, 0.71}
\definecolor{battleshipgrey}{rgb}{0.52, 0.52, 0.51}
\hypersetup{colorlinks=true,pdfborder=000,  citecolor  = blue, citebordercolor= {magenta}}
\hypersetup{linkcolor=black}

\makeatletter
\let\reftagform@=\tagform@
\def\tagform@#1{\maketag@@@{(\ignorespaces\textcolor{purple}{#1}\unskip\@@italiccorr)}}
\renewcommand{\eqref}[1]{\textup{\reftagform@{\ref{#1}}}}
\makeatother
\makeatletter

\DeclareUrlCommand\ULurl@@{%
  \def\UrlLeft{\uline\bgroup}%
  \def\UrlRight{\egroup}}
\def\ULurl@#1{\hyper@linkurl{\ULurl@@{#1}}{#1}}
\DeclareRobustCommand*\ULurl{\hyper@normalise\ULurl@}
\makeatother

\def\lessim{\ \lower4pt\hbox{$
		\buildrel{\displaystyle <}\over\sim$}\ }
\def\gessim{\ \lower4pt\hbox{$\buildrel{\displaystyle >}
		\over\sim$}\ }

\numberwithin{equation}{section}

\newcommand{\R}{\mathbb{R}}
\newcommand{\E}{\mathbb{E}}

\newcommand{\Prob}{\mathbb{P}}

\newcommand{\Z}{\mathbb{Z}}

\newcommand{\supp}{\mathrm{supp} \hspace{2pt}}

\newcommand{\tendsto}[2]{\xrightarrow[#1 \to #2]{}}

\newcommand{\ind}{\mathbbm{1}}

\newcommand{\diam}{\mathrm{diam}}
\newcommand{\Aut}{\mathrm{Aut}}

\newtheorem{lemma}{Lemma}[section]

\newtheorem{defn}{Definition}
\newtheorem{thm}{Theorem}[section]
\newtheorem{rmk}{Remark}[section]
\newtheorem{prop}{Proposition}[section]
\newtheorem{cor}{Corollary}[section]
\newtheorem{claim}{Claim}

\graphicspath{ {./figures/} }

\begin{document}
\author{Christian Gorski \thanks{Department of Mathematics, Northwestern University, Email: christiangorski2022@u.northwestern.edu, research partially supported by NSF Grant DMS-1502632 RTG: Analysis on Manifolds.}
   \\
	\small{Northwestern University}}

\title{Strict monotonicity for first passage percolation on graphs of polynomial growth and quasi-trees}
\maketitle
\begin{abstract}
   In \cite{vdBK} van den Berg and Kesten prove a strict monotonicity theorem for first passage percolation on $\Z^d$, $d \ge 2$:
   given two probability measures $\nu$ and $\tilde{\nu}$ with finite mean, if $\tilde{\nu}$ is strictly more variable
   than $\nu$ and $\nu$ is subcritical in an appropriate sense, the time constant associated to $\tilde{\nu}$ is strictly smaller than
   the time constant associated to $\nu$.
   In this paper, an analogous result is proven for (not necessarily almost-transitive) graphs of strict polynomial growth and for bounded degree graphs
   quasi-isometric to trees which satisfy a certain geometric condition we call ``admitting detours.''
   It is also proven that if a bounded degree graph does not admit detours, then such a strict monotonicity theorem with respect to variability cannot hold.
   Large classes of graphs are shown to admit detours, and we conclude that for example any Cayley graph of a virtually nilpotent
   group which is not isomorphic to the standard Cayley graph of $\Z$ satisfies strict monotonicity with respect to variability, as does any Cayley graph
   of $F \rtimes F_k$, $F$ a nontrivial finite group and $F_k$ a free group.
   
   Moreover, it is proven that for graphs of strict polynomial growth and bounded degree graphs
   quasi-isometric to trees, if the weight measure is subcritical in an appropriate sense, then it is ``absolutely continuous with respect to
   the expected empirical measure of the geodesic.'' This implies
   a strict monotonicity theorem with respect to stochastic domination of measures, whether or not
   the graph admits detours.
   \end{abstract}

\section{Introduction}
Given an infinite graph $G$ and a probability measure $\nu$ supported on $[0,\infty)$, we can create a random metric $T$
on the vertex set $V$ of $G$ in a natural way; the ``length'' or ``weight'' of each edge is an independent random variable sampled
according to $\nu$, and the distance between two vertices is simply the total ``length'' of the ``shortest'' edge path between them.
This model is called first passage percolation and was introduced by Hammersley and Welsh \cite{HW} as a
model for the spread of fluid through a porous medium. These distances are thought of as the time it takes for fluid
to flow from one point to another, and $T(x,y)$ is often called the ``passage time''. For a survey on this model,
the reader is directed to \cite{Aspects, ADH}.

One is interested in the large scale geometry of the random metric $T$; 
for instance, if $G=(V,E)$ is a Cayley graph of a virtually nilpotent
group, then under mild assumptions the sequence of random metric spaces $(V, \frac{1}{t} T)$ almost surely converges
to a deterministic limit space $(G_{\infty}, d_{\infty})$ (\cite{BenjaminiTessera},
\cite{CD, Aspects} for the standard Cayley graph of $\Z^d$) as $t \to \infty$.
In the case that the group is $\Z^d$,
the limit space is $\R^d$ with some norm. 
The norm is determined by the ``time constants''
\[
   \mu_v := \lim_{n \to \infty} \frac{ T(0,nv) }{ n } = \lim_{n \to \infty} \frac{ \E T(0,nv) }{n},
\]
where $v$ ranges over unit vectors in $\R^d$.
Due to the second equality above, we see that the scaling limit only depends on $T$ through $\E T$,
and in fact throughout this paper we focus on $\E T$ rather than $T$.

The metric $\E T$ does depend on the weight distribution $\nu$, and understanding this dependence is the subject of this paper.
Van den Berg and Kesten \cite{vdBK} showed that if a probability measure $\tilde{\nu}$ is ``strictly more variable'' than
a probability measure $\nu$, both have finite mean, and $\nu$ is subcritical (in a sense to be described later), then for $G$ the standard Cayley graph of $\Z^d$, $d \ge 2$,
we have a strict inequality of time constants $\tilde{\mu}_v < \mu_v$ for all $v \ne 0$.
In fact, their proof shows that
\[
   \liminf_{d(x,y) \to \infty} \frac{ \E T(x,y) - \E \tilde{T}(x,y) }{d(x,y)} > 0,
\]
where $d(x,y) := \sum_{i=1}^d |x_i - y_i|$ is the graph distance between the vertices $x,y \in \Z^d$ in the standard Cayley graph of $\Z^d$.
Note that this inequality makes sense for any graph; one does not require the existence of any scaling limits or
even ``time constants.'' We will often abbreviate the above inequality as $\E \tilde{T} \ll \E T$.
This naturally raises the question: for which other graphs $G$ does the same conclusion hold?
Let us say that such graphs have the \emph{van den Berg-Kesten (vdBK) property}. 
(For precise definitions of ``more variable'', ``subcritical'', and the vdBK property, see Section \ref{sec:def}). 
A main result of this paper is a generalization of van den Berg-Kesten's result:

\begin{thm} \label{thm:nilpotentvdBK}
   Let $G$ be any Cayley graph of a finitely generated virtually nilpotent group.
   If $G$ is not isomorphic as a graph to the standard Cayley graph of $\Z$, then $G$ has the vdBK property.
\end{thm}

One might wonder if \emph{all} graphs have the vdBK property. This is not true; the easiest counterexample is when
the graph $G$ is a tree. In this case, since there is only one self-avoiding path between any two points, we have
that $\frac{\E T(x,y) }{ d(x,y) }$ is a \emph{constant} equal to the mean of $\nu$.
It is easy to produce two different probability measures $\nu, \tilde{\nu}$ with the same mean such that $\tilde{\nu}$
is more variable than $\nu$ (see the proof of Theorem \ref{thm:notvdBK}).
This is, of course, why the standard Cayley graph of $\Z$ had to be excluded from the above theorem.

However, trees are not the only counterexample. Consider the Cayley graph of the free group $F(a,b)$ on the
two letters $a,b$ which is associated to the [redundant] generating set $\{a, b, ab\}$ (see Figure \ref{fig:F2Cayley}).
It is not hard to see that, although there is more than one self-avoiding path between any two points, each self-avoiding
path between two points must pass through every vertex of the edge-geodesic path between those two points,
and that each step, one only has the choice to travel along the edge lying in the geodesic, or to take a particular path of length
two. Hence, in this case $\frac{ \E T(x,y) }{ d(x,y) }$ is a constant given by $\E \min( w_1, w_2 + w_3 )$, where $w_1,w_2,w_3$
are independent variables with distribution $\nu$. One can again produce two distinct distributions, one more variable than
the other, such that their ``time constants'' are equal, contradicting the vdBK property.

\begin{figure}[t]
   \centering
   \includegraphics[scale=.45]{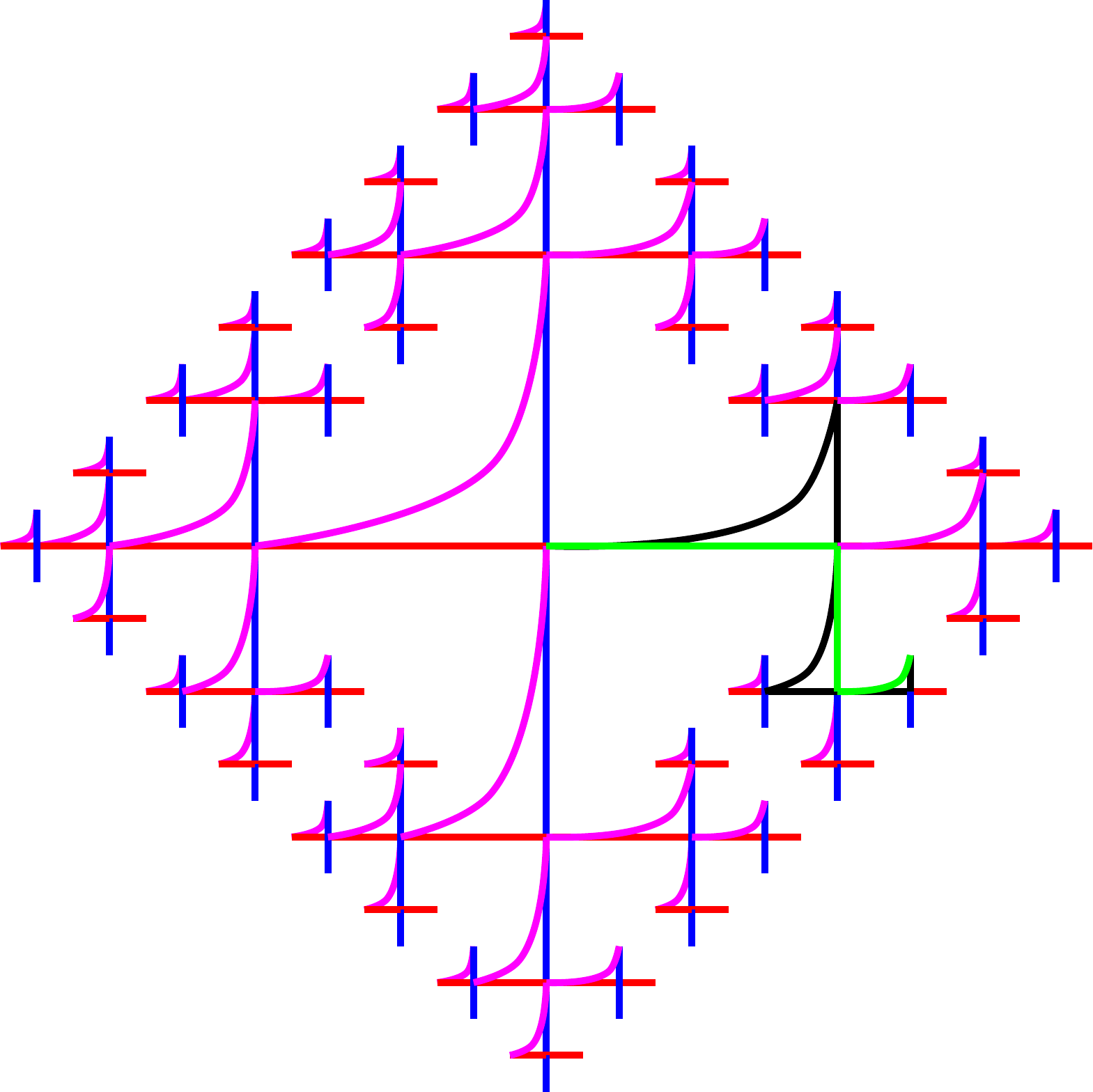}
   \caption{The Cayley graph of the free group $F(a,b)$ with respect to the generating set
   \{a,b,ab\}. In green is the unique edge-geodesic path from 1 to $ab^{-1}ab$. Every self-avoiding 
   path from 1 to $ab^{-1}ab$ in this graph must visit all the vertices of the green
   path and can only use green or black edges.}
   \label{fig:F2Cayley}
\end{figure}

It turns out that the crucial property for determining whether a graph is vdBK is a property of the graph which we
call ``admitting detours'' (defined in Section \ref{sec:detours}). If a bounded degree graph does not admit detours, then it is not vdBK
(Theorem \ref{thm:notvdBK} below). On the other hand, our main theorems prove that given one of two quite different ``large scale''
assumptions on the geometry of $G$, admitting detours \emph{implies} the vdBK property:

\begin{restatable}{thm}{polygrowthvdBK} \label{thm:polygrowthvdBK}
   Let $G$ be a graph of strict polynomial growth. Then $G$ is vdBK if and only if $G$ admits detours.
\end{restatable}

\begin{restatable}{thm}{qitree} \label{thm:qitree}
   Let $G$ be a bounded degree graph which is quasi-isometric to a tree.
   Then $G$ is vdBK  if and only if $G$ admits detours. In fact, if $G$ admits detours, then whenever $\tilde{\nu}$ is strictly more variable than $\nu$, we have
   $\E \tilde{T} \ll \E T$.
\end{restatable}

The first theorem is used to prove Theorem \ref{thm:nilpotentvdBK}: the latter follows from the former once we prove
that all Cayley graphs of virtually nilpotent groups which are not isomorphic to the standard Cayley graph of $\Z$ admit
detours, which is proven in Appendix \ref{app:grouptheory}.
The second theorem can be combined with results proved in Appendix \ref{app:grouptheory} to show the corollary:
\begin{thm} \label{thm:virtfree}
   Let $G$ be a Cayley graph for a group $\Gamma$ which is virtually free, 
   not isomorphic to $\Z$ or $\Z/2 * \Z/2$, and which either contains
   a finite index subgroup with nontrivial center or contains a nontrivial finite normal 
   subgroup. Then $G$ has the van den Berg-Kesten property.
   For example, if $\Gamma$ contains $F_k \times F$ as a finite-index subgroup, where $F_k$ is the free group on $k \ge 1$ letters
   and $F$ is any nontrivial finite group, or if $\Gamma$ is isomorphic to a semidirect product
   $F \rtimes F_k$, then any Cayley graph of $\Gamma$ is vdBK.
\end{thm}

Finally, as noted in \cite{vdBK}, the question of strict inequalities of time constants is related to
``absolute continuity with respect to the expected empirical measure.'' What precisely we mean by this is explained in Section \ref{sec:abscont};
note that this condition does not imply \emph{existence} of a limiting expected empirical measure. In any case, the methods of our
paper easily prove absolute continuity of the weight distribution with respect to the expected empirical measure under the same ``large-scale'' assumptions (see
the Section \ref{sec:def} for the definition of exponential-subcriticality):

\begin{restatable}{thm}{qitreeabscont} \label{thm:qitreeabscont}
   Let $G$ be a bounded degree graph which is quasi-isometric to a tree. Then for any probability measure $\nu$
   on $[0, \infty)$ with finite mean, $\nu$ is absolutely continuous with respect to
   the expected empirical measure of the associated first passage percolation $T$.
   Moreover, if $\nu$ strictly stochastically dominates a measure $\tilde{\nu}$ with finite mean, then $\E \tilde{T} \ll \E T$.
\end{restatable}

\begin{restatable}{thm}{polygrowthabscont} \label{thm:polygrowthabscont}
   Let $G$ be a graph of strict polynomial growth. Suppose that $\nu$ has finite mean 
   and is exponential-subcritical. 
   Then $\nu$ is absolutely continuous with respect to
   the expected empirical measure of the associated first passage percolation $T$.
   Moreover, if $\nu$ strictly stochastically dominates a measure $\tilde{\nu}$ with finite mean, then $\E \tilde{T} \ll \E T$.
\end{restatable}

The layout of the paper is as follows: in Section \ref{sec:def} we establish definitions and notations, particularly the definition of the vdBK property.
In Section \ref{sec:general} we collect various lemmata, most of which are essentially proven in \cite{vdBK}, which
we will need to prove our main theorems. The key conclusion of this section is that, in order to prove that $\E \tilde{T} \ll \E T$, it suffices to show 
that the expected number of times the $T$-geodesic between two points $x$ and $y$ passes through a certain type of configuration 
called a ``feasible pair''
is \emph{linear} in $d(x,y)$. In Section \ref{sec:detours} we introduce the concept of ``admitting detours,'' show that this
is a necessary condition for a graph to be vdBK, and then give examples of graphs which admit detours.
In Section \ref{sec:qitrees} we prove Theorem \ref{thm:qitree}, which will follow almost immediately from the results of Section \ref{sec:general} 
combined with a characterization of graphs quasi-isometric to trees.
Because paths are so constrained in this setting, it is not hard to produce local events which imply that the $T$-geodesic from $x$ to $y$
passes through a feasible pair, and this makes the proof quite simple.

In Section \ref{sec:polygrowth} we prove Theorem \ref{thm:polygrowthvdBK}, which is much more involved.
The three key components are a Peierls-type lemma, a resampling argument, and a ``geometric construction'' (a 
construction of a set of weights suitable for use in the resampling argument).
Although this general strategy is the same as in \cite{vdBK}, the methods given here apply to general graphs
of strict polynomial growth which are not necessarily almost-transitive.
The geometric constructions in particular are rather different from those of \cite{vdBK} and are quite involved,
since we are given the task of manipulating the geodesic while remaining largely agnostic to the fine geometry of the graph.
Indeed, these geometric constructions are the most involved part of the proof of Theorem \ref{thm:polygrowthvdBK}.
At the end of this section we give some examples of graphs which are not almost-transitive to which our results apply.

Lastly, in Section \ref{sec:abscont} we prove absolute continuity with respect to the expected empirical measure for graphs of strict polynomial growth
and for graphs quasi-isometric to trees, which implies a strict monotonicity theorem
with respect to stochastic domination, regardless of whether the graph in question
admits detours. The proofs are just easier versions of the proofs of the main theorems of the paper.
Appendix \ref{app:grouptheory} gives proofs of the statements in Section \ref{sec:detours} regarding which Cayley graphs admit detours.

\section{Passage times and the van den Berg-Kesten property} \label{sec:def}

By a \emph{graph} we mean a pair $G = (V,E)$ of sets and an ``endpoint'' or ``boundary'' map from $E$ 
to the set of subsets of $V$ of size 2.
In particular, we allow more than one edge between each pair of vertices but we do not allow self-loops. (Disallowing self-loops
is simply a matter of convenience; virtually all questions considered here are easily seen to be equivalent for a graph $G$ with 
self-loops and the graph $G'$ obtained from $G$ by deleting all self-loops).
Throughout, the ``ambient graph'' $G$ is tacitly assumed to be connected, locally finite (i.e. each vertex has finite degree)
and infinite (that is, $V$ is countably infinite); we will often however consider subgraphs of $G$ which are finite and/or disconnected.

A path $\pi$ in $G$ is an alternating sequence of vertices and edges (starting and ending with a vertex) such that
the vertices immediately preceding and following an edge comprise the edge's boundary.
If $\pi$ starts at $x \in V$ and ends at $y \in V$, we often write $\pi: x \to y$.
We will typically abuse notation and use the same symbol $\pi$ to refer to the set of edges appearing in the path $\pi$ (so $\pi \subset E$).
$|S|$ denotes the cardinality of the set $S$, so in particular, if $\pi$ is a path, $|\pi|$ is the number of edges appearing in the path
(again abusing notation and considering $\pi$ as a subset of $E$). If $\pi$ does not contain any repeated edges,
then this agrees with the usual notion of length of a path. In fact, we will mostly be concerned with paths
which do not have any repeated vertices; we call such paths \emph{self-avoiding} (or \emph{vertex-self-avoiding}).


A graph $G$ gives a natural metric on $V$ by
\[
   d(v,w) := \inf \{ |\gamma| : \gamma: v \to w \} = \inf \{ |\gamma| : \gamma: v \to w \mbox{ self-avoiding} \}.
\]
We write $B(x,R)$ for the ball $\{ y \in V : d(x,y) \le R \}$ in this metric and write $S(x,R)$ for the sphere $\{ y \in V : d(x,y) = R\}$.

More generally, given a function $w:E \to [0,\infty)$ we can define
\[
   T(\pi) := \sum_{e \in \pi} w(e).
\]
(For general $\pi$, the sum should be over the \emph{sequence} of edges given by $\pi$; if $\pi$ contains no repeated edges,
this may again be considered as an edge \emph{set}). We then get a pseudo-metric on $V$ given by
\[
   T(v,w) := \inf \{ T(\gamma) : \gamma: v \to w \} = \inf \{ T(\gamma) : \gamma: v \to w \mbox{ self-avoiding} \}.
\]
We call $w(e)$ the \emph{weight} of the edge $e$.
We often call $T(\pi)$ the \emph{passage time} of the path $\pi$ and we call $T(v,w)$ the \emph{passage time} from $v$ to $w$.
If $\nu$ is a probability measure on $[0,\infty)$ then we get a \emph{random} function $w:E \to [0,\infty)$ 
by taking the $\{ w(e) \}_{e \in E}$ to be an independent family of $\nu$-distributed random variables.
This gives a \emph{random} pseudometric $T$ on $V$. This model is called \emph{[independent] first passage percolation}.
Throughout the paper, $w$ and $T$ will represent the random weights and pseudo-metric given by a probability measure $\nu$.
Similarly, $\tilde{w}$ and $\tilde{T}$ will represent the random weights and pseudo-metric given by a probability measure $\tilde{\nu}$,
and similarly for any other diacritics.


Let $\nu$ and $\tilde{\nu}$ be two probability measures on $[0,\infty)$. We say that $\tilde{\nu}$ is \emph{more variable} than $\nu$
if for every concave nondecreasing function $f: \R \to \R$ we have
\[
   \int f d\tilde{\nu} \le \int f d\nu
\]
as long as both integrals converge absolutely. We say that $\tilde{\nu}$ is \emph{strictly more variable} than $\nu$
if $\tilde{\nu}$ is more variable than $\nu$ and $\tilde{\nu} \ne \nu$.


We now define some percolation thresholds associated to a graph $G$. For $p \in [0,1]$, denote by $G_p$ the random subgraph
of $G$ given by including each edge $e \in E(G)$ in $G_p$ independently with probability $p$, excluding with probability $1-p$.
We define the \emph{exponential percolation threshold} for $G$ to be
\[
   \underline{p_c} := 
   \sup 
   \left\{ p \in [0,1] : \limsup_{R \to \infty} \sup_{o \in V} \frac{1}{R} \log \Prob(G_p \mbox{ contains an edge path from } o \mbox{ to } B_G(o,R)^c) < 0 \right\}
\]
and we define the \emph{exponential geodesic percolation threshold} for $G$ to be
\[
   \vec{\underline{p_c}} :=
   \sup \left\{ p \in [0,1] : \limsup_{R \to \infty} \sup_{o \in V} \frac{1}{R} \log 
   \Prob \left( \begin{array}{c} G_p \mbox{ contains an edge path from } o \mbox{ to } \\
                                              B_G(o,R)^c \mbox{ which is edge-geodesic in } G
                     \end{array} \right) < 0 \right\}.
\]
Below the exponential percolation thresholds, we have uniform exponential upper bounds on connection events.
We call a measure $\nu$ \emph{exponential-subcritical} if either $\inf := \inf \supp \nu = 0$ and $\nu(\{0\}) < \underline{p_c}$
or $\inf > 0$ and $\nu(\{\inf\}) < \vec{\underline{p_c}}$.

\begin{defn}
   We say that an infinite graph $G$ \emph{has the van den Berg-Kesten (vdBK) property}
   if for every $\nu, \tilde{\nu}$ with finite mean such that 
   $\nu$ is exponential-subcritical and $\tilde{\nu}$ is strictly more variable than $\nu$, we have
   \begin{equation} \label{eq:strictineq}
      \liminf_{d(x,y) \to \infty} \frac{ \E T(x,y) - \E \tilde{T}(x,y) }{d(x,y)} > 0.
   \end{equation}
\end{defn}
We will often abbreviate the ``asymptotic strict inequality'' \eqref{eq:strictineq} as $\E \tilde{T} \ll \E T$.
The main theorems of this paper give sufficient or necessary conditions for a graph to be vdBK.

\subsection{Remarks on the condition of exponential subcriticality}
   Here are some remarks which, while not necessary to the proofs below, are
   worth noting, on the condition of exponential subcriticality, its relationship to various other percolation thresholds,
   and the extent to which it is a necessary assumption to get a strict monotonicity result.
   
   First, it is clear from the definitions that for any graph, $\vec{\underline{p_c}} \ge \underline{p_c}$, and a simple union bound (counting self-avoiding paths
   from a fixed vertex) shows that if $G$ has degree at most $D$, then $\underline{p_c} \ge 1/D > 0$.
   It is also clear that for any connected graph, $\underline{p_c} \le p_c$, where $p_c$ is the percolation threshold as usually defined:
   \[
      p_c := \inf \{ p \in [0,1] : \Prob( G_p \mbox{ contains an infinite edge path from } o ) > 0 \}.
   \]
   For almost-transitive graphs, the sharpness of the percolation threshold \cite{DuminilCopinTassion} shown by Duminil-Copin and
   Tassion implies that $\underline{p_c} = p_c$.
   (The proof of sharpness in \cite{DuminilCopinTassion} is stated for transitive graphs, but is not hard to generalize to almost-transitive graphs;
   here by \emph{almost-transitive graph} we mean a graph $G$ such that the action of $\Aut(G)$ on $V$ has finitely many orbits.)
   
   Furthermore, on amenable almost-transitive graphs (in particular graphs of polynomial growth),
   the original argument of Burton-Keane (\cite{BurtonKeane},
   see also \cite{HaggstromJonasson} for an explicitly general proof) shows that $p_c = p_u$, where $p_u$ is the \emph{uniqueness threshold}
   \[
      p_u := \inf \{ p \in [0,1] : \Prob( G_p \mbox{ contains a unique infinite connected component } ) = 1 \}.
   \]
   If $\nu(\{0\}) \ge p_u$, then one expects that $\lim_{d(x,y) \to \infty} \frac{ \E T(x,y) }{ d(x,y) } = 0$ (although this has 
   only been proven in certain cases, see e.g. Theorem 6.1 of \cite{Aspects}). In that case, it is impossible that $\E \tilde{T} \ll \E T$,
   so for polynomial growth almost-transitive graphs, the assumption on the atom at 0 is really as weak as one could hope for.
   
   In the case of the standard Cayley graph of $\Z^d$, $\vec{\underline{p_c}}$ is the classical oriented percolation threshold $\vec{p_c}$; this is because
   of the nature of edge-geodesics in this graph, combined with the sharpness results of Aizenman and Barsky \cite{AizenmanBarsky}.
   In fact, if $G$ is the standard Cayley graph of $\Z^d$, the condition here of being exponential-subcritical is precisely the condition of being
   ``useful'' in \cite{vdBK}.
   Furthermore, in this case, if $\inf \supp \nu := a > 0$ and $\nu(\{a\}) \ge \vec{p_c}$, then $\lim_{n \to \infty} \frac{ \E T(0, (n,...,n)) }{ dn } = a$ \cite{DurrettLiggett, Marchand}.
   So if $\nu \ne \delta_a$, we can
   take $\tilde{\nu} = \delta_a$ to get $\tilde{\nu}$ strictly more variable than $\nu$ but $\E \tilde{T} \nll \E T$.
   Thus, the assumption $\nu(\{a\}) < \vec{\underline{p_c}}$ in the definition of the vdBK property is also necessary at least in this setting.\footnote{
   Of course, this is because $\E \tilde{T} \ll \E T$ is equivalent to a strict inequality of time constants \emph{in all directions simultaneously};
   if one instead only cares about strict inequality of a time constant in a fixed direction, the assumption $\nu(a) < \vec{\underline{p_c}}$ may not be necessary;
   for instance Marchand \cite{Marchand} proved that for $G$ the standard Cayley graph of $\Z^2$, we get strict inequality in the $e_1$ direction
   without that assumption (as long as still $\nu(\{0\}) < p_c$).}
   
   It is reasonable to conjecture that similar behavior happens more generally, i.e. that for, say, almost-transitive polynomial growth graphs
   one has $\liminf_{d(x,y) \to \infty} \frac{ \E T(x,y) }{ d(x,y) } = \inf =: \inf \supp \nu$ whenever $\inf > 0$ and $\nu(\{\inf\}) \ge \vec{\underline{p_c}}$,
   but to show this would require defining an appropriate analogue of the oriented percolation threshold (which will likely not in general be associated
   to a literal oriented percolation model) and showing some sort of sharpness, and this is not explored here.
   
   On the other hand, for graphs quasi-isometric to a tree, we will see in Theorem \ref{thm:qitree} that exponential subcriticality as defined here is not necessary at all.
   In fact, in this setting, if $G$ admits detours (see Section \ref{sec:detours}), then
   $\E \tilde{T} \ll \E T$
    whenever $\tilde{\nu}$ is strictly more variable
   than $\nu$, with no further assumptions needed on either measure.
   This is consistent with the perspective that generally the \emph{uniqueness} threshold, rather than $p_c$, is the correct threshold to consider for the atom at 0,
   since almost-transitive graphs quasi-isometric to trees can have $p_c <1$ but always have $p_u = 1$ (since they have more than one end, see
   page 86 of \cite{HaggstromPeresSchonmann}). The proper ``uniqueness'' analogue of $\vec{\underline{p_c}}$ outside of the amenable
   case is unclear.
   
   Finally, percolation on graphs which are not almost-transitive is poorly understood, and so it is entirely unclear how close exponential
   subcriticality is to the ``right'' condition on $\nu$ to consider in this general setting. However, if $G$ has degree at most $D$ then
   the inequalities $\vec{\underline{p_c}} \ge \underline{p_c} \ge 1/D > 0$ tell us that our main theorems are never vacuous for bounded degree graphs;
   in particular, we get sufficient conditions to conclude strict monotonicity, even if the parameters $\vec{\underline{p_c}}$ and $\underline{p_c}$ are quite mysterious.

\section{Reduction to a lower bound on expected number of traversed ``feasible pairs''}
\label{sec:general}
In this section, we reduce the task of deducing a strict inequality $\E \tilde{T} \ll \E T$ to the task of showing that the $T$-geodesic traverses linearly many ``feasible pairs''
in expectation. Most of the argument from this section can be transferred directly from \cite{vdBK}, with the main difference being that we define a weaker notion
of ``feasible pair.'' Proofs are given where the necessary modifications from \cite{vdBK} are not obvious.
Note that the arguments of this section allow us to stop considering $\tilde{w}$ or $\tilde{T}$ and simply focus on understanding the $T$-geodesic.

First, note that we have the following theorem of van den Berg and Kesten:
\begin{thm}[\cite{vdBK}, Theorem 2.9a]
   Let $\nu$ and $\tilde{\nu}$ be probability measures on $[0,\infty)$ with finite mean such that $\tilde{\nu}$ is more variable than $\nu$. Then for all $x,y \in V$
   \[
      \E \tilde{T}(x,y) \le \E T(x,y).
   \]
\end{thm}
\noindent
Although the proof in \cite{vdBK} is stated only for $G = \Z^d$, it easily extends to all locally finite graphs.

We also have
\begin{thm}[\cite{Strassen, Whitt}]
   Let $\nu$ and $\tilde{\nu}$ be probability measures on $[0,\infty)$ with finite mean such that $\tilde{\nu}$ is strictly more variable than $\nu$.
   Then there exists a coupling $(w(e), \tilde{w}(e))$ such that $w(e)$ is $\nu$-distributed, $\tilde{w}(e)$ is $\tilde{\nu}$-distributed,
   and 
   \[
      \E[ \tilde{w}(e) | w(e) ] \le w(e)
   \]
   almost surely.
\end{thm}

Another lemma from \cite{vdBK} which we will need is the following:
\begin{lemma}[\cite{vdBK}, Lemma 4.5] \label{lem:wlog}
   We may assume without loss of generality that in our coupling 
   \begin{equation} \label{eq:extraassumption}
      \Prob( \tilde{w}(e) > w(e) ) > 0.
   \end{equation}
   Explicitly, either this holds, or there exists some $\bar{w}(e)$ such that $\E \tilde{T}(x,y) \le \E \bar{T}(x,y)$ for all $x,y \in V$,
   the distribution of $\bar{w}(e)$ is strictly more variable than $\nu$, and such that \eqref{eq:extraassumption}
   holds with $\tilde{w}$ replaced by $\bar{w}$, i.e. $\Prob( \bar{w}(e) > w(e) ) > 0$.
\end{lemma}

A key technical lemma we will use is the following:
\begin{lemma}[\cite{vdBK}, Lemma 4.8] \label{lem:technicallemma}
   Let $\nu$, $\tilde{\nu}$ be probability measures with finite mean such that $\tilde{\nu}$ is more variable than $\nu$ and such that 
   \eqref{eq:extraassumption} holds. Then there exist $\epsilon > 0$, $a>0$, $b>0$, $g > 0$, and
   a bounded Borel set $I_0 \subset [0, \infty)$ and $y_0 \in I_0$ with the following properties:
   \begin{itemize}
      \item For all $\delta > 0$, $\nu(I_0 \cap (y_0 - \delta, y_0 + \delta)) > 0$.
      \item For all $y \in I_0$,
      \[
         \Prob( \tilde{w}(e) > y + a | w(e) = y ) \ge b.
      \]
      \item For any $k \ge 1$, $\delta > 0$, and 
      any $y_1,...,y_k, y'_1,...,y'_{\lfloor (1 + \epsilon) k \rfloor} \in I_0 \cap (y_0 - \delta, y_0 + \delta)$, we have
      \[
         \sum_{i=1}^{k} (y_i + a) - \sum_{i=1}^{\lfloor (1 + \epsilon) k \rfloor} y'_i > kg \ge g.
      \]
   \end{itemize}
\end{lemma}
\begin{proof}
   The proof is essentially the same as that given in \cite{vdBK}, but simpler since we do not actually
   need as many conditions. By \eqref{eq:extraassumption}, 
   for some sufficiently small $a,b > 0$ there is
   some Borel set $B \subset [0, \infty)$ such that $\nu(B) > 0$ and for all $y \in B$,
   \[
      \Prob(\tilde{w}(e) > y + a | w(e) = y) \ge b.
   \]
   Let $y_0$ be a point of support for $B$, that is, a point such that $\nu(B \cap (y_0 - \delta, y_0 + \delta)) > 0$
   for all $\delta > 0$. Choose $\epsilon > 0$ sufficiently small such that $\epsilon y_0 < a$. Then choose $\delta_0 > 0$
   sufficiently small that
   \[
     \epsilon y_0 + 2 \delta_0 + \epsilon \delta_0 < a,
   \]
   and choose
   \[
      0 < g < a - (\epsilon y_0 + 2 \delta_0 + \epsilon \delta_0).
   \]
   Then we can take $I_0 := B \cap (y - \delta_0, y + \delta_0)$.
   The first two conditions clearly hold by construction; let us show the last condition:
   \begin{align*}
      \sum_{i=1}^{k} (y_i + a) - \sum_{i=1}^{\lfloor (1 + \epsilon) k \rfloor} y'_i &\ge k(y_0 - \delta_0 + a) - k(1 + \epsilon)(y_0 + \delta_0) \\
      &= k( a - 2\delta_0 - \epsilon y_0 - \epsilon \delta_0 ) > kg \ge g.
   \end{align*}
\end{proof}

\begin{defn}
   Let $\pi, \pi'$ be a pair of paths with the same starting and ending point such that $\pi \ne \pi'$ (as edge sets). 
   We say that $\pi'$ is a \emph{$\epsilon$-detour for $\pi$} if
   \[
      |\pi' \setminus \pi| \le (1 + \epsilon) |\pi \setminus \pi'|.
   \]
   Here, $\pi'$ and $\pi$ are identified with the sets of edges they contain, $\setminus$ denotes set difference, and $| \cdot |$ is the cardinality of a set.
\end{defn}
Note that the condition that $\pi' \ne \pi$ implies that $|\pi \setminus \pi'| \ge 1$. For otherwise we would have
that $|\pi' \setminus \pi| \le (1 + \epsilon) |\pi \setminus \pi'| = (1+ \epsilon) \cdot 0 = 0$, that is,
$|\pi' \setminus \pi| = |\pi \setminus \pi'| = 0$ and hence $\pi' = \pi$.
Intuitively, one should think of $\pi'$ as an ``alternate path'' which misses a fair number of edges of the original path
but is not much longer than the original path. A simple but useful observation is:
\begin{prop}
   $\pi'$ is an $\epsilon$-detour for $\pi$ if and only if $\pi$ and $\pi'$ have the same endpoints and
   \[
      |\pi'| - |\pi| \le \epsilon |\pi \setminus \pi'|.
   \]
\end{prop}
\begin{proof}
   This follows immediately from the facts that $|\pi \setminus \pi'| = |\pi| - |\pi \cap \pi'|$ and $|\pi' \setminus \pi| = |\pi'| - |\pi \cap \pi'|$, together
   with some algebraic manipulation.
\end{proof}

\begin{defn}
   Let $\epsilon$ and $I_0$ all be as in Lemma \ref{lem:technicallemma},
   and let $C$ be a constant.
   We call $(\alpha, \gamma)$ a \emph{feasible pair} with respect to the $T$-geodesic\footnote{
   A $T$-geodesic from $x$ to $y$ is a path $\pi:x \to y$ with $T(\pi) = T(x,y)$.
   In general there may be more than one $T$-geodesic; we implicitly fix an arbitrary well-ordering
   on self-avoiding paths in $G$ and define ``the'' $T$-geodesic $\pi$ from $x$ to $y$ to be the $T$-geodesic
   which is least in this ordering.
   On the other hand, it is not a priori obvious that a $T$-geodesic exists. If $\nu(\{0\}) < p_c(G)$ and $G$ is locally finite, then it is
   easily shown (see Proposition 4.4 in \cite{ADH}) that all pairs of points $x,y \in V$ admit a $T$-geodesic;
   in particular, if $\nu$ is exponential-subcritical, $T$-geodesics exist.
   In Section \ref{sec:qitrees} we do not assume that $\nu$ is exponential-subcritical, but the arguments are easily modified
   to avoid the assumption that $T$-geodesics exist, by considering paths $\pi:x \to y$ with $T(\pi) \le T(x,y) + \epsilon$
   and letting $\epsilon \to 0$.}
   $\pi$ from $x$ to $y$ if both $\alpha$ and $\gamma$ are self-avoiding,
   $\gamma$ is an $\epsilon$-detour for $\alpha$ of length at most $C(1+\epsilon)$, 
   $\alpha$ is a subpath of $\pi$, and for all $e \in (\alpha \cup \gamma) \setminus (\alpha \cap \gamma)$, $w(e) \in I_0$.
\end{defn}
This notion of course depends on $C$, $\epsilon$, and $I_0$ even though this is suppressed in the notation.
Here $C$ is an unspecified constant, but in practice there will be one particular $C=C(\epsilon)$ that we end up using.
These detours turn out to be key to proving the strict inequalities we want to show, as we shall see in the next lemma.

The following lemma is essentially contained within Lemma 5.19 and the proof of Theorem 2.9(b)
from Proposition 5.22 in \cite{vdBK}, but we write it here to be 
explicit about what the necessary modifications are.
\begin{lemma} \label{lem:feasibletogap}
   Let $\nu, \tilde{\nu}$ have finite mean and be such that
   $\tilde{\nu}$ is strictly more variable than $\nu$ and \eqref{eq:extraassumption} holds. 
   Let $\epsilon$ and $I_0$ be given as in Lemma \ref{lem:technicallemma} and let $C$
   be fixed. Then there exists some constant $c_0 > 0$ such that if $G=(V,E)$ is a graph and $\{B_i\}_{i \in I} \subset E$
   is a family of disjoint subsets, for any $x,y \in V$ we have
   \[
      \E T(x,y) - \E \tilde{T}(x,y) \ge c_0 \sum_{i \in I} \Prob( B_i \mbox{ contains a feasible pair for the } T\mbox{-geodesic } \pi: x \to y ).
   \]
\end{lemma}
\begin{proof}
   As in \cite{vdBK}, let $\hat{w}:E \to [0,\infty)$ be given by
   \[
      \hat{w}(e) := \ind_{\xi(e) = 0} w(e) + \ind_{\xi(e) = 1} \tilde{w}(e),
   \]
   where $\{\xi(e)\}_{e \in E}$ is a family of i.i.d. $\mathrm{Unif}(\{0,1\})$ variables, also independent of $w,\tilde{w}$.
   As shown in Lemma 5.19 of \cite{vdBK}, $\E \tilde{T}(x,y) \le \E \hat{T}(x,y)$ for all $x,y \in V$,
   so it suffices to show the desired inequality with $\tilde{T}$ replaced by $\hat{T}$.
   
   We will call a pair $(\alpha,\gamma)$ \emph{advantageous} for the $T$-geodesic $\pi: x \to y$ if it is feasible for $\pi$
   and furthermore $\xi(e) = 0$ for all $e \in \gamma \setminus \alpha$, $\xi(e) = 1$ for all $e$ in a subset $S \subset \alpha \setminus \gamma$
   of size at least $|S| \ge \frac{1}{1+\epsilon} |\gamma \setminus \alpha|$, and if for all $e \in S$ we have $\tilde{w}(e) > w(e) + a$,
   where $a>0$ is as given in Lemma \ref{lem:technicallemma}.
   Note that, by Lemma \ref{lem:technicallemma}, if $(\alpha,\gamma)$ is advantageous then
   \[
      \hat{T}(\alpha) - \hat{T}(\gamma) = \hat{T}(\alpha \setminus \gamma) - \hat{T}(\gamma \setminus \alpha) 
      \ge \tilde{T}(S) - T(\gamma \setminus \alpha) \ge g,
   \]
   where $g > 0$ is as in the lemma.
   Furthermore, for any pair $(\alpha, \gamma)$ we have
   \[
      \E \left[ \ind_{\{(\alpha,\gamma) \mbox{ is advantageous}\}} \middle| w \right] \ge 
      2^{-(C(1+\epsilon) + C)} b^C \ind_{\{(\alpha,\gamma) \mbox{ is feasible}\}}.
   \]
   (Here we have used that $|\gamma| \le C(1+\epsilon)$).
   
   Therefore, consider a $T$-geodesic $\pi$ from $x$ to $y$. Construct another (random) path $\pi'$ by starting with $\pi$
   and, for each $B_i$, if $B_i$ contains an advantageous pair $(\alpha_i, \gamma_i)$ for $\pi$, replacing the subsegment $\alpha_i$
   with $\gamma_i$. (If $B_i$ contains more than one advantageous pair, choose the least one in some arbitrary ordering).
   We then have
   \begin{align*}
      \hat{T}(\pi) - \hat{T}(\pi') \ge g \sum_i \ind_{\{ B_i \mbox{ contains an advantageous pair for } \pi \}}. 
   \end{align*}
   Since, as shown in Lemma 5.19 of \cite{vdBK}, $\E T(\pi) \ge \E \hat{T}(\pi)$, we have
   \[
      \E T(x,y) - \E \hat{T}(x,y) \ge \E \hat{T}(\pi) - \E \hat{T}(\pi') \ge g \sum_{i \in I} \Prob(B_i \mbox{ contains an advantageous pair for } \pi).
   \]
   But (again using some fixed ordering on pairs inside $B_i$) we have
   \begin{align*}
      &\Prob(B_i \mbox{ contains an advantageous pair for } \pi) \\
      \ge &\sum_{(\alpha,\gamma) \subset B_i} \E \left[ \ind_{\{(\alpha,\gamma) \mbox{ is the least pair in } B_i \mbox{ which is feasible}\}}
                                                                                  \ind_{\{(\alpha,\gamma) \mbox{ is advantageous} \}} \right] \\
      = &\sum_{(\alpha,\gamma) \subset B_i} \E \left[ \ind_{\{(\alpha,\gamma) \mbox{ is the least pair in } B_i \mbox{ which is feasible}\}}
                                                                                  \E[\ind_{\{(\alpha,\gamma) \mbox{ is advantageous} \}} | w] \right] \\
      \ge &2^{-(C(1+\epsilon) + C)} b^C 
            \sum_{(\alpha,\gamma) \subset B_i} \E \left[ \ind_{\{(\alpha,\gamma) \mbox{ is the least pair in } B_i \mbox{ which is feasible}\}} \right] \\
      = & 2^{-(C(1+\epsilon) + C)} b^C \Prob(B_i \mbox{ contains a feasible pair for } \pi).
   \end{align*}
   Thus we have the lemma with $c_0 := g \cdot 2^{-(C(1+\epsilon) + C)} b^C > 0$.
\end{proof}
Inequalities up to a constant factor will appear many times in this paper, so from here we fix the following notation.
For two functions $f$ and $g$ of a parameter $t$, we will write $f(t) \lessim g(t)$ or $g(t) \gessim f(t)$ if there is
some constant $c>0$ and $t_0 < \infty$ such that $f(t) \le c g(t)$ for all $t \ge t_0$.
In this paper our parameter $t$ is typically either $d(x,y)$ or $R$, and which it is should be clear from context.

Finally, it will be convenient to ``upgrade'' to the following lemma (where the same hypotheses on $\nu$ and $\tilde{\nu}$ are assumed):
\begin{lemma} \label{lem:feasibletovdBK}
Let $\{B_i\}_{i \in I}$ be a family of subgraphs of $G$ and suppose that 
\[
   \sup_{i \in I} \# \{ j \in I : B_j \cap B_i \ne \emptyset \} < \infty.
\]
Then, if
\[
   \sum_{i \in I} \Prob( B_i \mbox{ contains a feasible pair for the } T\mbox{ geodesic } \pi : x \to y) \gessim d(x,y)
\]
for all $x,y \in V$ with $d(x,y)$ sufficiently large, then
\[
   \liminf_{d(x,y) \to \infty} \frac{ \E T(x,y) - \E \tilde{T}(x,y) }{ d(x,y) } > 0.
\]
\end{lemma}
\begin{proof}
   First, consider the graph whose vertex set is $I$ and whose edges are $\{i, j\}$ such that $B_i \cap B_j \ne \emptyset$.
   Our first assumption states precisely that this graph has degree bounded by some constant, let's call it $D' < \infty$.
   Then this graph can be colored by $D'+1$ colors using a greedy coloring.
   Hence we get a decomposition $I = \bigsqcup_{\ell=1}^{D' + 1} I_{\ell}$ such that for each fixed $\ell$, for all $i,j \in I_{\ell}$,
   if $i \ne j$ then $B_i \cap B_j = \emptyset$. Moreover, we have
   \begin{align*}
       & \max_{\ell \in \{1,...,D'+1\}} \sum_{i \in I_{\ell}} \Prob( B_i \mbox{ contains a feasible pair for the geodesic } \pi : x \to y) \\
       \ge &\frac{1}{D' + 1} \sum_{i \in I} \Prob( B_i \mbox{ contains a feasible pair for the geodesic } \pi : x \to y) 
       \gessim d(x,y).
   \end{align*}
   Thus we will have our lemma once we show that for each $\ell$
   \[
      \E T(x,y) - \E \tilde{T}(x,y) \gessim \sum_{i \in I_{\ell}} \Prob(B_i \mbox{ contains a feasible pair for the geodesic } \pi : x \to y).
   \]
   But since the $\{B_i\}_{i \in I_{\ell}}$ are disjoint families, this follows immediately from Lemma \ref{lem:feasibletogap}.
\end{proof}

In light of the previous lemma, our strategy for proving our main theorems will be to find suitable subgraphs $B_i$ of $G$
and then prove that the expected number of $B_i$ containing a feasible pair for the $T$-geodesic from 
$x$ to $y$ is at least a constant times $d(x,y)$.

\section{Graphs that admit detours} \label{sec:detours}
Here we introduce and prove facts about the key fine-geometric condition on our graphs.
\begin{defn}
   We say that a graph $G$ \emph{admits detours} if for every $\epsilon > 0$, there exists some $C$ such that
   for every self-avoiding path $\pi$ in $G$ of length $C$, there exists a self-avoiding $\epsilon$-detour $\pi'$ for $\pi$.
\end{defn}
Our main theorems say that at least in certain coarse-geometric settings, this fine-geometric condition on the graph is equivalent to the
vdBK property.
We first give an equivalent condition and show that this condition is \emph{necessary} for a graph to have the vdBK property.
We then give examples of graphs which admit detours.

\subsection{vdBK graphs admit detours}
Recall that a path $\pi$ from $v$ to $w$ is called an \emph{edge-geodesic} if for all paths $\pi'$ from
$v$ to $w$, $|\pi| \le |\pi'|$. $\pi$ is called a \emph{unique (edge)-geodesic} if for all $\pi' \ne \pi$ from $v$ to $w$, $|\pi| < |\pi'|$.

\begin{prop} \label{prop:detourequiv}
   $G$ admits detours if and only if $G$ \emph{admits detours along unique geodesics} in the following sense: for all
   $\epsilon > 0$, there exists $C < \infty$ such that for every unique edge-geodesic $\pi$ of length $C$, there exists 
   an $\epsilon$-detour $\pi'$ for $\pi$.
\end{prop}
%
\begin{proof}
   The forward implication is clear. Now assume that $G$ admits detours along unique geodesics.
   Note that the $\epsilon$-detour $\pi'$ for a unique geodesic $\pi$ can be made self-avoiding simply by loop erasing;
   the resulting path is still an $\epsilon$-detour for $\pi$ because the process of loop erasing cannot
   increase $|\pi' \setminus \pi|$ and cannot decrease $|\pi \setminus \pi'| \ge 1$.
   So it only remains to construct self-avoiding $\epsilon$-detours for self-avoiding paths which are not unique geodesics.
   Let $\pi$ be a self-avoiding path which is not a unique geodesic, and let $\pi'$ be an edge-geodesic connecting the endpoints of $\pi$
   which is not equal to $\pi$.
   Since $\pi'$ is a geodesic, we have
   \[
      0 \le |\pi| - |\pi'| = |\pi \setminus \pi'| - |\pi' \setminus \pi|,
   \]
   so that
   \[
      |\pi' \setminus \pi| \le |\pi \setminus \pi'| \le (1+ \epsilon)|\pi \setminus \pi'|
   \]
   for all $\epsilon > 0$.
\end{proof}
We can now easily prove that admitting detours is a \emph{necessary} condition for a graph to be vdBK.

\begin{thm} \label{thm:notvdBK}
   Let $G$ be a graph which does not admit detours.
   Then there exists a sequence of pairs $(x_n,y_n) \in V^2$ with $d(x_n,y_n) \tendsto{n}{\infty} \infty$
   and a pair of atomless measures $\nu, \tilde{\nu}$ with finite mean which are supported away from $0$ 
   and such that $\tilde{\nu}$ is strictly more variable than $\nu$ but
   \[
      \E T(x_n, y_n) = \E \tilde{T}(x_n, y_n)
   \]
   for all $n$.
   In particular, if $G$ has bounded degree, then $G$ does not satisfy the vdBK property.
\end{thm}
\begin{proof}
   Since $G$ does not admit detours, by Proposition \ref{prop:detourequiv} there exists $\epsilon_0 > 0$ such that
   for each $n$ we have a unique geodesic $\pi_n$ of length $n$ which
   does not admit a $\epsilon_0$-detour, which is to say that, if $x_n$ and $y_n$ are the endpoints of $\pi_n$, then any other 
   self-avoiding $\pi'_n$ from
   $x_n$ to $y_n$ satisfies
   \[
      |\pi'_n \setminus \pi_n| \ge (1 + \epsilon_0)|\pi_n \setminus \pi'_n|.
   \]
   Note that, canceling a term of $T(\pi_n \cap \pi'_n)$, we always have
   \[
      T(\pi'_n) - T(\pi_n) = T(\pi'_n \setminus \pi_n) - T(\pi_n \setminus \pi'_n).
   \]
   Now assume $\nu$ is supported on $[1,1+\epsilon_0]$. Then for any $\pi'_n \ne \pi_n$ we have
   \begin{align*}
      T(\pi'_n \setminus \pi_n) - T(\pi_n \setminus \pi'_n) &\ge 1 \cdot |\pi'_n \setminus \pi_n| - (1+ \epsilon_0)|\pi_n \setminus \pi'_n| \\
      & \ge(1 + \epsilon_0) |\pi_n \setminus \pi'_n| - (1 + \epsilon_0)|\pi_n \setminus \pi'_n| = 0.
   \end{align*}
   That is, when $\nu$ is supported on $[1, 1+ \epsilon_0]$, $\pi_n$ is almost surely has optimal passage time, that is,
   \[
      T(x_n,y_n) = T(\pi_n) \mbox{ a.s.}
   \]
   But then
   \[
      \E T(x_n,y_n) = \E T(\pi_n) = (\E w) d(x_n,y_n).
   \]
   In particular, if both $\nu$ and $\tilde{\nu}$ are supported on $[1,1+\epsilon_0]$ and $\E w = \E \tilde{w}$, we get
   \[
      \E T(x_n,y_n) = (\E w) d(x_n,y_n) = (\E \tilde{w}) d(x_n,y_n) = \E \tilde{T}(x_n,y_n),
   \]
   so to complete our proof we just need to find two such $\nu, \tilde{\nu}$ such that $\tilde{\nu}$ is strictly more variable than $\nu$.
   For example, we can take $\tilde{\nu}$ to be the uniform measure on $[1, 1+\epsilon_0]$ and $\nu$ to be the uniform measure
   on $[1+(\epsilon_0/4), 1+(3\epsilon_0/4)]$ (see Example 2.17 in \cite{vdBK}).
   Finally, if $G$ has bounded degree, then $\nu(\{1+(\epsilon_0/4)\}) = 0 < 1/D \le \vec{\underline{p_c}}$, so $\nu$ is exponential-subcritical
   and the pair $\nu, \tilde{\nu}$ contradicts the vdBK property.
\end{proof}

\subsection{Examples of graphs which admit detours}
Proposition \ref{prop:detourequiv} gives us an easy way to produce graphs which do not admit detours, namely by ``doubling'' edges.
Simply take any graph $G$ and create a new graph $G'$ by taking the edge set of $G$ and adding an extra edge between each $v, w \in V$ which are
connected by an edge in $G$. Since every edge has a ``parallel'' edge, $G'$ has no unique geodesics, and hence
by Proposition \ref{prop:detourequiv} admits detours.

This is a rather ``cheap'' way to get a graph that admits detours, especially since in first-passage percolation often the graphs one is interested in
are simple, i.e. contain no parallel edges. However, this is a simple way to see that the property of admitting detours is not a quasi-isometry invariant;
every graph $G$ is quasi-isometric to a graph $G'$ which admits detours, so admitting detours is a ``fine'' rather than a ``coarse'' geometric property.

The property is not group-theoretic either; that is, for some groups $\Gamma$, some Cayley graphs of $\Gamma$ admit detours and others do not.
This can be seen using the same technique as above if one allows Cayley graphs to have double edges (for discussion on what exactly is meant
by ``Cayley graph'' see Appendix \ref{app:grouptheory}). But even if one restricts to simple Cayley graphs, there are counterexamples.
The standard Cayley graph of $\Z$ does not admit detours (since it is a tree), but every Cayley graph of $\Z$ not isomorphic to this one does.
Similarly, Cayley graphs of $\Z/2 * \Z/2$ which are isomorphic to the standard Cayley graph of $\Z$ do not admit detours, but all others do. 
This is proven in Appendix \ref{app:grouptheory}.

On the other hand, there are several properties of groups which ensure that \emph{all} of their Cayley graphs admit detours. For instance, we have
\begin{restatable*}{prop}{finitenormalsubgroup} \label{prop:finitenormalsubgroup}
   Let $\Gamma$ be a finitely generated group, and suppose that $\Gamma$ contains $F \unlhd \Gamma$ a nontrivial finite normal subgroup.
   Then any Cayley graph of $\Gamma$ admits detours.
\end{restatable*}

\begin{restatable*}{prop}{withcenter} \label{prop:withcenter}
   Let $\Gamma$ be a finitely generated group not isomorphic to $\Z$ or $\Z/2 * \Z/2 \cong \Z \rtimes \Z/2$ 
   with a finite index subgroup $H$ such that $H$ has nontrivial center. Then any Cayley graph $G$ 
   of $\Gamma$ admits detours.
\end{restatable*}
These allow us to conclude:

\begin{restatable*}{thm}{nilpotentdetours} \label{thm:nilpotentdetours}
   Let $G$ be a Cayley graph of a virtually nilpotent group. If $G$ is not isomorphic as a graph
   to the standard Cayley graph of $\Z$, then $G$ admits detours.
\end{restatable*}

The proofs of all of these facts are entirely combinatorial and group-theoretic, and are deferred to Appendix \ref{app:grouptheory}.
There may be many weaker group-theoretic conditions which ensure that every Cayley graph of a group admits detours;
the ones proven here were mostly chosen in order to prove Theorem \ref{thm:nilpotentdetours}, since this is needed to 
prove Theorem \ref{thm:nilpotentvdBK}.
Of course, they readily apply to many groups which are not virtually nilpotent.

\section{Proof of Theorem \ref{thm:qitree}} \label{sec:qitrees}
In this section, we prove the following:
\qitree*

A metric space which is quasi-isometric to a tree (where the tree is given the usual graph metric) is called a \emph{quasi-tree}.
The following is a well-known equivalent condition for a geodesic metric space to be a quasi-tree
(the original, slightly weaker condition is due to Manning \cite{Manning}; the following extension is a well-known
consequence, see e.g. \cite{BBF}):
\begin{thm}[Manning's bottleneck criterion]
   A geodesic metric space $X$ is a quasi-tree if and only if there exists some $\Delta < \infty$ such that
   for every $x,y \in X$, for every geodesic $[x,y]$ from $x$ to $y$, for every $z \in [x,y]$, any path $\pi$ from $x$ to $y$
   intersects $B(z,\Delta)$.
\end{thm}

\begin{cor}
   Let $G = (V,E)$ be a graph which is a quasi-tree. Then there exists $R < \infty$ such that for any $x,y \in V$, for any edge geodesic $[x,y]$ from $x$ to $y$
   and any $z \in V([x,y])$, every path $\pi$ from $x$ to $y$ intersects $E(B(z,R))$.
\end{cor}
Here (and later in the paper), if $S \subset V$, then $E(S) \subset E$ is defined to be the set of edges of $G$ which have both endpoints lying in $S$,
and if $S \subset E$, then $V(S) \subset V$ is defined to be the set of vertices which are an endpoint of an edge in $S$.
\begin{proof}
   $(V,d)$ is naturally a subspace of the geodesic metric space $(G,d)$ given by the geometric realization of $G$ (i.e. the 1-dimensional metric cell complex where each $e \in E$
   corresponds to 1-cell in $G$ isometric to $[0,1]$, joining 0-cells corresponding to the endpoints of $e$). The combinatorial edge-geodesics we study in this paper
   correspond to geodesics in $(G,d)$, and one quickly sees that the corollary holds with $R = \Delta + 1$.
\end{proof}


\begin{proof}[Proof of Theorem \ref{thm:qitree}]
   We only need to prove that, if $G$ admits detours, then we have $\E \tilde{T} \ll \E T$ whenever $\tilde{\nu}$ is strictly
   more variable than $\nu$, since the other direction is given by
   Theorem \ref{thm:notvdBK}.
   
   To this end, let $\nu, \tilde{\nu}$ have finite mean with $\tilde{\nu}$ strictly more variable than $\nu$ and first assume that
   \eqref{eq:extraassumption} holds. 
   Then let $\epsilon>0, I_0$ be given
   as in Lemma \ref{lem:technicallemma}.
   Since $G$ admits detours, there is some $C$ such that every self-avoiding path $\pi$ of length $C$ admits a self-avoiding $\epsilon$-detour
   (which is necessarily of length at most $(1+\epsilon)C$).
   Since $G$ is a quasi-tree, take $R < \infty$ such that for all $x,y \in V$, for any geodesic $[x,y]$ from $x$ to $y$,
   any path $\pi: x \to y$ intersects $E(B(z,R))$ for all $z \in V([x,y])$.
   
   Now, define the family $\{ B_v := B(v,R+C(2+\epsilon)) : v \in V \}$.
   First we claim that for any $v$,
   \begin{align*}
      &\Prob( B_v \mbox{ contains a feasible pair for the geodesic } \pi:x \to y ) \\
      \ge &\Prob( \pi \mbox{ visits } B(v,R) \mbox{ and leaves } B_v, w(e) \in I_0 \mbox{ for all } e \in E(B_v)).
   \end{align*}
   To see this, note that if $\pi$ visits $B(v,R)$ and exits $B_v$, there is a segment $\alpha$ of $\pi$ of length at least $C$
   contained in $B(v,R+C)$; this segment admits a self-avoiding $\epsilon$-detour $\gamma$ contained in $B(v,R+C(2+\epsilon))$.
   Then, if also $w(e) \in I_0$ for all $e \in E(B_v)$, $(\alpha,\gamma)$ forms a feasible pair.
   
   Next note that if $v \in V([x,y]) \setminus B_y$ then \emph{any} path from $x$ to $y$ visits $B(v,R)$ and exits $B_v$,
   and so for such $v$ we have
   \begin{align*}
      &\Prob( B_v \mbox{ contains a feasible pair for the geodesic } \pi:x \to y ) \\
      \ge &\Prob( w(e) \in I_0 \mbox{ for all } e \in E(B_v)) = \nu(I_0)^{|E(B_v)|} \ge (\nu(I_0))^{(D+1)^{R+C(2+\epsilon)+1}} =: c
   \end{align*}
   where $D$ is the maximum degree of $G$.
   Therefore for $d(x,y)$ sufficiently large we have
   \begin{align*}
      &\sum_{v \in V} \Prob(B_v \mbox{ contains a feasible pair for the geodesic } \pi:x \to y ) \\ 
      \ge &\sum_{v \in V([x,y]) \setminus B_y} \Prob( w(e) \in I_0 \mbox{ for all } e \in E(B_v)) \\
      \ge &|V([x,y]) \setminus B_y| c \ge c d(x,y) - c(D+1)^{R+C(2+\epsilon)} \gessim d(x,y).
   \end{align*}
   Moreover, we have that
   \[
      \sup_v \# \{ w : B_w \cap B_v \ne \emptyset \} \le \sup_v |B(v,2(R + C(2 + \epsilon)))| \le (D+1)^{2(R + C(2 + \epsilon))} < \infty,
   \]
   and so by Lemma \ref{lem:feasibletovdBK} we have that
   \[
      \liminf_{d(x,y) \to \infty} \frac{ \E T(x,y) - \E \tilde{T}(x,y) }{d(x,y)} > 0,
   \]
   as desired.
   
   On the other hand, if $w$ and $\tilde{w}$ do not satisfy \eqref{eq:extraassumption}, then take $\bar{w}$ as in Lemma \ref{lem:wlog};
   applying our above argument to $\bar{w}$ gives
   \[
            \liminf_{d(x,y) \to \infty} \frac{ \E T(x,y) - \E \tilde{T}(x,y) }{d(x,y)} \ge  
            \liminf_{d(x,y) \to \infty} \frac{ \E T(x,y) - \E \bar{T}(x,y) }{d(x,y)}> 0,
   \]
   and so we are done.
\end{proof}

As a corollary we also obtain Theorem \ref{thm:virtfree}:
\begin{proof}[Proof of Theorem \ref{thm:virtfree}]
   Let $\Gamma$ be a virtually free group. Since free groups have Cayley graphs which are regular trees, any
   Cayley graph of $\Gamma$ is quasi-isometric to a regular tree, and so by Theorem \ref{thm:qitree} a Cayley
   graph of $\Gamma$ is vdBK if and only if it admits detours.
   If $\Gamma$ has a finite index subgroup with nontrivial center and is not isomorphic to $\Z$ or $\Z/2 * \Z/2$, 
   then by Proposition \ref{prop:withcenter},
   all its Cayley graphs admit detours.
   If $\Gamma$ has a finite normal subgroup, then by Proposition \ref{prop:finitenormalsubgroup},
   all its Cayley graphs admit detours; hence under either condition all Cayley graphs of $\Gamma$ are vdBK.
\end{proof}

\section{Proof of Theorem \ref{thm:polygrowthvdBK}} \label{sec:polygrowth}
We say that a graph has \emph{strict polynomial growth} if there exists some $0<d<\infty$ and some $0<c_1 \le C_1<\infty$ such that
for all $R \ge 1$,
\[
   c_1R^d \le \inf_{v \in V} |B(v,R)| \le \sup_{v \in V} |B(v,R) \le C_1R^d.
\]
(Note that this in particular entails that $G$ has bounded degree).
It is well known that Cayley graphs of finitely generated virtually nilpotent groups are of strict polynomial growth;
in fact, $|B(R)|/R^d$ converges to a constant as $R \to \infty$ \cite{Pansu}.
Moreover, it is clear that half-planes, sectors, and many other subgraphs of the standard Cayley graph of $\Z^d$
have strict polynomial growth.

The goal of this section is to prove the following:
\polygrowthvdBK*

\noindent
Given Theorem \ref{thm:polygrowthvdBK}, we can quickly prove Theorem \ref{thm:nilpotentvdBK} as follows:
\begin{proof}[Proof of Theorem \ref{thm:nilpotentvdBK} given Theorem \ref{thm:polygrowthvdBK}] 
   Let $\Gamma$ be a virtually nilpotent group. Then any Cayley graph for $\Gamma$ has strict polynomial growth \cite{Pansu}.
   If a Cayley graph $G$ of $\Gamma$ is not isomorphic as a graph to the 
   standard Cayley graph of $\Z$, then $G$ admits detours, by Theorem \ref{thm:nilpotentdetours}.
   So Theorem \ref{thm:polygrowthvdBK} implies that $G$ is vdBK.
\end{proof}
 
\subsection{A Peierls argument for graphs of strict polynomial growth}
In the previous case, where $G$ was quasi-isometric to a tree, we benefitted from the fact that we could find
areas which the geodesic visited with probability 1, and hence for such an area $A$ and any event $C$,
$\Prob(\{\pi \mbox{ visits } A \} \cap C) = \Prob(C)$.
Here we have to deal with graphs which may have many paths with mostly disjoint support between each pair of points,
and so a generic event $C$ will not have $\Prob(\{\pi \mbox{ visits } A \} \cap C) = \Prob(C)$.
In fact, there are very few events $C$ for which we can get nontrivial inequalities for $\Prob(\{\pi \mbox{ visits } A \} \cap C)$.
Therefore, a key tool will be the following lemma, which will ensure that, for certain types of local events,
the geodesic will in expectation visit $\gessim d(x,y)$ many regions where the prescribed events hold.
This is a Peierls-type argument, a special case of which was used in \cite{vdBK}.

\begin{lemma} \label{lem:peierls}
   Let $G = (V,E)$ be a graph. Suppose that for each sufficiently large $R < \infty$ we have the following:
      \begin{itemize}
         \item A partition $V = \bigsqcup_i B_i^R$, with $\sup_i \diam (B_i^R) < \infty$.
         \item A collection of subsets $\tilde{B}_i^R \subset E$ indexed by the same index set as the $B_i^R$.
         \item A collection of events $A_i^R$, each of which depends only upon the weights of the edges in $\tilde{B}_i^R$.
      \end{itemize}
   For each $R$ then construct the simple graph $G^R$ whose vertex set is $\{B_i^R\}_i$ and is such that two distinct vertices
   $B_i^R$ and $B_j^R$ are joined by an edge if and only if there is an edge in the original graph with one endpoint lying
   in $B_i^R$ and the other lying in $B_j^R$.
   Also construct the simple graph $\tilde{G}^R$ with the same vertex set which is such that two distinct vertices
   $B_i^R$ and $B_j^R$ are joined by an edge if and only if $\tilde{B}_i^R \cap \tilde{B}_j^R \ne \emptyset$.
   Suppose that there exists $D < \infty$ such that for all $R$, the degree of both $G^R$ and $\tilde{G}^R$
   is bounded by $D$. Suppose also that
   \[
      \rho(R) := \sup_i \Prob \left( (A_i^R)^c \right) \tendsto{R}{\infty} 0.
   \]
   Then, for all sufficiently large $R$, there exist $c_2(R), \epsilon_2(R) > 0$ such that for all sufficiently large $d(x,y)$,
   \[
      \Prob \left( 
      \begin{array}{c} \exists \gamma:x \to y \mbox{ visiting at most } c_2 d(x,y) \mbox{ distinct } \\
                              B_i^R \mbox{ such that } A_i^R \mbox{ holds }  \end{array}
                \right)
                \le e^{-\epsilon_2 d(x,y)}
   \]
\end{lemma}
\begin{rmk} \label{rmk:strongerpeierls}
   From the proof it will be clear that the following weaker condition is sufficient: let $D(R)$ be the maximum degree of $G^R$
   and let $\tilde{D}(R)$ be the maximum degree of $\tilde{G}^R$. Then for each $R$ such that
   \[
      (2D(R) \rho(R)^{\frac{1}{\tilde{D}(R) + 1}}) < 1,
   \]
   we have the desired bound.
   For instance, since in our later geometric constructions we can ensure $\rho(R) \le e^{-cR}$,
   this allows us to extend Theorem \ref{thm:polygrowthvdBK} from graphs
   of strict polynomial growth to graphs with $R^{d'} \lessim |B(R)| \lessim R^d$ with $d-d'<1$, since
   in that case our Voronoi construction below will give $\tilde{D}(R), D(R) \lessim R^{d-d'} = o(R)$, but it
   is difficult to come up with an example of a graph which has such a property but is not already of strict polynomial growth.
\end{rmk}
\begin{proof}
   Let $\gamma$ be a path from $x$ to $y$. This induces a path $\tilde{\gamma}$ in $G^R$ in a natural way: 
   $\tilde{\gamma}$ starts at the unique $B_{i_1}^R$ containing $x$, and each time $\gamma$ crosses an edge
   from a vertex in some $B_i^R$ to a vertex in some distinct $B_{i'}^R$, $\tilde{\gamma}$ crosses an 
   edge from $B_i^R$ to $B_{i'}^R$
   Note that since the diameter of the $B_i^R$ is bounded uniformly in $i$, there exists $\eta(R) > 0$ such
   that if $\gamma: x \to y$, then $\tilde{\gamma}$ visits at least $\eta d(x,y)$ distinct $B_i^R$.
   We want to bound the probability that (for some $c_2(R)$ to be chosen later) 
   some such $\tilde{\gamma}$ visits at most $c_2 d(x,y)$ $B_i^R$
   such that $A_i^R$ holds. First, note that if $\tilde{\gamma}$ visits at most $c_2 d(x,y)$ $B_i^R$ such
   that $A_i^R$ holds, a self-avoiding path obtained from $\tilde{\gamma}$ from erasing loops has the same property.
   So it suffices to bound the probability that some self-avoiding path $\tilde{\gamma}$ in $G^R$ which starts
   at $B_{i_1}^R \ni x$ and ends at $B_{i_2}^R \ni y$ visits at most $c_2 d(x,y)$ $B_i^R$ such that $A_i^R$ holds;
   to reduce clutter, let us write $B_i$ and $A_i$ instead of $B_i^R$ and $A_i^R$.
   
   Now, for a fixed self-avoiding path $\tilde{\gamma}$ visiting $k$ distinct $B_i$, we have
   \begin{align*}
      \Prob \left( \begin{array}{c} \tilde{\gamma} \mbox{ visits at most } c_2 d(x,y) \\
                B_i \mbox{ such that } A_i \mbox{ holds } \end{array} \right)
                \le \sum_{S \subset V(\tilde{\gamma}), |S| = k - c_2 d(x,y)} \Prob \left( \bigcap_{B_i \in S} A_i^c \right).
   \end{align*}
   Each such $S \subset \{ B_i \}_i$ contains a subset $S'$ which is independent in $\tilde{G}^R$ (that is,
   no two elements of $S'$ are joined by an edge of $\tilde{G}^R$) and which has size at least $|S'| \ge \frac{1}{D + 1}|S|$.
   From the definition of $\tilde{G}^R$ we see that if $S'$ is an independent set in $\tilde{G}^R$ then the collection
   of events $\{A_i^R\}_{B_i^R \in S'}$ is independent. Hence the above is bounded by
   \begin{align*}
      \sum_{\substack{S \subset V(\tilde{\gamma}), \\ |S| = k - c_2 d(x,y)}} \Prob \left( \bigcap_{B_i \in S'} A_i^c \right) 
      = \sum_{\substack{S \subset V(\tilde{\gamma}), \\ |S| = k - c_2 d(x,y)}} \prod_{B_i \in S'} \Prob(A_i^c) 
      \le {k \choose c_2 d(x,y)} \rho^{\frac{k - c_2 d(x,y)}{D + 1}}.
   \end{align*}
   On the other hand, the number of self-avoiding paths of length $k$ in $G^R$ starting at the unique $B_{i_1} \ni x$
   is at most $D^k$. Thus we have
   \begin{align*}
      \Prob \left( 
      \begin{array}{c} \exists \gamma:x \to y \mbox{ visiting at most } c_2 d(x,y) \mbox{ distinct } \\
                              B_i^R \mbox{ such that } A_i^R \mbox{ holds }  \end{array}
                \right)
                &\le \sum_{k = \lceil \eta d(x,y) \rceil}^{\infty} D^k {k \choose c_2 d(x,y)} \left(\rho^{\frac{1}{D + 1}}\right)^{k - c_2 d(x,y)} \\
      &\le \left( \rho^{\frac{1}{D + 1}} \right)^{-c_2 d(x,y)} \sum_{k=\lceil \eta d(x,y) \rceil}^{\infty} \left( 2D \rho^{\frac{1}{D+1}} \right)^k;
   \end{align*}
   for $R$ sufficiently large we have $2D \rho^{\frac{1}{D + 1}} < 1$, and then the right hand side above is equal to 
   \[
      \left( \rho^{\frac{1}{D + 1}} \right)^{-c_2 d(x,y)} \left( 2D \rho^{\frac{1}{D+1}} \right)^{\lceil \eta d(x,y) \rceil} \cdot \frac{1}{1 - 2D\rho^{\frac{1}{D+1}}}.
   \]
   If we choose $c_2 > 0$ sufficiently small that
   \[
      -c_2 \log (\rho^{1/(D+1)}) + \eta \log ( 2D \rho^{1/(D+1)} ) > 0
   \]
   then our upper bound decays exponentially in $d(x,y)$, and so we are done.
\end{proof}

Now, assuming that $G$ has strict polynomial growth, we construct a suitable family $\{B_i\}, \{\tilde{B}_i\}$.
First, for each $R$, choose a maximal subset $\{ o^R_i \}_i \subset V$ which is $R$-separated, that is, such that
if $i \ne j$, then $d(o_i,o_j) \ge R$. Also fix an arbitrary well-ordering on the indices $i$.
Maximality implies that for each vertex $v \in V$,
there exists some $i$ such that $d(o_i,v) \le R$. For each $i$, let $B_i$ be the ``Voronoi tile'' containing $o_i$,
that is, set 
\[
   B_i := \{ v \in V : d(o_i,v) < d(o_j,v) \mbox{ for all } j < i, d(o_i,v) \le d(o_j,v) \mbox{ for all } j \ge i \}.
\]
($B_i$ consists of the vertices which are closer to $o_i$ than any other $o_j$, but we ``break ties'' when $v$
is equidistant from $o_i$ and $o_j$ using the ordering on indices).
We see that $V = \bigsqcup_i B_i$ and that $\sup_i \diam B_i \le 2R$ (since each $B_i \subset B(o_i,R)$).
Next, we fix $0 < \Sigma < \infty$ (a scaling parameter that will be chosen to suit our separate constructions below).
We have the following:
\begin{prop} \label{prop:voronoiworks}
   Taking $B_i$ to be the Voronoi tiles and $\tilde{B}_i := E(B(o_i, \Sigma R))$ gives families satisfying the hypotheses of 
   Lemma \ref{lem:peierls}. That is, the associated sequence of graphs $G^R$, $\tilde{G}^R$ both have degree uniformly bounded in $R$.
\end{prop}
\begin{proof}
   Fixing some $o_i$, we have
   \begin{align*}
      \{ j : B_j \sim B_i \mbox{ in } G^R \} \subset \{ j : d(o_i,o_j) \le 2R + 1 \} \subset \{ j : B_j \subset B(o_i,3R+1) \},
   \end{align*}
   as well as
   \begin{align*}
      \{j : B_j \sim B_i \mbox{ in } \tilde{G}^R \} \subset \{ j : d(o_i,o_j) \le 2\Sigma R \} \subset \{ j : B_j \subset B(o_i, (2\Sigma + 1)R \}.
   \end{align*}
   Thus in order to bound both degrees it suffices to show that, given any constant $\Sigma'$, the quantity
   \[
      \# \{ j : B_j \subset B(o_i, \Sigma' R) \}
   \]
   is uniformly bounded in both $i$ and $R$. To this end, note that, since $o_i$ is $R$-separated, it follows that
   $B(o_i, (R/2) - 1) \subset B_i$. So using our volume bounds and the fact that the $B_j$ are disjoint we have
   \begin{align*}
      \# \{ j : B_j \subset B(o_i, \Sigma' R) \} c_1[(R/2)-1]^d \le \sum_{B_j \subset B(o_i, \Sigma' R)} |B_j|
      \le |B(o_i, \Sigma' R)| \le C_1(\Sigma' R)^d,
   \end{align*}
   so that
   \[
      \# \{ j : B_j \subset B(o_i, \Sigma' R) \} \le \frac{C_1(\Sigma' R)^d}{c_1[(R/2)-1]^d} \tendsto{R}{\infty} \frac{C_1}{c_1}(2 \Sigma')^d,
   \]
   so we are done.
\end{proof}
\begin{rmk}
   This is actually the only point in the proof where we use strict polynomial growth. In every other part of the proof,
   we will only use that we have a uniform strictly subexponential volume bound and bounded degree (which
   is equivalent to a uniform bound on $|B(v,1)|$).
   If one could find a suitable family for more general subexponential growth graphs, the methods in this paper would
   immediately show that such graphs are vdBK if and only if they admit detours.
   However, constructing such a family would take some ingenuity; if for instance we attempted to do the Voronoi
   construction for a graph with growth of order $e^{\sqrt{R}}$, the degree bounds we get from the above analysis
   are superpolynomial, and even using the stronger form of Lemma \ref{lem:peierls} (see Remark \ref{rmk:strongerpeierls})
   will require at least a strictly sublinear bound on the degree of $\tilde{G}^R$ to use our geometric constructions below,
   where the failure probabilities have order $-\log \rho(R) \sim R$.
\end{rmk}

Lastly, let us use these lemmata to prove the following, which will be very important to our later constructions.
\begin{lemma} \label{lem:bddawayfrominf}
   Let $G$ be a graph with strict polynomial growth, and suppose that $\nu$ is exponential-subcritical.
   Then, there exist $q, c>0$ such that, for all $x,y \in V$ with $d(x,y)$ sufficiently large, 
   \[
      \Prob( T(x,y) < (\inf + q)d(x,y) ) \le e^{-cd(x,y)}.
   \]
\end{lemma}
\begin{rmk}
   The conclusion of the above lemma also holds for \emph{any} graph $G$ of degree at most $D$ if one assumes $\nu(\{\inf\}) < 1/D$;
   this is proved in the course of proving Lemma A.1 in \cite{Tessera}.
\end{rmk}
\begin{proof}


   
   First, suppose $\inf = 0$; since $\nu$ is exponential-subcritical, $\nu(\{0\}) < \underline{p_c}$, and we can pick $q' > 0$ sufficiently small that
   if $\nu([\inf,\inf+q']) < \underline{p_c}$. Then, by the definition of $\underline{p_c}$,
   there is some $c'>0$ such that for any $R$ sufficiently large, for any $v \in V$,
   \[
      \Prob(v \mbox{ is connected to } B(v,R)^c \mbox{ by a path of edges which each have weight} < \inf + q')
      \le e^{-c'R}.
   \]
   In particular, for any $\Sigma \ge 2$, we have that 
   \begin{align*}
      \Prob(\exists p \in S(v, \Sigma R), x \in B(v,R), \mbox{ path } \alpha:p \to x \mbox{ in } B(v,\Sigma R) 
      \mbox{ s.t. } w(e) < \inf + q' \mbox{ for all } e \in \alpha) \\
      \le \Prob( \exists x \in B(v,R), p' \in S(x, (\Sigma - 1)R), \mbox{ path } \alpha:x \to p' \mbox{ s.t. }
      w(e) < \inf + q' \mbox{ for all } e \in \alpha) \\
      \le |B(v,R)| e^{-c'(\Sigma - 1) R} \le C_1R^d e^{-c'(\Sigma -1)R} \tendsto{R}{\infty} 0.
   \end{align*}
   In particular,
   \begin{align*}
      \inf_{v \in V} \Prob(\mbox{all paths from } S(v,\Sigma R) \mbox{ to } B(v,R) \mbox{ contain at least one edge of weight } \ge \inf + q')
      \tendsto{R}{\infty} 1,
   \end{align*}
   and so by Lemma \ref{lem:peierls}, for all sufficiently large $R$, there exist $c_2(R) >0, \epsilon_2(R) > 0$ such that
   for all sufficiently large $d(x,y)$,
   \[
      \Prob \left( 
      \begin{array}{l} \exists \gamma:x \to y \mbox{ visiting at most } c_2 d(x,y) \mbox{ distinct } B_i \mbox{ such that } \\
                           \mbox{ all paths from } S(o_i,\Sigma R) \mbox{ to } B(o_i,R) \mbox{ contain at least one} \\
                           \mbox{ edge of weight } \ge \inf + q' \end{array} \\
                \right)
                \le e^{-\epsilon_2 d(x,y)}.
   \]
   Now, each $B(o_i, \Sigma R)$ intersects at most $D'$ other $B(o_i', \Sigma R)$ by Proposition \ref{prop:voronoiworks}, 
   and so if a path $\gamma$ visits at least $c_2 d(x,y)$ $B_i$  such that all paths from $S(o_i,\Sigma R)$ to $B(o_i,R)$ contain
   at least edge with weight at least $\inf + q'$, there is some collection of at least $\frac{1}{D'+1} c_2 d(x,y)$ \emph{disjoint}
   $B(o_i, \Sigma R)$ with this property such that $\gamma$ visits $B_i$. If $x \notin B(o_i,\Sigma R)$ then in particular $\gamma$
   starts outside of $B(o_i, \Sigma R)$ and so since $\gamma$ visits $B_i \subset B(o_i,R)$, some subpath of $\gamma$
   joins $S(o_i,\Sigma R)$ to $B(o_i,R)$ and so some edge of $\gamma \cap E(B(o_i, \Sigma R))$ has weight at least $\inf + q'$.
   So by disjointness we conclude that $\gamma$ has at least $\frac{c_2}{D'+1} d(x,y) - 1$ edges of length at least $\inf + q'$,
   and so
   \[
      T(\gamma) \ge (\inf) d(x,y) + q' \left( \frac{c_2}{D'+1} d(x,y) - 1 \right)
   \]
   in this case. So taking $q := q' c_2/(2(D'+1))$, we see that whenever $d(x,y) \ge 2(D'+1)/c_2$ we have
   \[
      \Prob(T(\gamma) < (\inf + q)d(x,y)) \le e^{-\epsilon_2 d(x,y)},
   \]
   and the lemma follows.
   
   Now, suppose that $\inf > 0$. Then choose $q' > 0$ such that $\nu([\inf, \inf + q']) < \vec{\underline{p_c}}$ to obtain $c' > 0$ such that for any $R$ sufficiently large,
   for any $v \in V$ we have 
   \[
      \Prob \left(
                \begin{array}{c}
                v \mbox{ is connected to } B(v,R)^c \mbox{ by an edge-geodesic path} \\ 
                \mbox{of edges which each have weight} < \inf + q'
                \end{array}
                \right)
      \le e^{-c'R}.
   \]
   Then arguing similarly as above, by Lemma \ref{lem:peierls}, for all sufficiently large $R$, there exist $c_2(R) >0, \epsilon_2(R) > 0$ such that
   for all sufficiently large $d(x,y)$,
   \[
      \Prob \left( 
      \begin{array}{l} \exists \gamma:x \to y \mbox{ visiting at most } c_2 d(x,y) \mbox{ distinct } B_i \mbox{ such that } \\
                           \mbox{ all edge-geodesic paths from } S(o_i,\Sigma R) \mbox{ to } B(o_i,R) \mbox{ contain} \\ 
                           \mbox{ at least one edge of weight } \ge \inf + q' \end{array} \\
                \right)
                \le e^{-\epsilon_2 d(x,y)}.
   \]
   Similar to above, we then see that (except on an exponentially small event) every path $\gamma$ from $x$ to $y$ contains at least
   $\frac{c_2}{D'+1} d(x,y) - 1$ disjoint subpaths which are either not edge-geodesic, or contain an edge of weight at least $\inf + q'$.
   Each such subpath $\gamma_i$ has passage time $T(\gamma_i) \ge (\inf)|\gamma_i| + \min(\inf, q')$.
   So taking $q := \min(q', \inf) c_2/(2(D'+1)) > 0$ and $c = \epsilon_2$ gives the lemma.
   
\end{proof}
\begin{rmk}
   This is the only part of the proof where we use the exponential subcriticality of $\nu$.   
   \end{rmk}
\begin{rmk}
   This lemma implies in particular that if $G$ is a Cayley graph of a finitely generated virtually nilpotent group and if $\nu(\{0\}) < p_c$,
   then there exists $a>0$ such that for all $x,y \in V$, $\E T(x,y) \ge a d(x,y)$.
   This means, for instance, that the results of \cite{BenjaminiTessera} giving the existence of a scaling limit apply when
   $\nu$ has an exponential moment and $\nu(\{0\}) < p_c$ (a weaker condition than the condition $\nu(\{0\}) < 1/D$ quoted in that paper).
\end{rmk}

\subsection{Proof strategy: a resampling scheme}
Note that if we have any family of events $A_i^R$ as in Lemma \ref{lem:peierls}, we have that in particular
\begin{align*}
   &\sum_i \Prob( \{\mbox{the geodesic } \pi:x \to y \mbox{ visits } B^R_i\} \cap A^R_i ) \\
   = &\E[ \# B_i \mbox{ such that } \pi \mbox { visits } B_i \mbox{ and } A^R_i \mbox {holds} ] \\
   \ge &(c_2 d(x,y)) \Prob( \pi \mbox{ visits at least } c_2d(x,y) \mbox{ } B_i \mbox{ such that } A^R_i \mbox{ holds}) \\
   \ge &(c_2 d(x,y)) (1 - e^{-\epsilon_2 d(x,y)}) \gessim d(x,y).
\end{align*}
We will say that $\pi$ \emph{crosses} $B(o_i, \Sigma R)$ if $\pi$ starts at a vertex outside $B(o_i, \Sigma R)$, ends at a vertex outside
$B(o_i, \Sigma R)$, and visits $B_i$. Since the number of $o_i$ such that $x \in B(o_i, \Sigma R)$ or $y \in B(o_i, \Sigma R)$
is bounded independent of $x$ and $y$, we see also from the above that
\[
   \sum_i \Prob( \{\mbox{the geodesic } \pi:x \to y \mbox{ crosses } B(o_i, \Sigma R) \} \cap A^R_i ) \gessim d(x,y).
\]

Thus, if we find a family of events $\{A^R_i\}$ such that for each $i$, $\Prob( B(o_i, \Sigma R) \mbox{ contains a feasible pair})$ is at least 
a positive constant (independent of $x,y,i$, but possibly depending on $R$) times 
\\ $\Prob( \{ \pi \mbox{ crosses } B(o_i, \Sigma R) \} \cap A^R_i )$, we will have
\[
   \sum_i \Prob( B(o_i, \Sigma R) \mbox{ contains a feasible pair for } \pi: x \to y ) \gessim d(x,y),
\]
and hence by Lemma \ref{lem:feasibletovdBK} we will have our theorem. (Note that here the role of the $B_i$ from 
Lemma \ref{lem:feasibletovdBK} is played by $B(o_i, \Sigma R)$, not the Voronoi tiles $B_i$ we defined in the last section).

We will obtain a bound of the form 
\[
   \Prob(B(o_i, \Sigma R) \mbox{ contains a feasible pair}) \ge c(R) \Prob( \{ \pi \mbox{ crosses } B(o_i, \Sigma R) \} \cap A^R_i )
\]
by introducing a resampling scheme, as in \cite{vdBK}. Explicitly, fix some $o_i$;
throughout the rest of the paper, we abbreviate $B(s) := B(o_i,s)$. Define new random weights 
$w^*:E \to [0,\infty)$ as follows: $w^*|_{E(B(\Sigma R)^c} = w|_{E(B(\Sigma R)^c}$, but the $w^*(e), e \in E(B(\Sigma R)$
are i.i.d. $\nu$-distributed random variables, also independent of $w$. (Recall that for $S \subset V$, we define $E(S) \subset E$ to
be the set of edges of $G$ with endpoints lying in $S$).
Note that $w$ and $w^*$ are equal in distribution.
For each $R$ we will define a $w$-measurable random set of configurations $E_w \subset [0, \infty)^{E(B(\Sigma R))}$ such that
\begin{equation} \label{eq:resampleevents}
   \{ \pi \mbox{ crosses } B(o_i, \Sigma R) \} \cap A_i^R \cap \{ w^*|_{E(B(\Sigma R))} \in E_w \} 
   \subset \{ B(\Sigma R) \mbox{ contains a feasible pair for } \pi^* \},
\end{equation}
where $\pi$ is the $T$-geodesic from $x$ to $y$ and $\pi^*$ is the $T^*$-geodesic from $x$ to $y$.
To reduce clutter, let us abbreviate the event $\{ w^*|_{E(B(\Sigma R)} \in E_w\}$ by $\{ w^* \in E_w\}$.
If in addition we ensure that the conditional probability $\Prob( w^* \in E_w | w) \ge c(R) > 0$ on the event
$\{ \pi \mbox{ crosses } B(\Sigma R) \} \cap A_i^R$ (where $c(R)$ is some non-random constant), we get
\begin{align*}
   \Prob( B(\Sigma R) \mbox{ contains a feasible pair for } \pi )
   &= \Prob( B(\Sigma R) \mbox{ contains a feasible pair for } \pi^* ) \\
   &\ge \Prob( \{ \pi \mbox{ crosses } B(o_i, \Sigma R) \} \cap A_i^R \cap \{ w^* \in E_w \} )\\
   &= \E\left[ \ind_{\{ \pi \mbox{ crosses } B(o_i, \Sigma R) \} \cap A_i^R\}} \E[ \ind_{\{ w^* \in E_w \}} | w ]\right] \\
   &\ge c(R) \Prob( \{ \pi \mbox{ crosses } B(o_i, \Sigma R) \} \cap A^R_i ),
\end{align*}
as desired. The discussion in this section is summarized in following proposition:
\begin{prop} \label{prop:conditionaltovdBK}
   Suppose there exist $w$-measurable events $A_i^R$ satisfying the conditions of Lemma \ref{lem:peierls} 
   and $w$-measurable random sets of configurations $E_w$ such that for sufficiently large $R$ 
   \eqref{eq:resampleevents} holds and
   $\Prob(w^* \in E_w | w) \ge c(R)$ on the event $\{ \pi \mbox{ crosses } B(\Sigma R) \} \cap A_i^R$,
   where $c(R) > 0$ is a constant depending only on $R, \nu, \tilde{\nu},$ and $G$. Then 
   \[
      \liminf_{d(x,y) \to \infty} \frac{\E T(x,y) - \E \tilde{T}(x,y)}{d(x,y)} > 0.
   \]
\end{prop}
Thus the meat of the proof of Theorem \ref{thm:polygrowthvdBK} 
consists of performing a ``geometric'' construction to obtain suitable $A_i^R$ and $E_w$.

\subsection{Geometric construction: bounded case} \label{bddsupp}
First, suppose that $\nu$ has bounded support. We want to construct $A_i^R$ and $E_w$ satisfying the hypotheses
of Proposition \ref{prop:conditionaltovdBK}.
Denote by $\inf$ the infimum of the support of $\nu$
and denote by $\sup$ the supremum of the support of $\nu$.
Assume that \eqref{eq:extraassumption} holds, and then choose $\epsilon>0, y_0, I_0$
as in Lemma \ref{lem:technicallemma}. 
Assuming that $G$ admits detours, let $C$ be such that every self-avoiding path of length $C$
admits a self-avoiding $\epsilon$-detour. Set $C' := C(3 + 2\epsilon)$.
Assume that $\nu$ is exponential-subcritical, and
then let $q>0$ be the parameter given by Lemma \ref{lem:bddawayfrominf}.
Denote by $D$ the maximum degree of $G$.

First, let us consider the case that $y_0 = \sup$; in fact this allows us to do a much simpler construction. 
In this case, choose $\Sigma > 2$ large enough that $(\inf + q)\left(1 - \frac{1}{\Sigma} \right) > \inf$ and
choose $\delta > 0$ such that $\inf + \delta < (\inf + q)\left(1 - \frac{1}{\Sigma}\right)$.
Choose a sequence $\delta_{\sup}(R) \tendsto{R}{\infty} 0$ such that for each $R$ 
$\nu([\sup - \delta_{\sup}(R), \sup]) > 0$ but $\lim_{R \to \infty} \nu([\inf,\sup - \delta_{\sup}(R)])^{D C_1 (\Sigma R)^d} = 1$.
Then let $A_i^R := A_1 \cap A_2$ where $A_1$ and $A_2$ are as follows:
\[
   A_1 := \left\{ \begin{array}{c} 
                           \mbox{For all vertices } v,w \in B(\Sigma R) \mbox{ with } d(v,w) \ge R, \\ 
                           \mbox{ all paths } \gamma \mbox{ from } v \mbox{ to } w 
                           \mbox{ in } B(\Sigma R) \mbox{ have } T(\gamma) \ge (\inf + q)d(v,w)
                        \end{array} \right\}.
\]
\[
   A_2 := \left\{ w(e) \le \sup - \delta_{\sup} \mbox{ for all } e \in E(B(\Sigma R)) \right\}.
\]
We see that both events only depend on the weights of edges in $B(\Sigma R)$, by choice of $\delta_{\sup}(R)$ we
have $\Prob(A_2) \tendsto{R}{\infty} 1$, and by Lemma \ref{lem:bddawayfrominf} we have that for sufficiently large $R$
\[
   \Prob(A_1^c) \le \sum_{\substack{v,w \in B(\Sigma R),\\ d(v,w) \ge R}} \Prob( T(v,w) < (\inf + q) d(v,w) )
   \le (C_1 R^d)^2 e^{-\epsilon_2 R} \tendsto{R}{\infty} 0
\]
uniformly in $i$, so the hypotheses of Lemma \ref{lem:peierls} are satisfied.

Now in this case set of configurations $E_w$ does not actually depend on $w$; we simply set
\[
   E_w := \left\{ \omega \in [0,\infty)^{E(B(\Sigma R))} : \omega(e) \in 
                                                                           \begin{array}{lc}
                                                                              [\inf, \inf + \delta) &\mbox{ if } e \in E(B(\Sigma R - C')), \\ \relax
                                                                              [\sup, \sup - \delta_{\sup}] \cap I_0 &\mbox{ otherwise} 
                                                                           \end{array} \right\}.
\]

\begin{figure}[t]
   \centering
   \includegraphics[scale=.5]{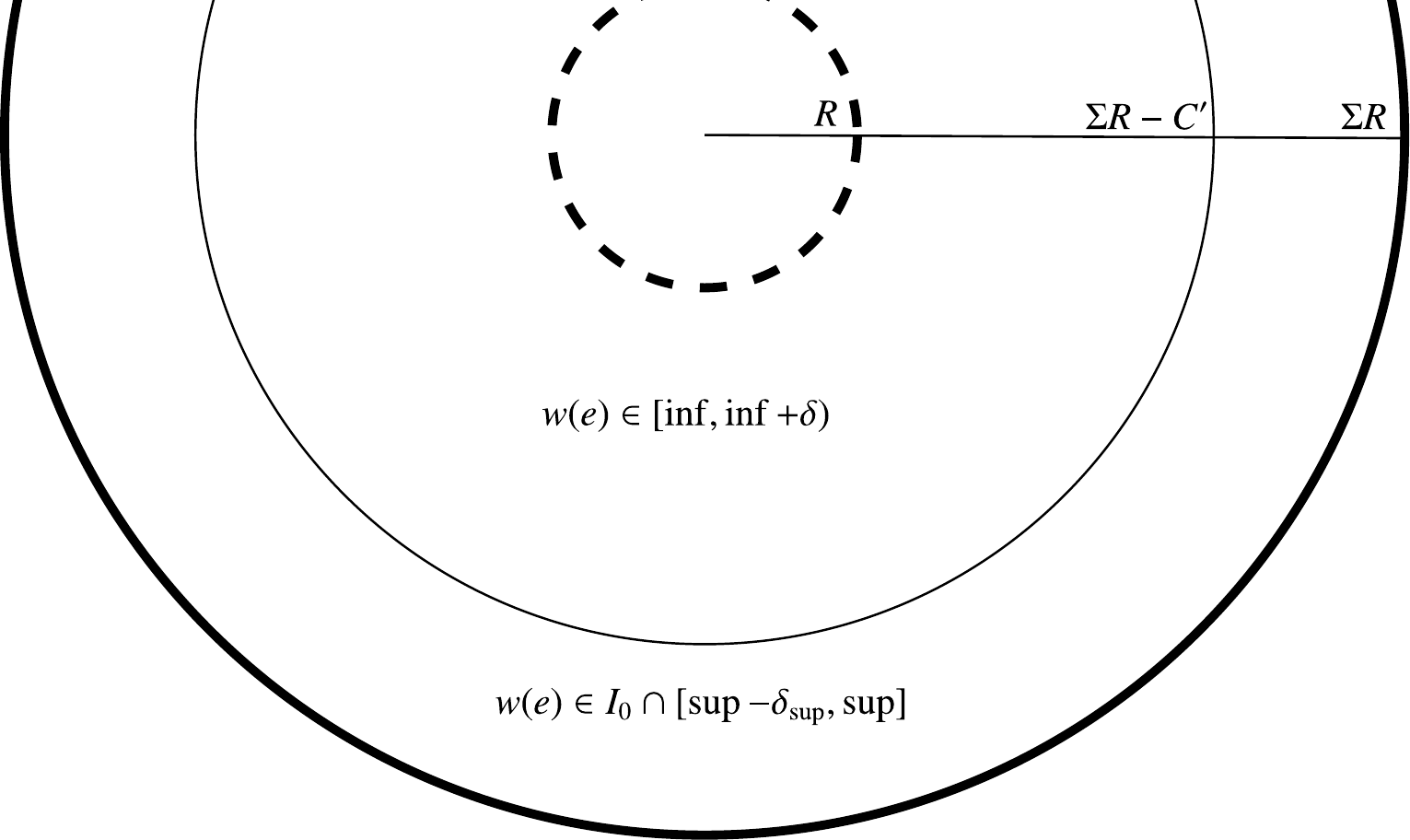}
   \caption{A schematic diagram of the prescribed set of configurations $E_w$ in the case
   that $\nu$ has bounded support and $y_0 = \sup$.}
   \label{fig:boundedconstruction1}
\end{figure}

Let us show that, for sufficiently large $R$, 
on the event $\{ \pi \mbox{ crosses } B(\Sigma R) \} \cap A_1 \cap A_2 \cap \{w^* \in E_w\}$, $B(\Sigma R)$ contains
a feasible pair for any $T^*$-geodesic.

For a subset $S \subset E$, denote by $T_S(p,q)$ the infimal weight of a path from $p$ to $q$ which only uses edges lying in $S$.
First, let $a$ and $b$ be points of $S(\Sigma R)$ such that $T_{E(B(\Sigma R))^c}(x,a)$ and $T_{E(B(\Sigma R))^c}(b,y)$
are infimal. Fix a $T$-geodesic $\alpha \subset E(B(\Sigma R))^c$ from $x$ to $a$, an edge geodesic $[a,o_i]$ from $a$ to $o_i$,
an edge-geodesic $[o_i,b]$ from $o_i$ to $b$, and a $T$-geodesic $\beta \subset E(B(\Sigma R))^c$ from $b$ to $y$,
and define $\pi' := \alpha * [a,o_i] * [o_i,b] * \beta$.
We claim that $T^*(\pi') < T(\pi)$ when $R$ is sufficiently large. To see this, first note that, if $v$ and $w$ are the first
and last vertices of $\pi$ lying on $S(\Sigma R)$, we have
\[
   T^*(\pi'_{x,a}) + T^*(\pi'_{b,y}) = T(\pi'_{x,a}) + T(\pi'_{b,y}) \le T(\pi_{x,v}) + T(\pi_{w,y}),
\]
where here and elsewhere, for a path $\gamma$ and vertices $p,q \in V(\gamma)$, $\gamma_{p,q}$ denotes the
subpath of $\gamma$ starting at $p$ and ending at $q$.

Next, since $\pi$ crosses $B_i$, $\pi_{v,w}$ contains at least two subsegments connecting $S(\Sigma R)$ and $S(R)$,
and so since $A_1$ holds we have
\[
   T(\pi_{v,w}) \ge 2(\inf + q)(\Sigma - 1)R = (\inf + q)\left(1 - \frac{1}{\Sigma}\right) 2 \Sigma R,
\]
while if $w^* \in E_w$, we have
\[
   T^*(\pi'_{a,b}) \le 2 \Sigma R(\inf + \delta) + (\sup)C'.
\]
Since by construction $\inf + \delta < (\inf + q)\left(1 - \frac{1}{\Sigma}\right)$ and $(\sup)C' = o(R)$,
for sufficiently large $R$ we have $T^*(\pi') < T(\pi)$.

Now, consider a $T^*$-geodesic $\pi^*$ from $x$ to $y$. On our event, we have $w^* \ge w$ on $E(B(\Sigma R - C'))^c$,
so if $\pi^*$ did not intersect $E(B(\Sigma R - C'))$, we would have $T^*(\pi^*) \ge T(\pi) > T^*(\pi')$, a contradiction.
Thus, $\pi^*$ must visit $B(\Sigma R - C')$. In particular, it contains a subpath connecting $S(\Sigma R)$ and $S(\Sigma R - C')$,
and so a subpath connecting $S(\Sigma R - C(1+\epsilon))$ and $S(\Sigma R - C' + C(1 + \epsilon))$, which must
have length at least $C' - 2C(1+\epsilon) = C$. Choose a self-avoiding $\epsilon$-detour $\gamma$ for such a segment.
Since $\gamma$ has length at most $C(1+\epsilon)$, it is contained in $E(B(\Sigma R)) \setminus E(B(\Sigma R - C'))$.
But since $w^* \in E_w$, this means that all the edges of both $\gamma$ and the subsegment of $\pi^*$ 
have weights in $I_0$. Hence $B(\Sigma R)$ contains a feasible pair for $\pi^*$.

Furthermore, since $y_0 = \sup$, by the construction of $I_0$ we have
\[
   \Prob( w^* \in E_w ) \ge \min \left( \nu([\inf, \inf + \delta)), \nu([\sup - \delta_{\sup}, \sup] \cap I_0) \right)^{D C_1 (\Sigma R)^d } > 0
\]
independent of $o_i$, so both hypotheses of Proposition \ref{prop:conditionaltovdBK} hold.

Now we suppose that $y_0 < \sup$ and do a different construction of the $A_i^R$ and $E_w$. 
Again take $\epsilon, y_0, I_0, C, C', q$ as above.
Then take some large $\Sigma_0 > 2$ such that
\[
   \inf < \left(1 - \frac{1}{\Sigma_0} \right) (\inf + q) < \sup;
\]
then take some $\delta_0 > 0$ sufficiently small that
\[
   \inf + \delta_0 < \left(1 - \frac{1}{\Sigma_0} \right) (\inf + q) < \sup,
\]
\[
   \sup - \E w - 2\delta_0 > 0,
\]
and
\[
   \sup - y_0 - 2\delta_0 > 0.
\]
(Note that $\E w < \sup$ since in the case that $\nu$ is Dirac, $y_0 = \sup$).
Next, fix some $0 < s < \left( 1 - \frac{1}{\Sigma_0} \right) \frac{(\inf+q)}{\sup}$ such that
\[
   (\inf + \delta_0) + s \sup < \left(1 - \frac{1}{\Sigma_0}\right) (\inf + q).
\]
Then fix some $\Sigma \ge \Sigma_0$ such that $s\Sigma > 1$.
Also fix some 
some $0 < \kappa < \frac{\sup - \delta_0 - \E w }{ \sup - \inf } s$.

The event $A_i^R$ will be defined as the intersection of three events $A_1 \cap A_2 \cap A_3$.
We set
\[
   A_1 := \left\{ \begin{array}{c} 
                           \mbox{For all vertices } v,w \in B(\Sigma R) \mbox{ with } d(v,w) \ge R, \\ 
                           \mbox{ all paths } \gamma \mbox{ from } v \mbox{ to } w 
                           \mbox{ in } B(\Sigma R) \mbox{ have } T(\gamma) \ge (\inf + q)d(v,w)
                        \end{array} \right\},
\]
just as in the first case.
We set
\[
   A_2 := \left\{ \begin{array}{c} 
                           \mbox{For all vertices } v,w \in B(\Sigma R) \mbox{ with } d_{E(B(\Sigma R))}(v,w) \ge R, \\ 
                           T_{E(B(\Sigma R))}(v,w) \le (\E w + \delta_0) d_{E(B(\Sigma R))}(v,w)
                           \end{array} \right\}.
\]
For this, note that for each fixed pair of points $v,w$ with $d_{E(B(\Sigma R))}(v,w) \ge R$, fixing an edge-minimal path $\gamma: v \to w$ in $B(\Sigma R)$,
we have 
\begin{align*}
   \Prob(T_{E(B(\Sigma R))}(v,w) > (\E w + \delta_0) d_{E(B(\Sigma R))}(v,w)) 
   &\le \Prob( T(\gamma) > (\E w + \delta_0) |\gamma| ), \\
\end{align*}
which, since $T(\gamma)$ is just a sum of $|\gamma|$ i.i.d. $\nu$-distributed random variables, decays exponentially in $|\gamma|$, (hence $R$),
by a standard Chernoff bound ($\nu$ has bounded support and hence exponential moments).
Since the number of pairs of such $(v,w)$ is strictly subexponential in $R$, we have $\Prob(A_2^c) \tendsto{R}{\infty} 0$, as desired.
Clearly also $A_2$ only depends on the weights of edges in $B(\Sigma R)$.

Lastly we choose for each $R$ some $0 \le \delta_{\sup}(R) < \delta_0$ such that $\nu([\sup - \delta_{\sup}, \sup]) > 0$
and $\nu([\inf,\sup - \delta_{\sup}(R)])^{D C_1 (\Sigma R)^d} \tendsto{R}{\infty} 1$, and then set
\[
   A_3 := \left\{ w(e) \le \sup - \delta_{\sup} \mbox{ for all } e \in E(B(\Sigma R)) \right\}.
\]
Clearly $A_3$ only depends on edges in $B(\Sigma R)$ and by our construction of $\delta_{\sup}(R)$, 
we have $\Prob(A_3) \tendsto{R}{\infty} 1$ uniformly in $i$, as desired.

Now, let $a',b' \in S(\Sigma R)$ be such that $T_{E(B(\Sigma R))^c}(x,a')$ and $T_{E(B(\Sigma R))^c}(b',y)$ are minimal.
Choose edge geodesics $[a',o_i]$ and $[b',o_i]$. Let $a \in V([a',o_i]), b \in V([b',o_i])$ be the unique vertices such that
$d(a,a'), d(b,b') = \lceil s \Sigma R \rceil$.
Moreover, for each $t \in [0,\Sigma R - \lceil s \Sigma R \rceil] \cap \Z$, let $a_t \in V([a,o_i]), b_t \in V([b,o_i])$ be the unique vertices such that
$d(a,a_t), d(b,b_t) = t$.
Now, let $t_a \ge 0$ be minimal such that
\[
   d(a_{t_a+1}, [b,o_i]) \le 2C',
\]
and let $t_b \ge 0$ be minimal such that
\[
   d(b_{t_b+1}, [a,o_i]) \le 2C',
\]
and set $ c:= a_{t_a}$, $d:=b_{t_b}$.
Note that minimality implies that for all $0 \le t \le t_a$ we have $d(a_t,[b,o_i]) \ge 2C' + 1$ and for all $0 \le t \le t_b$
we have $d(b_t,[a,o_i]) \ge 2C' + 1$.
Here we have tacitly used the fact that $d(a,b) \ge d(a',b') - 2\lceil s \Sigma R \rceil \gessim R$
is strictly larger than $2C'$ for sufficiently large $R$. 
To see the bound $d(a',b') - 2\lceil s \Sigma R \rceil \gessim R$, 
let $v$ and $w$ be the entry and exit points from $B(\Sigma R)$ of the $T$-geodesic $\gamma: x \to y$, and
note that
\[
   d(a',b') \ge \frac{1}{\sup} T(a',b') \ge \frac{1}{\sup} T(v,w) \ge \frac{\inf + q}{\sup} \left(1 - \frac{1}{\Sigma}\right) 2 \Sigma R,
\]
so 
\[
   d(a,b) \ge d(a',b') - 2\lceil s\Sigma R \rceil \ge \left[ \frac{\inf + q}{\sup} \left(1 - \frac{1}{\Sigma}\right) - s \right](2 \Sigma R) - 1,
\]
which is $\gessim R$ by choice of $s$.
The bound $T(a',b') \ge T(v,w)$ comes from the fact that, since $v,w$ lie on the $T$-geodesic from $x$ to $y$,
$T(x,y) = T(x,v) + T(v,w) + T(w,y) \le T(x,a') + T(a',b') + T(b',y)$, and by definition of $a',b'$ we have
$T(x,v) + T(w,y) \ge T(x,a') + T(b',y)$.
The bound $T(v,w) \ge (\inf + q) \left(1 - \frac{1}{\Sigma}\right) 2 \Sigma R$ comes from the fact that
$\pi$ crosses $B_i$ and hence contains at least two paths connecting $S(\Sigma R)$ and $S(R)$,
which, since $A_1$ holds, have total passage time at least $2(\inf + q)(\Sigma - 1)R$.


Now consider the sets of integers
\[
   S_n(C',\kappa) := \{ n \lfloor \kappa \Sigma R \rfloor + j : j \in [0,C'] \cap \Z \} \subset \Z
\]
and
\[
   S'_n(C',\kappa) := \{ n \lfloor \kappa \Sigma R \rfloor + j : j \in [C(1+\epsilon),C(2+\epsilon)] \cap \Z \} \subset \Z,
\]
where $n \ge 0, n \in \Z$. Then let $\alpha_n$ and $\beta_n$ respectively be the subpaths of $[a,c]$ and $[b,d]$ respectively induced by the 
vertex sets $\{ a_t : t \in S_n \}$ and $\{ b_t : t \in S_n \}$. 
Similarly let $\alpha'_n$ and $\beta'_n$ be induced by $\{ a_t : t \in S'_n \}$ and $\{ b_t : t \in S'_n \}$.
For each $n \ge 0$ with $(n+1)\lfloor \kappa \Sigma R \rfloor \le t_a, t_b$,
fix a self-avoiding $\epsilon$-detour $\gamma_n$ for $\alpha'_n$ and a self-avoiding $\epsilon$-detour $\delta_n$ for $\beta'_n$.
Note that by construction each $\alpha_n \cup \gamma_n$ is disjoint from $[b,d]$ and all $\beta_m \cup \delta_m$, and vice versa.
Moreover, $\alpha_n \cup \gamma_n$ is disjoint from $\alpha_m \cup \gamma_m$ for $n \ne m$, and the same is true for the
$\beta_n \cup \delta_n$.

Finally, define 
\[
   S_I := \bigcup_{\substack{n \ge 0, \\ (n+1)\lfloor \kappa \Sigma R \rfloor \le t_a, t_b}}  (\alpha_n \cup \gamma_n) \cup (\beta_n \cup \delta_n),
\]
define
\[
   S_{\inf} := ([a,c] \cup [b,d]) \setminus S_I,
\]
and set $S_{\sup} := E(B(\Sigma R)) \setminus (S_{\inf} \cup S_I)$.
For each $R$ we choose $0<\delta_{\inf}(R) < \delta_0$ sufficiently small that 
$(D C_1 R^d +2)\delta_{\inf} < \sup - \delta_0 - y_0$.
We finally define our random set of configurations by
\begin{align*}
   E_w :=
   \left \{ \omega \in [0, \infty)^{E(B(\Sigma R))} : 
                 \omega(e) \in 
                      \begin{array}{lc}
                         I_0 \cap (y_0 - \frac{\delta_{\inf}}{2}, y_0 + \frac{\delta_{\inf}}{2}) & e \in S_I \\ \relax
                         [\inf, \inf + \delta_{\inf}] & e \in S_{\inf} \\ \relax
                         [\sup - \delta_{\sup}, \sup] & e \in S_{\sup}
                      \end{array} 
   \right \}.
\end{align*}

\begin{figure}[t]
   \centering
   \includegraphics[scale=.65]{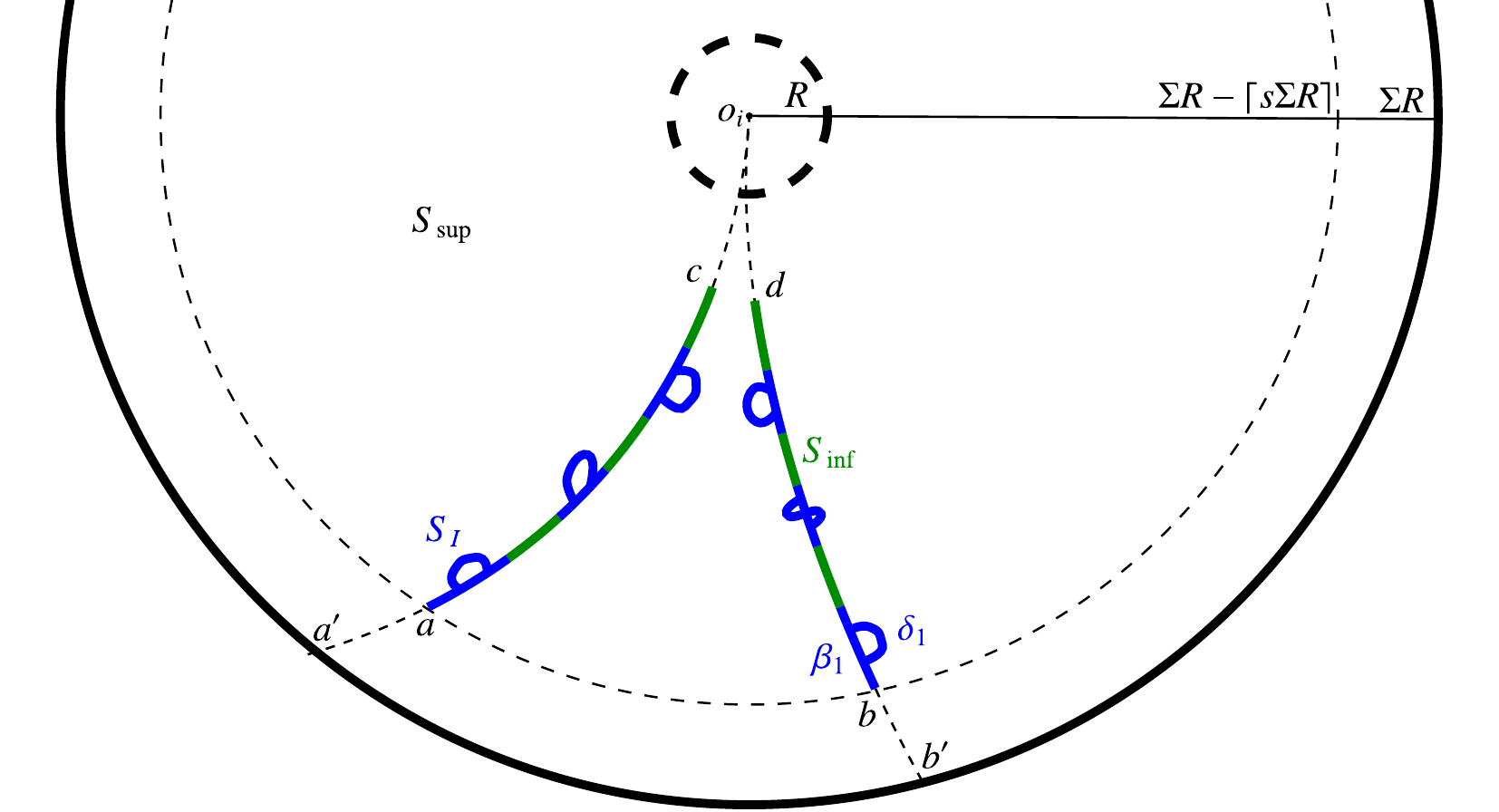}
   \caption{A schematic diagram of the prescribed set of configurations $E_w$ in the case
   that $\nu$ has bounded support and $y_0 < \sup$.}
   \label{fig:boundedconstruction2}
\end{figure}

Now let us prove that $A_1 \cap A_2 \cap A_3 \cap \{ \pi \mbox{ crosses } B_i \} \cap \{ w^* \in E_w \}$
is contained in the event that $B(\Sigma R)$ contains a feasible pair with respect to $T^*$.

First, define a path $\pi'$ by taking a $T$-geodesic from $x$ to $a'$ in $B(\Sigma R)^c$, then taking the path $[a',c]$,
taking an edge-geodesic from $c$ to $d$, taking $[d,b']$ and then taking a $T$-geodesic from $b'$ to $y$ in $B(\Sigma R)^c$.
For all sufficiently large $R$, on the event $\{\pi \mbox{ crosses } B_i\} \cap A_i^R \cap \{ w^* \in E_w \}$, we 
have that $T^*(\pi') < T(\pi)$. To see this, first note that by definition of $a',b'$, if $v,w$ are the first entrance and last exit of $\pi$ from $B(\Sigma R)$
then we have $T(\pi_{x,v}) + T(\pi_{w,y}) = T(x,v) + T(w,y) \ge T(x,a') + T(b',y) = T^*(\pi'_{x,a'}) + T^*(\pi'_{b',y})$.
Thus it suffices to show that $T(\pi_{v,w}) > T^*(\pi'_{a',b'})$ for sufficiently large $R$. Since $\pi$ visits $B_i \subset B(R)$
and since $A_1$ holds we have
\[
   T(\pi_{v,w}) \ge (\inf + q)2(\Sigma - 1)R = (\inf + q)\left( 1 - \frac{1}{\Sigma} \right) 2 \Sigma R,
\]
whereas 
\begin{align*}
   T(\pi'_{a',b'}) &\le \left[ (\inf + \delta_{\inf}) + \frac{C'}{\lfloor \kappa \Sigma R \rfloor} (y_0 + \delta_0) \right](d(a,c) + d(b,d))  
   + (\sup) (2s \Sigma R + d(c,d)) \\
   &\le (\inf + \delta_{\inf} + (\sup)s ) 2 \Sigma R + o(R),
\end{align*}
so this follows from our choice to ensure $\inf + \delta_{\inf} + (\sup) s < (\inf + q)\left( 1 - \frac{1}{\Sigma} \right)$.
(We get the bound $d(c,d) = o(R)$ as follows: assume that $t_a \le t_b$; in the opposite case the argument is analogous.
By definition there exists some $t' \ge t_b + 1$ such that $d(a_{t_a + 1}, b_{t'}) \le 2C'$. But then 
\[ |t' - (t_a + 1)| = |d(o_i, b_{t'}) - d(o_i, a_{t_a + 1})| \le d(b_{t'}, a_{t_a + 1}) \le 2C', \] 
that is, $t' \le t_a + 1 + 2C' \le t_b + 1 + 2C'$,
and so 
\[
   d(c,d) \le d(c,a_{t_a +1}) + d(a_{t_a + 1}, b_{t'}) + d(b_{t'}, b_{t_b}) \le 4C' + 2 = O(C') = o(R).)
\]

Now, let $\pi^*$ be a $T^*$-geodesic from $x$ to $y$. We show that $\pi^*$ traverses a feasible pair.

We first show that if $p,q \in V(\pi^*) \cap V(S_{\inf})$ with $p$ and $q$ lying in the same connected component of $S_{\inf} \cup S_I$,
then $\pi^*_{p,q} \subset S_{\inf} \cup S_I $.
To see this, note that, when $w^* \in E_w$, if $e$ is an edge in $[a,c]$ or $[b,d]$ with one endpoint in $S(t)$ and one in $S(t+1)$, then
\[
   w^*(e) \le \inf \{ w^*(e') : e' \mbox{ has one endpoint in } S(t) \mbox{ and the other in } S(t+1) \} + \delta_{\inf}.
\]
This is because, if $e \in S_{\inf}$, then $e' \in S_{\inf}$ or $e' \in S_{\sup}$ and if $e \in S_I$ then $e' \in S_I$ or $e' \in S_{\sup}$.

Since every path from $p$ to $q$ must have at least one edge connecting $S(t)$ to $S(t+1)$ for all $t,t+1$
between $d(a,p)$ and $d(a,q)$, we see that
\[
   T^*([p,q]) \le T^*(\alpha) + \delta_{\inf} |\alpha|
\]
for any path $\alpha$ from $p$ to $q$. If furthermore $\alpha$ leaves $S_{\inf} \cup S_I$, then it contains at least one edge
of weight at least $\sup - \delta_{\sup}$; such an edge has weight at least $\sup - \delta_{\sup} - y_0 - \delta_{\inf}$
greater than any edge in $[p,q]$. Hence in this case we get the bound
\[
   T^*([p,q]) + \sup - \delta_{\sup} - y_0 - \delta_{\inf} \le T^*(\alpha) + \delta_{\inf} (|\alpha| - 1).
\]
But applying our assumption on $\delta_{\inf}(R)$ we get
\[
   T^*(\alpha) - T^*([p,q]) \ge \sup - \delta_{\sup} - y_0 - (|\alpha| + 2)\delta_{\inf} \ge \sup - \delta_0 - y_0 - (|B(\Sigma R)|+2)\delta_{\inf}
   > 0.
\]
That is, such an $\alpha$ is not optimal, and hence an optimal $T^*$-path $\pi^*_{p,q}$ must lie in $S_{\inf} \cup S_I$.

Hence, if we can show that $V(\pi^*)$ contains some $p$ and $q$ which lie in the same connected component of $S_{\inf} \cup S_I$
but lie in different components of $S_{\inf}$, then we can apply the previous argument to deduce that
$\pi^*$ passes through some $\alpha_n \cup \gamma_n$ or $\beta_n \cup \delta_n$, and then use
the following proposition to conclude that $\pi^*_{p,q}$ contains a feasible pair:
\begin{prop} \label{prop:detourtodetour}
   Let $\xi$ be an edge geodesic in $G$, and let $\gamma$ be a self-avoiding $\epsilon$-detour for a subpath of $\xi$.
   Suppose that $w^*(e) \in I_0$ for all $e \in \xi \cup \gamma$.
   Let $\pi^*$ be a $T^*$-geodesic, and suppose that some subpath of $\pi^*$ has the same endpoints as $\xi$
   and that this subpath is contained in $\xi \cup \gamma$. Then $\xi \cup \gamma$ contains a feasible pair for $\pi^*$.\footnote{
   Technically we should include assumptions controlling the lengths of these paths to satisfy our definition of a feasible pair;
   in our applications of this proposition it is easy to see that the length of the detour is at most $C'(1+\epsilon)$.}
\end{prop}
\begin{proof}[Proof of Proposition]
   Let $\xi'$ be the subpath of $\xi$ such that $\gamma$ is an $\epsilon$-detour for $\xi'$, and write $\xi = \xi_1 * \xi' * \xi_2$.
   Let us also abuse notation and denote by $\pi^*$ the subpath of $\pi^*$ contained in $\xi \cup \gamma$ which has the same endpoints as $\xi$.
   If $\pi^* = \xi$,
   then $\xi_1 * \gamma * \xi_2$ is an $\epsilon$-detour for $\pi^*$; loop-erasing then gives a \emph{self-avoiding} $\epsilon$-detour
   $\gamma'$ for $\xi$ (see the proof of Proposition
   \ref{prop:detourequiv}), so $(\pi^*, \gamma')$ forms a feasible pair.
   If $\pi^* \ne \xi$, then since $\xi$ is an edge geodesic, $\xi$ is a self-avoiding $\epsilon$-detour for $\pi^*$ (see the proof of Proposition
   \ref{prop:detourequiv}), and hence $(\pi^*, \xi)$ forms a feasible pair.
\end{proof}
So it only remains to find such $p$ and $q$.
The idea is that, in order to make up for the slow edges $\pi^*$ runs over when it enters and exits $B(\Sigma R)$,
$\pi^*$ must visit many fast edges; we will then use the pigeonhole principle to conclude that it must contain suitable $p$ and $q$.

Explicitly, first note that since $T^*(\pi^*) \le T^*(\pi') < T(\pi) \le T(\pi^*)$, we have $T(\pi^*) - T^*(\pi^*) > 0$.
Since $w^* \ge w$ on $E(B(\Sigma R))^c \cup S_{\sup}$, $\pi^*$ must therefore contain some edges in $S_I \cup S_{\inf}$.
But note that by construction, any path connecting $S(\Sigma R)$ and $S_I \cup S_{\inf}$ contains a
a subpath which lies in $S_{\sup}$ and connects two points in $B(\Sigma R)$ of distance at least $s \Sigma R > R$.
Since $\pi^*$ starts and ends outside of $B(\Sigma R)$ and visits $S_{\inf} \cup S_I$, it contains at least
two such subpaths, $\alpha$ and $\beta$. We then have
\[
   T^*(\alpha) \ge (\sup - \delta_{\sup}) |\alpha|, T^*(\beta) \ge (\sup - \delta_{\sup}) |\beta|
\]
and 
\[
   T(\alpha) \le (\E w) |\alpha|, T(\beta) \le (\E w) |\beta|
\]
(since $A_2$ holds). Since $w^* \ge w$ on $S_{\sup}$ we then have
\[
   T^*(\pi^* \cap S_{\sup}) - T(\pi^* \cap S_{\sup}) \ge  T^*(\alpha \cup \beta) - T(\alpha \cup \beta)
   \ge (\sup - \E w - \delta_{\sup}) s (2 \Sigma R).
\]
Since $T^*(\pi^* \cap E(B(\Sigma R))^c) - T(\pi^* \cap E(B(\Sigma R))^c) = 0$, in order to ensure that $T^*(\pi^*) - T(\pi^*) < 0$,
it must be the case that
\[
   T(\pi^* \cap (S_{\inf} \cup S_I)) - T^*(\pi^* \cap (S_{\inf} \cup S_I)) > (\sup - \E w - \delta_{\sup}) s (2 \Sigma R).
\]
Since each edge $e$ admits savings at most $w(e) - w^*(e) \le \sup - \inf$, this gives
\[
   |\pi^* \cap (S_{\inf} \cup S_I)| > \frac{ \sup - \E w - \delta_{\sup} }{ \sup - \inf } s(2 \Sigma R).
\]
Moreover, since each component of $S_I$ is composed of less than $2C'$ edges
\[
   |S_I| \le 2C' \frac{2 \Sigma R}{\lfloor \kappa \Sigma R \rfloor} = O(C') = o(R),
\]
and so
\[
   |\pi^* \cap S_{\inf}| \ge \frac{ \sup - \E w - \delta_{\sup} }{ \sup - \inf } s(2 \Sigma R) - o(R);
\]
since by assumption $\kappa < \frac{ \sup - \delta_0 - \E w }{\sup - \inf} s$, for sufficiently large $R$ we have in particular
\[
      |\pi^* \cap S_{\inf}| > 2 \kappa \Sigma R.
\]
Since $S_{\inf} \cup S_I$ has two connected components, at least one of the components contains more than
$\kappa \Sigma R$ edges of $\pi^* \cap S_{\inf}$. But each connected component of $S_{\inf}$ contains at most 
$\lfloor \kappa \Sigma R \rfloor - C'$ edges, so $V(\pi^*)$ must contain some pair of points $p,q$ which lie in 
different connected components of $S_{\inf}$ but in the same connected component of $S_{\inf} \cup S_I$, as desired.
Thus, this construction satisfies \eqref{eq:resampleevents}. 

To see that the construction satisfies the other hypothesis of Proposition \ref{prop:conditionaltovdBK}, note that
\begin{align*}
   \Prob(w^* \in E_w | w) &= 
   \nu([\inf, \inf + \delta_{\inf}))^{|S_{\inf}|} 
   \nu(I_0 \cap (y_0 - \frac{\delta_{\inf}}{2}, y_0 + \frac{\delta_{\inf}}{2}))^{|S_I|} 
   \nu([\sup - \delta_{\sup}, \sup])^{|S_{\sup}|} \\
   &\ge 
   \min\left( \nu([\inf, \inf + \delta_{\inf})), \nu(I_0 \cap (y_0 - \frac{\delta_{\inf}}{2}, y_0 + \frac{\delta_{\inf}}{2}) ), \nu([\sup - \delta_{\inf}, \sup])\right)^{
   D C_1 (\Sigma R)^d}
\end{align*}
is bounded away from $0$ independently of $o_i$ and $x,y$, as desired.

\subsection{Geometric construction: unbounded case} \label{unbddsupp}
Now, suppose $\nu$ has unbounded support. We construct the relevant events $A_i^R$ and configurations $E_w$ and show that
they satisfy \eqref{eq:resampleevents}.
The main challenge for the case that $\nu$ has unbounded support is in ensuring that the beginning and end of our
prescribed path are far enough away from each other that we ``have enough room'' to make a segment and a detour which
don't collide with the rest of the path. Once we construct our prescribed path it will not be hard to force the 
resampled geodesic $\pi^*$ to take it, since we can resample the prescribed path to have very small passage time
and resample the surrounding edges to have arbitrarily large passage time.

Again assume that \eqref{eq:extraassumption} holds, and then choose $\epsilon>0, y_0, I_0$
as in Lemma \ref{lem:technicallemma}. 
Assume that $\nu$ is exponential-subcritical and
let $q > 0$ be the parameter from Lemma \ref{lem:bddawayfrominf}.
Then fix $\sigma > \max(2, \frac{2(\inf + q)}{q})$ and $\Sigma > \sigma$.
The event $A_i^R$ will be constructed as the intersection of five events $A_i^R := A_1 \cap A_2 \cap A_3 \cap A_4 \cap A_5$.
The first event is
\[
   A_1 := \{\mbox{every path } \gamma: v \to w \mbox{ in } B(\Sigma R) \mbox{ with } d(v,w) \ge R \mbox{ satisfies } 
   T(\gamma) \ge (\inf + q)d(v,w)\}.
\]
This evidently only depends on the edges in $E(B(\Sigma R))$. Moreover, by Lemma \ref{lem:bddawayfrominf}, 
for all sufficiently large $R$ we have
\[
   \Prob(A_1^c) \le \sum_{\substack{ v,w \in B(o_i, \Sigma R), \\ d(v,w) \ge R}} \Prob(T(v,w) < (\inf + q)d(v,w))
   \le |B(o_i,\Sigma R)| e^{-c R} \le C_1 R^d e^{-c R} \tendsto{R}{\infty} 0.
\]
For the next event we choose $\delta_{\inf}(R) \tendsto{R}{\infty} 0$ such that $\nu([\inf, \inf+\delta_{\inf}]) > 0$ and
${DC_1(\Sigma R)^d \delta_{\inf}(R) \le 1}$ for all $R$,
and $ {\left(\nu([\inf + \delta_{\inf}(R), \infty))\right)^{DC_1(\Sigma R)^d} \tendsto{R}{\infty} 1}$. (Note that if there is an atom at $\inf$, then eventually
we will have $\delta_{\inf}(R)=0$, but $\delta_{\inf}(R) \ge 0$ always). Note that the second condition
implies in particular that $|E(B(\Sigma R))| \delta_{\inf}(R) \le 1$. We define
\[
   A_2 := \{ w(e) \ge \inf + \delta_{\inf} \mbox{ for all } e \in E(B(\Sigma R)) \}.
\]
This clearly only depends on the weights in $E(B(\Sigma R))$ and the third condition on $\delta_{\inf}(R)$ implies that 
\[
   \Prob(A_2) = \nu([\inf + \delta_{\inf}, \infty))^{|E(B(\Sigma R)|} \ge \nu([\inf + \delta_{\inf}, \infty))^{DC_1(\Sigma R)^d} \tendsto{R}{\infty} 1.
\]
For the third event, we choose $M(R) \tendsto{R}{\infty} \infty$ such that $\nu^{*DC_1(\Sigma R)^d}([0,M(R)]) \tendsto{R}{\infty} 1$.
We set
\[
   A_3 := \left\{ \sum_{e \in E(B(\Sigma R))} w(e) \le M \right\}.
\]
It is clear by the choice of $M(R)$ that $\Prob(A_3) \tendsto{R}{\infty} 1$. Also note that since $\nu$ is assumed to have infinite support,
$\nu( (M(R), \infty)) > 0$ for all $R$.

Let us call a value $p \in \supp \nu$ \emph{$(\delta,\eta)$-resamplable} if $\nu([p,p+\delta)) \ge \eta$. 
Set $\delta_{sim}(R) := (D C_1 R^d)^{-1}$.
Then, using Proposition \ref{prop:resampling} below, choose $\eta(R) > 0$ such that
\[
   \nu( \{ p : p \mbox{ is } (\delta_{sim}(R), \eta(R)) \mbox{-resamplable} \} )^{DC_1 R^d} \ge 1 - e^{-R}.
\]
Set
\[
   A_4 := \left\{ w(e) \mbox{ is } (\delta_{sim},\eta) \mbox{-resamplable for all } e \in E(B(\Sigma R)) \right\}.
\]
Clearly $A_4$ only depends on weights of edges in $E(B(\Sigma R))$, and by our choice of $\eta(R)$ we have
\[
   \Prob(A_4) \ge 1 - e^{-R} \tendsto{R}{\infty} 1.
\]

The event $A_5$ is more complicated to describe, so we delay its description and the proof that $\Prob(A_5) \tendsto{R}{\infty} 1$
until the end of the section.

Next we describe the construction of $E_w$. Denote by $\pi$ the geodesic from $x$ to $y$, and denote by $v$ and $w$
the first vertex of $\pi$ which lies in $B(\Sigma R)$ and the last vertex of $\pi$ which lies in $B(\Sigma R)$, respectively.
As will be proved in Lemma \ref{lem:a5} at the end of the section, the event $A_5$ implies that, for some $\Sigma R \ge r \ge \sigma R$,
we have disjoint self-avoiding paths $\alpha$ and $\beta$ with the following properties:
\begin{enumerate} \label{pathconditions}
   \item $\alpha$ starts at $x$ and ends at a point $v' \in S(r)$; moreover $V(\alpha) \cap B(r-1) = \emptyset$.
   \item $\beta$ starts at a point $w' \in S(r)$ and ends at $y$; moreover $V(\beta) \cap B(r-1) = \emptyset$.
   \item $d_{E(B(r))}(v',w') > K := 4C(1 + \epsilon)$.
   \item $\alpha$ coincides with $\pi$ until its last entrance into $B(\Sigma R)$
            and $\beta$ coincides with $\pi$ after its first exit from $B(\Sigma R)$.
            Explicitly,
            Choose $\tilde{v}$ to be the last entrance of $\alpha$ into $B(\Sigma R)$, so that $\alpha_{\tilde{v},v'}$
            is the connected component of $E(B(\Sigma R)) \cap \alpha$ containing $v'$.
            Similarly choose $\tilde{w}$ to be the first exit of $\beta$ from $B(\Sigma R)$, so that 
            $\beta_{w',\tilde{w}}$ is the connected component of $E(B(\Sigma R) \cap \beta$ containing $w'$.
            We have that $\tilde{v},\tilde{w} \in V(\pi)$, and $\pi_{x,\tilde{v}} = \alpha_{x,\tilde{v}}$
            and $\pi_{\tilde{w},y} = \beta_{\tilde{w},y}$.
   \item Let $v_r \in S(r)$ be the vertex of $\pi$ immediately preceding the first vertex of $\pi$ which lies in $B(r-1)$, 
            and let $w_r \in S(r)$ be the vertex of $\pi$ immediately following the last vertex of $\pi$ which lies in $B(r-1)$.
            Then $|\alpha_{\tilde{v},v'}| \le |\pi_{\tilde{v},v_r}|$ and $|\beta_{w',\tilde{w}}| \le |\pi_{w_r,\tilde{w}}|$.
\end{enumerate}
Now choose edge-geodesics $[v',o_i]$ from $v'$ to $o_i$ and $[o_i,w']$ from $o_i$ to $w'$.
Again let $C$ be such that every self-avoiding path of length $C$ admits a self-avoiding $\epsilon$-detour.
Let $a$ be the vertex of $[v',o_i]$ which is distance $C(1+\epsilon)$ from $v'$.
Let $b$ be the vertex of $[v',o_i]$ which is distance $C(2+\epsilon)$ from $v'$.
Then $[a,b]:=[v',o_i]_{a,b}$ is a self-avoiding path of length $C$, and hence it admits a self-avoiding $\epsilon$-detour $\gamma$.
\begin{prop}
   $\gamma$ is contained in $B(r-1)$ and $V(\gamma) \cap V([o_i,w']) = \emptyset$.
\end{prop}
\begin{proof}
   The first claim follows from the fact that $\gamma$ has length at most $C(1+\epsilon)$;
   To see the second claim, suppose to the contrary
   that there was some $z \in V(\gamma) \cap V([o_i,w'])$. Since $d_{B(r)}(z,a) \le C(1+\epsilon)$ and $z$ and $a$ both lie on edge-geodesics
   to $o_i$, we have that 
   \begin{align*}
      |d_{B(r)}(v',a) - d_{B(r)}(w',z)| &= |[d_{B(r)}(v',o_i) - d_{B(r)}(a,o_i)] - [d_{B(r)}(w',o_i) - d_{B(r)}(z,o_i)]|  \\
      &= |d_{B(r)}(a,o_i) - d_{B(r)}(z,o_i)| \le d_{B(r)}(a,z) \le C(1+\epsilon),
   \end{align*}
   and therefore
   \[
      d_{B(r)}(w',z) \le d_{B(r)}(v',a) + C(1+\epsilon) = 2C(1 + \epsilon),
   \]
   hence
   \[
      d_{B(r)}(v',w') \le d_{B(r)}(v',a) + d_{B(r)}(a,z) + d_{B(r)}(z,w') \le 2C(1+\epsilon) + 2C(1 + \epsilon) = 4C(1 + \epsilon) = K,
   \]
   contradicting the fact that $d(v',w') > K$.
\end{proof}
Set $b'$ to be the vertex in $V([v',o_i])$ which has distance $C' = C(3 + 2\epsilon)$ from $v'$.
Set $o'$ to be the first intersection of $V([v',o_i])$ with $V([o_i,w'])$; the previous proposition shows
that $o'$ is strictly closer to $o_i$ than $b'$.
Define as usual $[v',o']:=[v',o_i]_{v',o'}$, $[o',w']:=[o_i,w']_{o',w'}$, $[v',b'] := [v',o_i]_{v',b'}$.
We now define the following subsets of $E(B(\Sigma R))$:
\begin{align*}
   S_I &:= [v',b'] \cup \gamma, \\
   S_{\inf} &:= \left(\alpha_{\tilde{v},v'} * [v',o'] * [o',w'] * \beta_{w',\tilde{w}}\right) \setminus S_I, \\
   S_{sim} &:= (\alpha \cup \beta) \cap E(B(\Sigma R)) \setminus S_{\inf} \\
   S_{M} &:= E(B(\Sigma_R)) \setminus (S_I \cup S_{\inf} \cup S_{sim}).
\end{align*}
Note that these sets are all pairwise disjoint and cover $E(B(\Sigma R))$.
Now we can finally define our set of configurations $E_w$:
\begin{align*}
   E_w :=
   \left \{ \omega \in [0, \infty)^{E(B(\Sigma R))} : 
                 \omega(e) \in 
                      \begin{array}{lc}
                         I_0 & e \in S_I \\ \relax
                         [\inf, \inf + \delta_{\inf}] & e \in S_{\inf} \\ \relax
                         [w(e), w(e) + \delta_{sim}) & e \in S_{sim} \\ \relax
                         [M, \infty) & e \in S_M 
                      \end{array} 
   \right \}.
\end{align*}
(We have used the assumption that $w \in A_1 \cap A_2 \cap A_3 \cap A_4 \cap A_5$ to construct $E_w$, and this is really the only case
we care about; off of this event we may define $E_w = \emptyset$).
\begin{figure}[t]
   \centering
   \includegraphics[scale=.45]{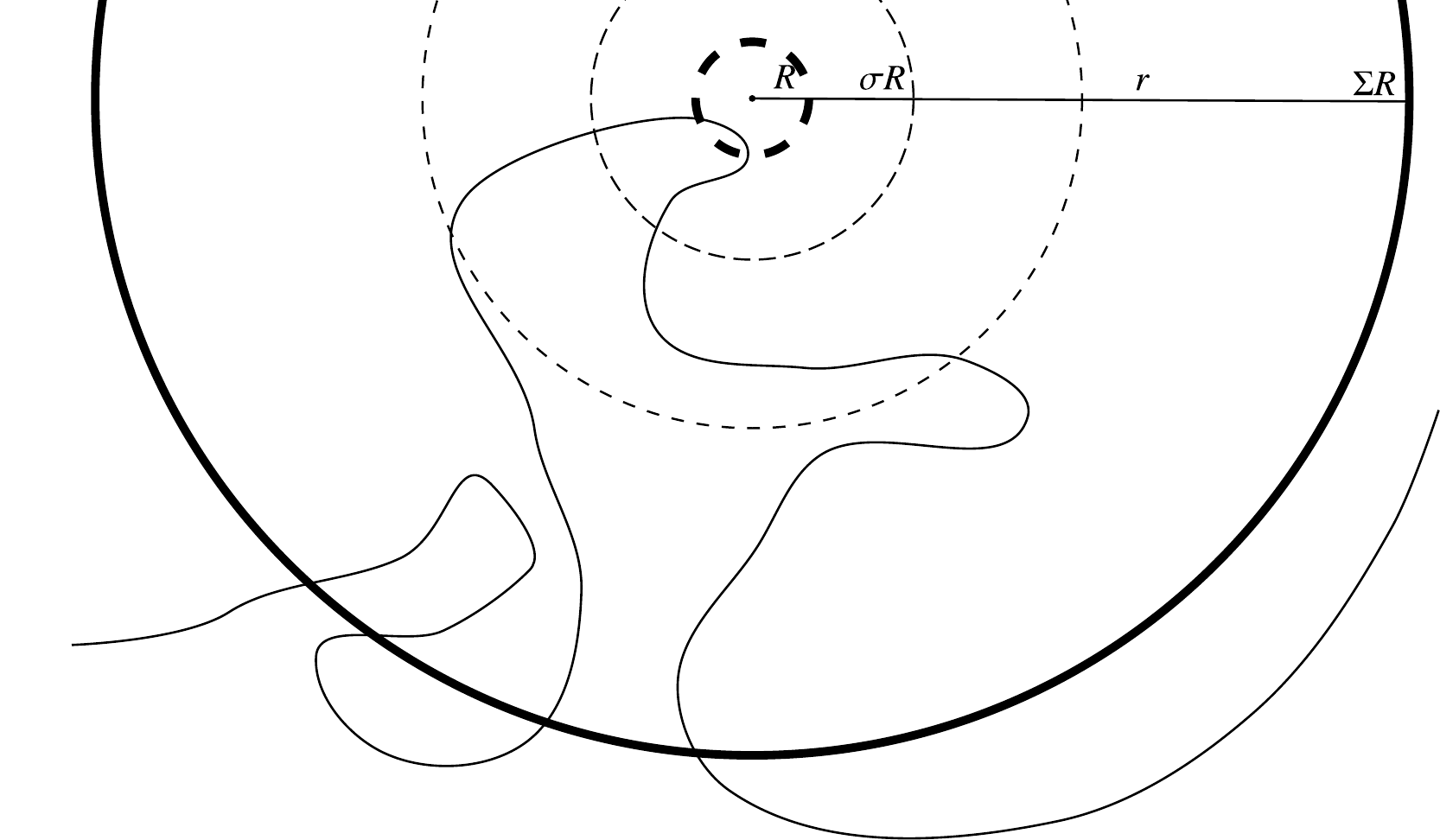}
   \includegraphics[scale=.45]{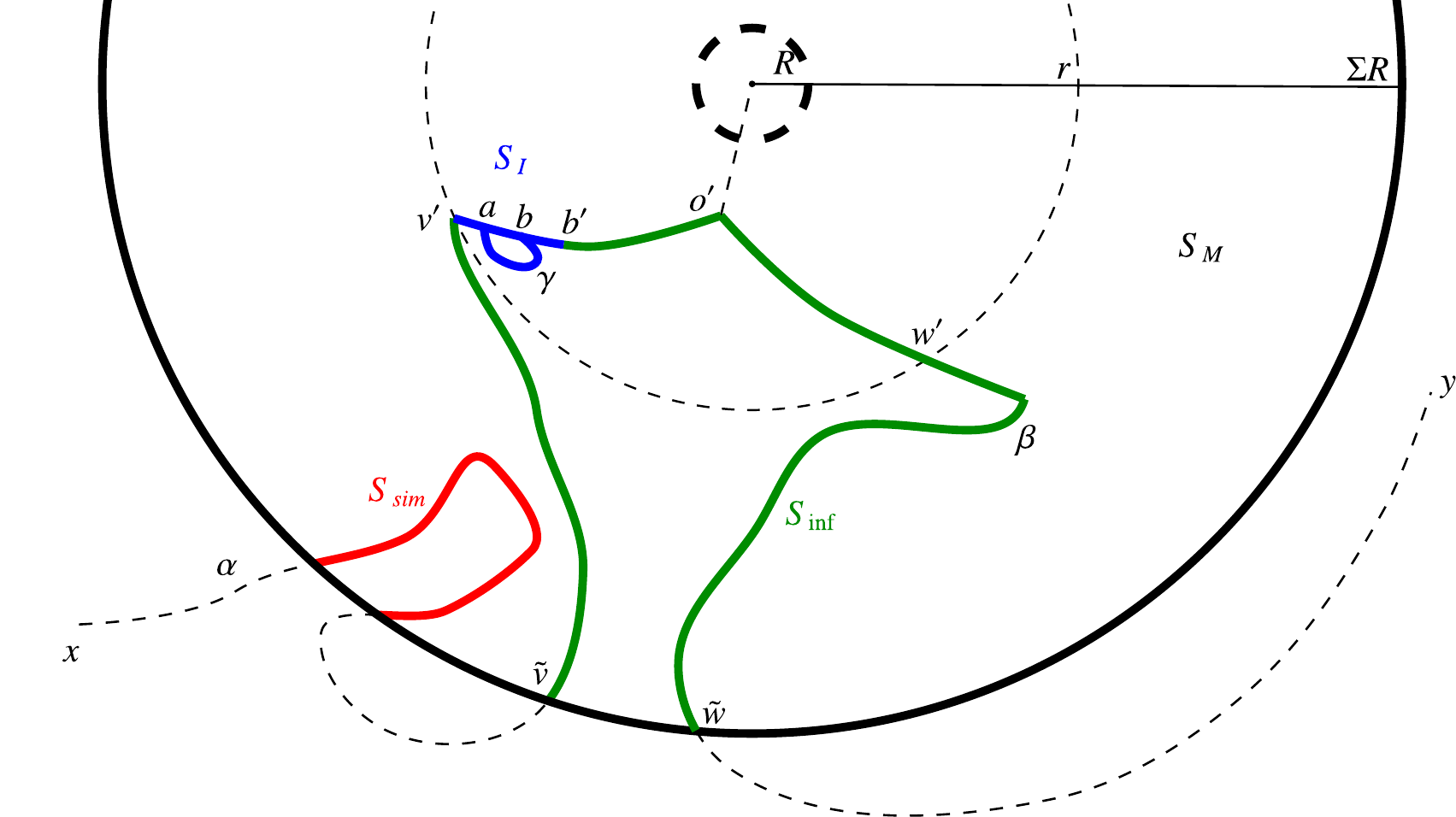}
   \caption{If the $T$-geodesic from $x$ to $y$ is as in the diagram on the left, the prescribed weights $E_w$
   might be given by the diagram on the right.}
   \label{fig:unboundedconstruction}
\end{figure}
%
   We now show that this choice of $A_i$ and $E_w$ satisfies \eqref{eq:resampleevents}.
   Set
   \[
      \pi' := \alpha * [v',o'] * [o',w'] * \beta.
   \]
   First we show that $T^*(\pi') < T(\pi)$.
   By construction we have $\pi'_{x,\tilde{v}} = \pi_{x,\tilde{v}}$ and $\pi'_{\tilde{w},y} = \pi_{\tilde{w},y}$.
   Moreover, each edge in either of those paths is also by construction either in $E(B(\Sigma R))^c$ or $S_{sim}$, and hence 
   when $w^* \in E_w$,
   \begin{align*}
      T^*(\pi'_{x,\tilde{v}} \sqcup \pi'_{\tilde{w},y}) &\le T(\pi'_{x,\tilde{v}} \sqcup \pi'_{\tilde{w},y}) + |E(B(\Sigma R))| \delta_{sim} \\
      &\le T(\pi_{x,\tilde{v}} \sqcup \pi_{\tilde{w},y}) + 1.
   \end{align*}
   Next, since $|\alpha_{\tilde{v},v'}| \le |\pi_{\tilde{v},v_r}|$, $|\beta_{w',\tilde{w}}| \le |\pi_{w_r,\tilde{w}}|$, since $A_2$ holds,
   and since $\alpha_{\tilde{v},v'}, \beta_{w',\tilde{w}} \subset S_{\inf}$, we have
   \[
      T^*(\alpha_{\tilde{v},v'} \sqcup \beta_{w',\tilde{w}}) \le 
      (\inf + \delta_{\inf})|\alpha_{\tilde{v},v'} \sqcup \beta_{w',\tilde{w}}|
      \le (\inf + \delta_{\inf})|\pi_{\tilde{v},v_r} \sqcup \pi_{w_r,\tilde{w}}|
      \le T(\pi_{\tilde{v},v_r} \sqcup \pi_{w_r,\tilde{w}}).
   \]
   Now, since $\pi'_{v',w'} \setminus [v',b'] \subset S_{\inf}$ and $[v',b'] \subset S_I$, we have
   \[
      T^*(\pi'_{v',w'}) \le (\inf + \delta_{\inf})2r + (\sup I_0)C(3+2\epsilon),
   \]
   while since $A_1$ holds and $\pi_{v_r,w_r}$ starts and ends at $S(r)$ ($r \ge \sigma R > 2R$) and visits $S(R)$, we have
   \[
      T(\pi_{v_r,w_r}) \ge 2(\inf + q)(r - R),
   \]
   so that 
   \begin{align} \label{eq:pathfaster}
      T(\pi_{v_r,w_r}) - T^*(\pi'_{v',w'}) &\ge 2R[ (q - \delta_{\inf})\frac{r}{R} - (\inf + q) ] - (\sup I_0)C(3+2\epsilon) \nonumber \\
      &\ge 2R[ (q - \delta_{\inf}) \sigma - (\inf + q) ] - (\sup I_0)C(3+2\epsilon),
   \end{align}
   For $R$ sufficiently large we have $\delta_{\inf} < q/2$ so that 
   \[
      (q - \delta_{\inf}) \sigma - (\inf + q) > (q/2) \sigma - (\inf + q) > 0,
   \]
   that is, the coefficient of $R$ \eqref{eq:pathfaster} is strictly positive.
   Altogether we have
   \[
      T(\pi) - T^*(\pi') \ge 2R[ (q / 2) \sigma - (\inf + q) ] - (\sup I_0)C(3+2\epsilon) - 1 \gessim R,
   \]
   so in particular $T^*(\pi') < T(\pi)$ for all sufficiently large $R$.
   
   From this, we can conclude that the $T^*$-geodesic $\pi^*$ must contain some edges in $S_{\inf} \cup S_I$.
   For suppose it did not;
   since $w^* \ge w$ on $(S_{\inf} \cup S_I)^c$, we would have
   \[
      T^*(\pi^*) \ge T(\pi^*) \ge T(\pi) > T^*(\pi'),
   \]
   contradicting $T^*$-geodesicity of $\pi^*$.
   
   Next, we know that $\pi^*$ contains no edge in $S_M$. For suppose that it did; then, since $A_3$ holds,
   \begin{align*}
      T^*(\pi^*) &\ge T^*(\pi^* \cap E(B(\Sigma R))^c) + M \\
                     &\ge T(\pi^* \cap E(B(\Sigma R))^c) + \sum_{e \in E(B(\Sigma R))} w(e) \\
                     &\ge T(\pi^*) \ge T(\pi) > T^*(\pi'),
   \end{align*}
   again contradicting $T^*$-geodesicity of $\pi^*$.
   
   Note that $S_{\inf} \cup S_I$ and $S_{sim}$ by construction share no vertices in common, and so we see
   that $\pi^*$, as a self-avoiding path which enters $S_{\inf}$, does not intersect $S_M$ and eventually exits
   $B(\Sigma R)$, must contain $\pi'_{\tilde{v},v'}$ and $\pi'_{b',\tilde{w}}$ or their reverses as a subpath.
   In particular, some subpath of $\pi^*$ has endpoints $v',b'$ and is restricted to $S_I = [v',b'] \cup \gamma$
   and hence by Proposition \ref{prop:detourtodetour} $S_I$ contains a feasible pair for $\pi^*$, and we are done
   showing that \eqref{eq:resampleevents} is satisfied.

To complete the proof that the $A_i^R, E_w$ satisfy 
the hypotheses of Proposition \ref{prop:conditionaltovdBK},
it remains to prove the ``resampling lemma'' 
relevant to $A_4$, to describe and prove the relevant properties of $A_5$,
and to give a lower bound on the conditional probability of $\{ w^* \in E_w\}$.
\begin{prop} \label{prop:resampling}
   For any fixed $\delta > 0$, we have
   \[
      \lim_{\eta \to 0} \nu( \{ p : p \mbox{ is } (\delta,\eta) \mbox{-resamplable}\}) = 1.
   \]
\end{prop}
\begin{proof}
   By continuity of measure we have that
   \[
      \lim_{\eta \to 0} \nu( \{p : \nu([p,p+\delta)) > \eta ) = \nu( \{p : \nu([p,p+\delta)) > 0),
   \]
   so it will suffice to show that
   \[
      \nu( \{ p : \nu([p,p+\delta)) = 0 \}) = 0.
   \]
   Set $N := \{ p : \nu([p,p+\delta)) = 0 \}$. We claim that there is a countable subset $X \subset N$ such that
   \[
      N \subset \bigcup_{p \in X} [p,p+\delta).
   \]
   Once we know this, the proposition follows, since then
   \[
      \nu(N) \le \nu \left( \bigcup_{p \in X} [p,p+\delta) \right) \le \sum_{p \in X} \nu([p,p+\delta)) = 0.
   \]
   To construct $X$, first set $X_0 := \emptyset$. 
   For each $i \ge 0$, consider $n_{i+1} := \inf N \setminus \left( \bigcup_{p \in X_i} [p,p+\delta) \right)$.
   If $n_{i+1} \in N$, then set $X_{i+1} := X_i \cup \{n_{i+1}\}$. Otherwise choose a (countable) sequence $S_{i+1}$
   of points of $N$ approaching $n_{i+1}$ and set $X_{i+1} := X_i \cup S_{i+1}$.
   It is simple to inductively check that each $X_i$ is countable and that $\bigcup_{p \in X_i} [p,p+\delta)$
   covers at least $N \cap [0,i\delta)$, so $X := \cup_{i=1}^{\infty} X_i$ is a countable subset of $N$
   with $\bigcup_{p \in X} [p,p+\delta) \supset N$, as desired.
\end{proof}

Now we describe the event $A_5$ and its properties.
The intuition is as follows: considering the $T$-geodesic $\pi: x \to y$, for each ball $B(r)$, if the first entrance of $\pi$ into that ball
is far from the last exit of $\pi$ from that ball, then we have ``enough room'' to do our construction, that is, we have
paths satisfying (1)-(5) above. So we want to bound the probability that, to the contrary, for all radii $r$, the first entry and last exit
are close. In fact, an even weaker event gives us ``enough room,'' and we bound the probability of the failure of this event
by showing that it would entail that $\pi$ is constrained to a ``narrow'' subgraph as it crosses $B(o_i, \Sigma R)$, making it unlikely
that the geodesic would enter so deep into $B(o_i, \Sigma R)$ before turning around.

For the formal construction of the event, first, given a pair of points $p,q \in S(\Sigma R)$, take edge-geodesics $[p,o_i]$ and $[q,o_i]$ from $p$ and $q$ respectively to $o_i$.
For each $\Sigma R \ge r \ge \sigma R$, let $p^r$ and $q^r$ be the unique elements of $V([p,o_i]) \cap S(r)$ 
and $V([q,o_i]) \cap S(r)$ respectively. We then define
\[
   S^r_0(p,q) := \left(B_{E(B(r))}(p^r,3K) \cup B_{E(B(r))}(q^r,3K)\right) \cap S(r),
\]
where $K := 4C(1 + \epsilon)$. Then for each $\ell \ge 0$ we define
\[
   S^r_{\ell}(p,q) := \left\{ z \in B(\Sigma R) \setminus B(r-1) : d_{E(B(\Sigma R) \setminus B(r-1))}(z,S^r_0(p,q)) = \ell \right\}.
\]
Lastly, for $\Sigma R - 3K \ge r \ge \sigma R$, set
\[
   S^r(p,q) := \bigsqcup_{\ell=0}^{3K} S^r_{\ell}(p,q)
\]
and define the event
\[
   C^r(p,q) := 
   \left \{ \begin{array}{c} \mbox{ there exist paths }\gamma_1,\gamma_2 \mbox{ in } S^r(p,q) \mbox{ such that } \\
                                      \mbox{ the endpoints } a_1,b_1 \mbox{ of } \gamma_1 \mbox{ lie in } S^r_{2K} \mbox{ and} \\
                                      |\gamma_1| \le K, \mbox{ one endpoint of } \gamma_2 \mbox{ lies in} \\
                                      S^r_{2K} \mbox{ and the other lies in } S^r_0, \mbox{ and } T(\gamma_2) \le T(\gamma_1)
             \end{array}
   \right\}.
\]
We now define the event $A_5$ by
\[
   A_5 := \bigcap_{\substack{p,q \in S(\Sigma R), \\ d_{B(\Sigma R)}(p,q) \le K}} 
                 \left( \bigcap_{r=\sigma R}^{\Sigma R - 3K} C^r(p,q) \right)^c,
\]
that is, $A_5$ is the event that for each pair $p,q$ of close points on $S(\Sigma R)$, $C^r(p,q)$ fails for at least some $r$.
Note that $A_5$ only depends on the weights of edges in $E(B(\Sigma R))$.
\begin{prop}
   There exists some constant $\rho < 1$ depending only on the degree $D$ of $G$, $\nu$, and $K$ such that 
   \[
      \Prob(C^r(p,q)) \le \rho
   \]
   for all $R,p,q,r$.
\end{prop}
\begin{proof}
   First note that, since each $S^r_0$ lies in the union of two balls of radius $3K$, $S^r_0$ contains at most $2(D+1)^{3K}$ vertices.
   Since the entirety of $S^r$ lies within distance $3K$ of $S^r_0$, we further have that
   \[
      |S^r| \le |S^r_0| (D+1)^{3K} \le 2(D+1)^{6K}.
   \]
   That is, we have a uniform bound on the possible number of vertices in $S^r$, and so it is not hard to see
   that the subgraph induced by $S^r(p,q)$ can only take on finitely many isomorphism types as all parameters except $D$ and $K$ vary.
   Hence, to show our claim, it suffices to show that for each fixed isomorphism type, $\Prob(C^r(p,q)) < 1$.
   (Here ``isomorphism type'' includes the relevant extra data of which subsets correspond to $S^r_0$ and $S^r_{2K}$,
   but even with this extra data it is easy to see that a bound on the number of vertices implies a bound on the number of 
   possible isomorphism types).
   
   To this end, fix an isomorphism type, and let $E'$ be the set of edges in $S^r$ which lie in some path in $S^r$ of length
   at most $K$ joining two vertices of $S^r_{2K}$. Since $\nu$ is assumed to have unbounded support (in particular it is not Dirac),
   there is some $a>0$ such that $\nu([0,a))>0$ and $\nu([a,\infty)) > 0$. Then the event
   \[
      \{ w(e) < a \mbox{ for all } e \in E', w(e) \ge a \mbox{ for all } e \notin E' \}
   \]
   has nonzero probability. Moreover, this event entails the failure of $C^r(p,q)$, since on it all candidates for $\gamma_1$
   necessarily have edges in $E'$ and hence have $T(\gamma_1) < aK$, while all candidates for $\gamma_2$
   must have at least $K$ edges lying in $E'^c$, and hence $T(\gamma_2) \ge aK > T(\gamma_1)$.
\end{proof}
\begin{prop}
   $\Prob(A_5) \tendsto{R}{\infty} 1$.
\end{prop}
\begin{proof}
   For each fixed $p,q$, note that whenever $S \subset [\sigma R, \Sigma R - 3K] \cap \Z$ is such that each element
   has distance at least $3K$ from every other element, the subgraphs
   $\{ S^r(p,q) : r \in S \}$ are all disjoint and hence the events $\{ C^r(p,q) : r \in S \}$ are all independent.
   Since $K, \sigma, \Sigma$ are constants fixed independent of $R$, it is easy to see that there is some $c_3 >0$ such that
   for all large $R$ we can pick such an $S$ with $|S| \ge c_3 R$, and so
   \begin{align*}
      \Prob \left( \bigcap_{r=\sigma R}^{\Sigma R - 3K} C^r(p,q) \right) &\le \Prob \left( \bigcap_{r \in S} C^r(p,q) \right) \\
      &= \prod_{r \in S} \Prob( C^r(p,q) ) \le \rho^{c_3 R},
   \end{align*}
   where $\rho < 1$ is provided by the previous proposition. But then we have
   \begin{align*}
       \Prob(A_4^c) \le \sum_{p,q \in S(\Sigma R)} \Prob \left( \bigcap_{r=\sigma R}^{\Sigma R - 3K} C^r(p,q) \right)
       \le (C_1 (\Sigma R)^d)^2 \rho^{c_3 R} \tendsto{R}{\infty} 0,
   \end{align*}
   as desired.
\end{proof}

Now we prove the key property of $A_5$.
\begin{lemma} \label{lem:a5}
    On the event $A_5 \cap \{ x,y \notin B(\Sigma R), \pi \mbox{ visits } B_i \}$, 
    for some $\Sigma R \ge r \ge \sigma R$, there exist paths $\alpha$ and $\beta$ satisfying conditions 1 through 5 above.
\end{lemma}
\begin{proof}
   Denote by $v$ and $w$ respectively the first and last vertices of $\pi$ which lie in $B(\Sigma R)$.
   Now, for each $\Sigma R \ge r \ge \sigma R$, define
   $v_r \in S(r)$ to be the vertex of $\pi$ immediately preceding the first vertex of $\pi$ lying in $B(r-1)$,
   and define $w_r \in S(r)$ to be the vertex of $\pi$ immediately following the last vertex of $\pi$ lying in $B(r-1)$.
   All of these are well defined, since $\pi$ starts and ends outside of $B(\Sigma R)$ and visits $B_i \subset B(R) \subset B(\sigma R)$.
   Then define $\alpha_r := \pi_{x,v_r}$, $\beta_r := \pi_{w_r,y}$.
   If for some $r$, $d_{E(B(r))}(v_r,w_r) > K$, then we can just take $\alpha = \alpha_r$ and $\beta = \beta_r$ and we are done.
   So from here on assume that $d_{E(B(r))}(v_r,w_r) \le K$ for all $\Sigma R \ge r \ge \sigma R$.
   
   Next, for each $r$ we define the set
   \[
      \tilde{S}^r := 
      \left( \bigcup_{p \in V(\alpha_r \cup \beta_r) \cap B(\Sigma R)} 
              \bigcup_{\substack{\gamma: p \to o_i \\ \mbox{edge geodesic}}} V(\gamma) \right) \cap S(r).
   \]
   Suppose that for some $\Sigma R \ge r \ge \sigma R$, there is some $z \in \tilde{S}^r$ 
   with $d_{E(B(r))}(z,v_r), d_{E(B(r))}(z,w_r) > K$. Then we can construct $\alpha$ and $\beta$ as follows.
   $z$ by definition lies on some edge-geodesic $\gamma$ from some point $p \in V(\alpha_r \cup \beta_r)$ to $o_i$.
   Consider the last vertex of $V(\gamma) \cap (V(\alpha_r \cup \beta_r))$ (that is, the nearest vertex to $o_i$), and call it $z'$.
   If $z' \in V(\alpha_r)$, set $\alpha := (\alpha_r)_{x,z'} * \gamma_{z',z}$ and $\beta := \beta_r$.
   If $z' \in V(\beta_r)$, set $\beta := \overline{\gamma}_{z,z'} * (\beta_r)_{z',y}$ and $\alpha := \alpha_r$ (here an overline denotes the reverse of a path).
   In either case, $\alpha$ and $\beta$ give disjoint self-avoiding paths because the original paths were disjoint and self-avoiding
   and because by construction $V(\gamma_{z',z})$ only intersects $V(\alpha_r \cup \beta_r)$ at $z'$.
   Conditions (1)-(3) are satisfied by choice of $z$, (4) is satisfied because $\alpha$ and $\overline{\beta}$ agree with $\alpha_r$ and 
   $\overline{\beta_r}$ until one of them reaches $z'$, 
   and from that point the path follows $\gamma$; in particular, it stays inside $B(\Sigma R)$ 
   until it reaches its endpoint. For (5), note that, since $\gamma$ is an edge-geodesic from $z'$ to $o_i$,
   \[
      |\gamma_{z',z}| = d(z',o_i) - d(z,o_i) = d(z',o_i) - r.
   \]
   Since $(\alpha_r)_{z', v_r}$ (or $(\overline{\beta}_r)_{z',w_r}$, if $z' \in V(\beta_r)$) is a path from $z'$ to $S(r)$,
   it must have length at least $d(z',o_i) - r$ by the triangle inequality, and so we get (5).
   
   Lastly, we show that, if both of the above conditions fail, i.e. for all $\Sigma R \ge r \ge \sigma R$ we have
   \begin{equation} \label{eq:closeentryexit}
      d_{E(B(r))}(v_r,w_r) \le K
   \end{equation}
    and
   \begin{equation} \label{eq:tightspace}
      \tilde{S}^r \subset \left(B_{E(B(r))}(v_r, K) \cup B_{E(B(r))}(w_r,K) \right) \cap S(r),
   \end{equation}
   then the event $A_5$ fails.
   
   For this, first note that, since every $V(\alpha_r \cup \beta_r) \cap B(\Sigma R)$ contains the entry and exit points $v$ and $w$,
   every $\tilde{S}^r$ contains $v^r, w^r$ (in the notation used in defining the set $S^r(p,q)$ in the case $(p,q) = (v,w)$).
   Then \eqref{eq:tightspace} implies that $v^r$ and $w^r$ are each distance at most $K$ from either $v_r$ or $w_r$.
   A general element $z \in \tilde{S}^r$ has the same property, and combining with \eqref{eq:closeentryexit} gives
   \[
      d(z,v^r) \le \min(d(z,v_r),d(z,w_r)) + d(v_r,w_r) + \min(d(v_r,v^r),d(w_r,v^r)) \le 3K,
   \]
   and similarly $d(z,w^r) \le 3K$. Hence $\tilde{S}^r \subset S^r_0(v,w)$. Moreover we have
   \begin{claim}
      If $\Sigma R \ge r+\ell, r \ge \sigma R$, then $v_{r+\ell},w_{r+\ell} \in S^r_{\ell}(v,w)$.
   \end{claim}
   \begin{proof}[Proof of Claim]
      Since $v_{r+\ell},w_{r+\ell} \in S(r+\ell)$ and $S^r_0(v,w) \subset S(r)$, we have that
      \[
         d_{E(B(\Sigma R) \setminus B(r-1))}(v_{r+\ell}, S^r_0(v,w)), d_{E(B(\Sigma R) \setminus B(r-1))}(w_r, S^r_0(v,w)) \ge \ell,
      \]
      so we only have to show the opposite inequality. For this, let $(v_{r+\ell})^r$ be as usual the intersection of $S(r)$
      with an edge geodesic from $v_{r + \ell}$ to $o_i$. Since $v_{r+\ell} \in V(\alpha_{r+\ell}) \subset V(\alpha_r)$,
      we have that $(v_{r + \ell})^r \in \tilde{S}^r \subset S^r_0(v,w)$; moreover, the geodesic from $v_{r+\ell}$ to
      $(v_{r + \ell})^r$ is a path of length $\ell$ which lies in $B(\Sigma R) \setminus B(r-1)$, and so we have
      \[
         d_{B(\Sigma R) \setminus B(r-1)}(v_{r+\ell}, S^r_0(v,w)) \le d_{B(\Sigma R) \setminus B(r-1)}(v_{r+\ell}, (v_{r+\ell})^r) = l,
      \]
      as desired. The argument for $w_{r+\ell}$ is the same.
   \end{proof}
   Finally, we contradict $A_5$. For each $\Sigma R - 3K \ge r \ge \sigma R$, consider $\gamma_3 := \pi_{v_{r+2K},v_r}$.
   Since $\gamma_3$ by construction does not visit $B(r-1)$, and since it starts at a point with distance
   $d_{E(B(\Sigma R) \setminus B(r-1))}(v_{r+2K}, S^r_0(v,w)) = 2K$ and ends at a point $v_r \in S^r_0(v,w)$,
   some subpath $\gamma_2$ of $\gamma_3$ is contained in $S^r(v,w)$, starts at $S^r_{2K}(v,w)$ and ends
   at $S^r_0(v,w)$.
   On the other hand, let $\gamma_1$ be an edge-geodesic from $v_{r+2K}$ to $w_{r+2K}$. By assumption,
   $d(v_{r+2K},w_{r+2K}) \le K$, so $|\gamma_1| \le K$; therefore $\gamma_1$ does not intersect $B(r-1)$,
   and since the endpoints of $\gamma_1$ lie in $S^r_{2K}(v,w)$, $\gamma_1$ is totally contained in $S^r(v,w)$.
   But since $\pi$ is a $T$-geodesic, we have
   \[
      T(\gamma_1) \ge T(\pi_{v_{r+2K},w_{r+2K}}) \ge T(\gamma_3) \ge T(\gamma_2),
   \]
   and so $C^r(v,w)$ holds. But then $C^r(v,w)$ holds for all $\Sigma R - 3K \ge r \ge \sigma R$, so $A_5$ fails.
\end{proof}

To apply Proposition \ref{prop:conditionaltovdBK}
it only remains to obtain a lower bound on $\Prob( w^* \in E_w | w)$ on the event $A_i^R$ which is
independent of $o_i$.
But on $A_i^R$ we have
\begin{align*}
   \Prob(w^* \in E_w | w) 
   &\ge \nu(I_0)^{|S_I|} \nu([\inf, \inf + \delta_{\inf}])^{|S_{\inf}|} \eta^{|S_{sim}|} \nu([M,\infty))^{|S_M|} \\
   &\ge \min( \nu(I_0), \nu([\inf,\inf+\delta_{\inf}]), \eta, \nu([M,\infty)))^{D C_1 (\Sigma R)^d} > 0,
\end{align*}
as desired.

\subsection{Proof of Theorem \ref{thm:polygrowthvdBK}}

Let $G$ be a graph of strict polynomial growth which admits detours. 
Let $\nu$ be an exponential-subcritical measure with finite mean, and let $\tilde{\nu}$ be a measure which has finite mean
and is strictly more variable than $\nu$.
First assume \eqref{eq:extraassumption}. Then
let $\epsilon > 0, I_0, y_0$ be as in Lemma \ref{lem:technicallemma}.
In case $\nu$ has bounded support, construct $B_i, B(o_i, \Sigma R), A_i^R,$ and $E_w$ as in Section \ref{bddsupp}.
In case $\nu$ has unbounded support, construct $B_i, B(o_i, \Sigma R), A_i^R,$ and $E_w$ as in Section \ref{unbddsupp}.
In their respective sections, we prove that both constructions satisfy the hypotheses of Proposition \ref{prop:conditionaltovdBK},
and so 
\[
   \liminf_{d(x,y) \to \infty} \frac{ \E T(x,y) - \E \tilde{T}(x,y) }{d(x,y)} > 0.
\]
Now, if $w, \tilde{w}$ do not satisfy \eqref{eq:extraassumption}, 
take $\bar{w}$ as in Lemma \ref{lem:wlog}. Then we have
\[
   \liminf_{d(x,y) \to \infty} \frac{ \E T(x,y) - \E \tilde{T}(x,y) }{d(x,y)} \ge 
   \liminf_{d(x,y) \to \infty} \frac{ \E T(x,y) - \E \bar{T}(x,y) }{d(x,y)}  
   > 0.
\]
Thus $G$ is vdBK. The reverse implication is given by Theorem \ref{thm:notvdBK}.

\subsection{Non-homogeneous graphs of polynomial growth}
Theorem \ref{thm:polygrowthvdBK} does not require almost-transitivity (although almost-transitivity does give us more information
about the condition of exponential-subcriticality, namely that $\underline{p_c} = p_c$ \cite{DuminilCopinTassion}).
This means that the theorem applies to a very broad class of graphs, but it can be difficult to produce examples of non-transitive
graphs which have \emph{strict} polynomial growth and for which it is easy to check whether the graph admits detours.
Here we give two examples (or one example and one counterexample).

First, the theorem can be applied to a broad range of subgraphs of the standard Cayley graph of $\Z^d$.
For instance $G := \Z_{\ge 0}^{d_1} \times \Z^{d_2} \subset \Z^{d_1 + d_2}$ will be vdBK whenever $d_1 + d_2 \ge 2$.
These graphs have growth bounds $B_G(R) \le B_{\Z^{d_1+d_2}}(R) \le 2^{-d_1} B_G(R)$, from which we can deduce
that $G$ has strict polynomial growth. Moreover, the unique geodesics in $G$ are all also unique geodesics in $\Z^d$,
(that is, they are represented by words of the form $e_i^k$, where $\{e_i\}$ is the standard generating set),
and when $d_1 + d_2 \ge 2$ one can easily see that these admit detours.

Moreover, we can apply the theorem to ``sectors'', that is, graphs $G_{\theta,\theta'}$ induced by the vertex subset
\[
   V_{\theta,\theta'} := \{ (x,y) \in \Z^2 : \theta \le \arctan(y/x) \le \theta' \}
\]
for fixed $\theta < \theta'$. Again we see that this is of strict polynomial growth.
Moreover, the unique geodesics in this graph are either already unique geodesics in $\Z^2$, or
they run along the ``boundary'' $\{ (x,y) : \arctan(y/x) \approx \theta \mbox{ or } \theta' \}$ (in fact most geodesics along
the boundary are also not unique). But again it is simple to check that these admit detours, and hence $G$ is vdBK.
Similar constructions can be done in higher dimensions, and in fact many more subgraphs of $\Z^d$ satisfy the
hypotheses of the theorem.

The next obvious candidate for a non-almost-transitive graph of strict polynomial growth is the infinite cluster
of a supercritical Bernoulli percolation on a Cayley graph of a virtually nilpotent group.
In fact, this will not have strict polynomial growth as we have defined here, since we require \emph{uniform}
volume lower bounds. But beyond that, one can see that (for $p < 1$) almost surely the cluster does \emph{not}
admit detours, and hence by Theorem \ref{thm:notvdBK} it is \emph{not} vdBK.

This can be seen by a simple ``finite energy'' type argument. For any $C < \infty$, choose a large radius $R \ge C$ such that
the probability that $B(R)$ intersects the infinite cluster is positive; this event is actually independent of the configuration of edges
inside $E(B(R))$, so chose a particular configuration in $E(B(R))$ such that all edges in contact with the vertex boundary of $B(R)$ are open,
such that all these edges are connected to each other by open edges, and such that these open edges on the boundary are connected
to an open path of length $\ge 3C$ which is otherwise surrounded by closed edges.
The probability that the boundary of $B(R)$ is connected to infinity \emph{and} that the restriction of the sampled configuration
restricted to $E(B(R))$ is our prescribed configuration, is also positive. One quickly sees that on this event, the infinite cluster
contains a self-avoiding path of length $C$ which does not admit a detour of length at most, say, $(3/2)C$.
The event that the infinite cluster contains a such a path is clearly a translation-invariant event, and so by ergodicity,
since this event occurs with positive probability, it occurs with probability $1$.
Intersecting all these events for a countable collection $C \to \infty$ shows that almost surely the infinite cluster does not admit detours.

Of course, for graphs which are not almost-transitive, the vdBK condition is quite strong. One may ask the following question
(which is equivalent in the case of almost-transitive graphs):
fix $o \in V$. If $\tilde{\nu}$ is strictly more variable than $\nu$ and $\nu$ is exponential-subcritical,
is
\[
   \liminf_{x \to \infty} \frac{ \E T(o,x) - \E \tilde{T}(o,x)}{d(o,x)} > 0?
\]
It is conceivable that the answer might be ``yes'' in the case of supercritical percolation clusters on nilpotent Cayley graphs,
since supercritical clusters ``generally behave like their underlying graph'' at large scales.
Perhaps the proofs in this paper could be adapted to this case, but it would require ``large scale'' and perhaps ``statistical'' weakenings
of the geometric properties used.

\section{Absolute continuity with respect to the expected empirical measure} \label{sec:abscont}
For a graph $G$ and a probability measure $\nu$ on $[0, \infty)$, we say that the associated first passage percolation 
\emph{has weight distribution absolutely continuous with respect to the expected empirical measure} if for any Borel set $A \subset [0, \infty)$ with $\nu(A) > 0$ we have
\[
   \liminf_{d(x,y) \to \infty} \frac{ \E[ \sum_{e \in \pi} \ind_{w(e) \in A} ] }{ d(x, y) } > 0,
\]
where $\pi$ denotes the $T$-geodesic from $x$ to $y$.
Note that this does not imply or presuppose that a literal expected empirical measure, that is, a weak limit of the expected empirical measures 
$\frac{1}{d(x,y)} \E \sum_{e \in \pi} \delta_{w(e)}$, exists,\footnote{In the $G=\Z^d$ case it was recently proven by Bates ~\cite{Bates} that for ``generic''
$\nu$, the sequence of random empirical measures $\frac{1}{d(0,nv)} \sum_{e \in \pi} \delta_{w(e)}$ in a fixed direction almost surely
weakly converges to a deterministic limit measure, an even stronger result than the existence of an \emph{expected} empirical measure in a particular direction.}
although it does imply that $\nu$ is absolutely continuous
with respect to any subsequential weak limit of this collection of measures. 
As noted in \cite{vdBK}, the above property implies strict monotonicity with respect to stochastic domination:

\begin{prop} \label{prop:stochdommonotonicity}
   Suppose that $\nu$ is absolutely continuous with respect to the expected empirical measure. Then whenever $\nu$ strictly
   stochastically dominates $\tilde{\nu}$, that is, whenever $\tilde{\nu} \ne \nu$ and there exists some coupling $(\tilde{w},w)$
   of $\tilde{\nu}$ and $\nu$ such that $\tilde{w} \le w$ almost surely, we have $\E \tilde{T} \ll \E T$.
\end{prop}
\begin{proof}
   Fix a coupling $(\tilde{w}, w)$ with $\tilde{w} \le w$; since $\tilde{\nu} \ne \nu$,
   $\Prob(\tilde{w} < w) > 0$, and so one can find sufficiently small $a > 0, b > 0,$ and Borel set $A \subset [0, \infty)$ such that $\nu(A) > 0$
   and such that for every $y \in A$,
   \[ \Prob( \tilde{w} < y - a | w = y) \ge b. \]
   We thus have
   \begin{align*}
      \E T(x,y) - \E \tilde{T}(x,y) &\ge \E[ T(\pi) - \tilde{T}(\pi)] \\
                                               &= \E \sum_{e \in \pi} (w(e) - \tilde{w}(e)) \\
                                               &= \E\left[ \sum_{e \in \pi} \E[ w(e) - \tilde{w}(e) | w(e) ] \right] \\
                                               &\ge \E \left[ \sum_{e \in \pi} ab \ind_{w(e) \in A} \right] \\
                                               &\gessim d(x,y),
   \end{align*}
   where the last inequality follows from the fact that $\nu$ is absolutely continuous with respect to the expected empirical measure.
\end{proof}
If $\nu$ strictly stochastically dominates $\tilde{\nu}$, then $\tilde{\nu}$ is strictly more variable than $\nu$,
so our theorems above already prove strict monotonicity with respect to stochastic domination for 
graphs which admit detours and are either quasi-trees or have strict polynomial growth (in the later case, on the
condition that $\nu$ is also exponential-subcritical).
However, we can prove absolute continuity with respect to the empirical measure---and hence strict monotonicity with respect to
stochastic domination---whether or not $G$ admits detours directly, by using essentially
identical methods to those above. 

\qitreeabscont*
\begin{proof}[Proof sketch]
   Let $A \subset [0, \infty)$ be Borel with $\nu(A) > 0$. Set $I_0 = A$. Set $C=1, \epsilon = 1$. Then do the same construction as in the proof of
   Theorem \ref{thm:qitree}. Whereas Theorem \ref{thm:qitree} gives $\gessim d(x,y)$ detours in expectation, now this construction
   gives $\gessim d(x,y)$ edges of $\pi$ which have weights in $I_0 = A$, as desired.
   More explicitly, the construction gives a family of subgraphs 
   $\{ B_i \}$ with 
   \[
      \sum_{i} \Prob(\mbox{the geodesic } \pi:x \to y \mbox{ contains an edge } e \mbox{ in } B_i \mbox{ of weight } w(e) \in A) \gessim d(x,y).
   \]
   Then arguing similarly as in the proof of Lemma \ref{lem:feasibletovdBK}, one gets a \emph{disjoint} family $\{ B_i \}$ with this property,
   and so one concludes that $\pi$ contains $\gessim d(x,y)$ edges with weight lying in $A$ in expectation, as desired.
   Lastly, the stochastic domination statement follows from Proposition \ref{prop:stochdommonotonicity}.
\end{proof}

\polygrowthabscont*
\begin{proof}[Proof sketch]
   Let $A \subset [0, \infty)$ be Borel with $\nu(A) > 0$ (one may without loss of generality replace $A$ with a bounded positive $\nu$-measure subset). 
   Assume $\nu$ is exponential-subcritical,
   and choose $q, \Sigma,$ and $\delta$ as in the first and simplest construction in the proof of
   Theorem \ref{thm:polygrowthvdBK}, that is, in the case that $\nu$ has bounded support and $y_0 = \sup$.
   Also define the event $A_1$ as in that construction.
   Then define $E_w = E$ by
   \[
      E := \left\{ \omega \in [0,\infty)^{E(B(\Sigma R))} : \omega(e) \in [\inf, \inf + \delta) \mbox{ if } e \in E(B(\Sigma R - 1)), 
                                                                                     \omega(e) \in A \mbox{ otherwise} \right\}.
   \]

   Using this construction, we have that for sufficiently large $R$, whenever $A_1$ holds, 
   the $T$-geodesic $\pi$ crosses $B_i$, and $w^* \in E_w = E$,
   there is a path entering $B(o_i, \Sigma R)$ which has $T^*$-weight strictly smaller than any path not entering $B(o_i, \Sigma R)$;
   hence the $T^*$ geodesic enters $B(o_i, \Sigma R)$ and therefore contains
   an edge with weight valued in $A$. 
   (Note that this is much simpler than in the proof of Theorem \ref{thm:polygrowthvdBK} when $y_0 \ne \sup$; this is because
   in that proof, we had to ensure that the $T^*$-geodesic made a long excursion away from the boundary of $B(o_i, \Sigma R)$,
   whereas here we only need it to hit a single edge with weight in $A$).
   
   The Peierls lemma (Lemma \ref{lem:peierls}) gives, in expectation, $\gessim d(x,y)$ $B_i$
   such that the $T$-geodesic visits $B_i$ and $A_1$ holds; combining this with resampling 
   (similar to Proposition \ref{prop:conditionaltovdBK}) thus gives in expectation we at least $\gessim d(x,y)$ $B(o_i, \Sigma R)$  which contain an
   edge of the $T$-geodesic with weight lying in $A$. Again arguing similarly as in the proof of Lemma \ref{lem:feasibletovdBK} then
   gives that, in expectation, the $T$-geodesic contains $\gessim d(x,y)$ edges with weight lying in $A$, as desired.
   The stochastic domination statement again follows from Proposition \ref{prop:stochdommonotonicity}.
\end{proof}

\subsection*{Acknowledgements}
The author thanks Antonio Auffinger for suggesting the problem of generalizing \cite{vdBK}, as well as for helpful conversations.
The author also thanks Alice Kerr, from whom he learned Manning's bottleneck criterion.

\appendix
\appendixpage

\section{Cayley graphs which admit detours} \label{app:grouptheory}
\subsection{Paths in Cayley graphs}
Let $\Gamma$ be a finitely generated group, $S$ a finite set, and $f:S \to \Gamma \setminus \{1\}$ a map whose image generates $\Gamma$.
We define the \emph{unreduced Cayley graph} associated to $(\Gamma, S)$ to be the graph $G = (V,E)$ with vertex set $V = \Gamma$
and edge set $E = \Gamma \times S$, where the boundary of the edge $e = (g,s)$ is $\{g,gf(s)\} \subset V$.
We define the \emph{reduced Cayley graph} associated to $(\Gamma, S)$ to be the graph $G = (V,E)$ whose vertex set
is $V = \Gamma$, whose edge set is the set 
$E = \{ \{v,w\} \subset \Gamma : v \ne w, \exists s \in S \cup S^{-1} \mbox{ s.t. } w = vf(s)\}$
and the boundary map is the natural inclusion. The reduced Cayley graph is the simple graph obtained from the unreduced Cayley graph
by deleting parallel edges. Typically the two graphs are isomorphic unless some $f(s) \in \Gamma$ is of order two, but 
the two graphs may also be nonisomorphic if one takes $f:S \to \Gamma$ to be a ``redundant'' generating set, containing repeated elements
or inverses of elements already included.
We call a graph a \emph{Cayley graph} it is a reduced or unreduced Cayley graph.

Note that the action of $\Gamma$ on itself by left multiplication induces an action of $\Gamma$ on $G$ by graph isomorphisms
whenever $G$ is a Cayley graph. So $\Gamma$ acts on the set of vertices of $G$, the set of edges of $G$,
the set of geodesics of $G$, the set of unique geodesics of $G$, etc.
The action of $\Gamma$ on the vertex set $V = \Gamma$ is transitive, and so any path and in particular any unique geodesic
is the translate of some path (unique geodesic) starting at $1 \in \Gamma = V$.

Finite paths in the unreduced Cayley graph associated to $(\Gamma, S)$ starting from the identity vertex are in one-to-one correspondence
with finite words in $S \sqcup S^{-1}$. (Here, $S$ is considered as an abstract set, and $S^{-1}$ consists of symbols of the form $s^{-1}$
where $s \in S$).  
In the case of a \emph{reduced} Cayley graph, there is a bijection between finite paths starting at the identity and elements of $(f(S) \cup f(S)^{-1})^*$
(i.e. finite words in the subset $f(S) \cup f(S)^{-1} \subset \Gamma$).
We use $A$ to denote $S \sqcup S^{-1}$ in the case that $G$ is the unreduced Cayley graph associated to $(\Gamma, S)$ and 
we use $A$ to denote $f(S) \cup f(S)^{-1} \subset \Gamma$ in the case that $G$ is the reduced Cayley graph associated to $(\Gamma, S)$.
So in either case, paths starting from $1 \in V = \Gamma$ correspond to finite words in $A$; we denote the set of
finite words in $A$ by $A^*$. 

$A^*$ together with the operation of concatenation is the free monoid on $A$.
Since $A$ has a natural ``formal inverse'' map $A \to A, a \mapsto a^{-1}$ in both cases, this induces a ``formal inverse'' map on
$A^*$ given by $(a_1 \cdots a_n)^{-1} := a_n^{-1} \cdots a_1^{-1}$.
Note that $\alpha^{-1}$ is not a true inverse of $\alpha$; the word $\alpha \alpha^{-1} \in A^*$
represents the concatenation of the path represented by $\alpha$ with its reverse, which is in particular not equal to the trivial path
(consisting of no edges) represented by the empty word in $A^*$.

We have an ``evaluation map'' $\rho: (S \sqcup S^{-1})^* \to \Gamma$
induced by $f:S \to \Gamma$ given by
\[
   \rho(s_1^{\epsilon_1} \cdots s_k^{\epsilon_k}) := f(s_1)^{\epsilon_1} \cdots f(s_k)^{\epsilon_k},
\]
where the $s_i \in S$, $\epsilon_i \in \{+1,-1\}$. 
We also have an evaluation map $\rho: (f(S) \cup f(S)^{-1})^* \to \Gamma$ just given by group multiplication in $\Gamma$;
$\rho(s_1 \cdots s_k)$ is the product $s_1 * \cdots * s_k \in \Gamma$. (Typically the group multiplication in $\Gamma$ is denoted in the same 
way as concatenation in $A^*$; we only write it with $*$ here to emphasize that $\rho : A^* \to \Gamma$ is not the identity map).

Note that in either case $\rho: A^* \to \Gamma$ is a homomorphism of monoids which respects the inverse operation,
i.e. $\rho(\alpha \beta) = \rho(\alpha) \rho (\beta)$, 
$\rho(\alpha^{-1}) = \rho(\alpha)^{-1}$ for all $\alpha, \beta \in A^*$. 
Geometrically, if $\alpha \in A^*$ represents a path
in $G$ starting from $1 \in V = \Gamma$, $\rho(\alpha) \in \Gamma = V$ is the endpoint of that path.
If $\alpha, \beta \in A^*$, then the path represented by $\alpha \beta$ is the concatenation of the path represented by $\alpha$
with the left-translate by $\rho(\alpha)$ of the path represented by $\beta$.
The path represented by $\alpha^{-1}$ is the left-translate by $\rho(\alpha)^{-1}$ of the reverse of the path represented by $\alpha$.

The condition that $\pi \in A^*$ corresponds to a geodesic in $G$ is exactly the condition that
for any $\pi' \in A^*$ such that $\rho(\pi')=\rho(\pi)$ we have $|\pi'| \ge |\pi|$ (where here $|\cdot|$ is the length of a word). 
The condition that $\pi \in A^*$ corresponds to a unique geodesic is precisely the condition that $\pi$ corresponds to a geodesic,
and for any $\pi'$ such that $\rho(\pi') = \rho(\pi)$ and $|\pi'|=|\pi|$ we have $\pi = \pi'$.
Recall also that the properties of being geodesic and uniquely geodesic pass to subpaths and are invariant under translations, 
and so the above properties pass to subwords.
In what follows we will use the same symbol to denote a word in $A^*$ and the path in $G$ starting from the identity which it corresponds to.
For instance, for $\alpha, \beta \in A^*$, $\rho(\alpha) \beta$ is the left-translate by $\rho(\alpha) \in \Gamma$ of the path in $G$
represented by $\beta$; $\rho(\alpha) \beta$ is a path in $G$ starting at $\rho(\alpha)$ and ending at $\rho(\alpha \beta)$, and 
the path $\alpha \beta$
is the concatenation of the paths $\alpha$ and $\rho(\alpha) \beta$.

\subsection{Sufficient conditions for a Cayley graph to admit detours}
In many groups, there are relatively explicit constructions for constructing an $\epsilon$-detour for a given unique geodesic.
All the constructions given in this section are very similar, and come from intuition from $\Z^d$, $d \ge 2$.
In the standard Cayley graph of $\Z^d$, unique geodesics are just long straight lines; one can construct a path with the same
endpoints which is totally edge-disjoint from the original and only two edges longer by simply first taking a step perpendicular to
the original path, following a translated version of the original path, and then taking a step back.
Our basic strategy will be: given a unique geodesic with some nice properties,
find two words of bounded length which when appended to the beginning and end of this path give
a path which has the same endpoints but misses a positive proportion of the edges in the original path; this
constructs detours for ``nice'' unique geodesics.
Then, if every unique geodesic contains a subpath with nice properties with length at least a positive proportion of the original length,
we can now construct a detour for a general unique geodesic simply by replacing the nice subpath with its detour.

The first large class of groups for which we prove that all Cayley graphs admit detours is the following:
\finitenormalsubgroup
\begin{proof}
   Set $\ell := \max_{f \in F} |f|$ (where $|f|$ is the minimum length of a word $w \in A^*$ with $\rho(w) = f$).
   Let $\pi \in A^*$ be a uniquely geodesic word. We first claim that there exists a subword $\pi'$ of $\pi$ with $|\pi'| \ge \frac{ |\pi| - \ell}{\ell + 1}$
   such that no nonempty subword $v$ of $\pi'$ has $\rho(v) \in F$. To see this, consider the path in the Cayley graph of $\Gamma/F$
   corresponding to $\pi$. A subword $v$ of $\pi$ with $\rho(v) \in F$ corresponds exactly to a loop in the path in $\Gamma/F$.
   We then loop-erase; that is, we can find a collection of disjoint subwords $\beta_1,...,\beta_k$ of $\pi$, 
   corresponding to loops in the path in $\Gamma/F$,
   such that if we take the word $\pi''$ obtained from $\pi$ by removing these subwords, we get a path in $\Gamma/F$ with the same endpoints
   (that is, $\rho(\pi'')$ and $\rho(\pi)$ have the same image under the map $\Gamma \to \Gamma/F$) and this path does not have any loops
   (that is, for all nonempty subwords $v$ of $\pi''$, we have $\rho(v) \notin F$).
   Note that since $\pi$ represents a geodesic path in the Cayley graph of $\Gamma$, we have
   \[
      |\rho(\pi)| = |\pi| = |\pi''| + \sum_{i=1}^k |\beta_i|,
   \]
   but on the other hand, since $\rho(\pi)$ and $\rho(\pi'')$ have the same image in $\Gamma/F$, there exists $f \in F$
   such that $f \rho(\pi'') = \rho(\pi)$, and therefore
   \[
      |\rho(\pi)| \le |f| + |\rho(\pi'')| \le \ell + |\pi''|,
   \]
   whence we conclude that
   \[
      \sum_{i=1}^k |\beta_i| \le \ell
   \]
   and hence (assuming without loss of generality that the $\beta_i$ are nonempty) also that $k \le \ell$.
   Since $\pi''$ was obtained from $\pi$ by deleting $k$ subwords from $\pi$, it is equal to the concatenation of 
   at most $k+1$ subwords of $\pi$.
   Thus, by the pigeonhole principle, there exists a subword $\pi'$ of $\pi$ with
   \[
      |\pi'| \ge \frac{|\pi''|}{k+1} \ge \frac{|\pi| - \ell}{\ell + 1},
   \]
   and such that no nonempty subword $v$ of $\pi'$ has $\rho(v) \in F$,
   as desired.
   
   Now, we construct a detour for $\pi'$. Let $f \in F \setminus \{1\}$. Since $F$ is normal, 
   $f^{\rho(\pi')} := (\rho(\pi'))^{-1} f \rho(\pi') \in F$. Taking geodesics $\gamma_1 \in A^*$ from 1 to $f$ and $\gamma_2 \in A^*$
   from 1 to $(f^{\rho(\pi')})^{-1}$ and setting
   \[
      \tilde{\pi}' := \gamma_1\pi' \gamma_2
   \]
   gives a path with $\rho(\tilde{\pi}') = f \rho(\pi') (f^{\rho(\pi')})^{-1} = \rho(\pi')$ (that is, $\tilde{\pi}'$ has the same endpoints as $\pi'$), and
   \[
      |\tilde{\pi}'| = |\pi'| + |\gamma_1| + |\gamma_2| \le |\pi'| + 2 \ell.
   \]
   We now show that if $\tilde{\pi}'$ intersects $\pi'$ it only does so in the first $\ell$ or the last $\ell$ edges of $\pi'$.
   Note that it suffices to show that the vertex sets $V(\pi')$ and $V(\tilde{\pi}')$ only intersect in the first $\ell + 1$ or the last $\ell + 1$ vertices of $V(\pi')$.
   (Recall that if $S \subset E$, then $V(S) \subset V$ is the set of vertices which are endpoints of some edge $e \in S$.)
   Note also that intersections of $V(\pi')$ and $V(\tilde{\pi}')$ correspond to initial subwords of $\pi'$ and $\tilde{\pi}'$ which have the same image 
   under $\rho: A^* \to \Gamma$.
   
   So first, suppose that for some initial subword $v$ of $\gamma_1$ and some initial subword $\pi_1$ of $\pi'$ we have that
   \[
      \rho(v) = \rho(\pi_1);
   \]
   geodesicity of $\pi'$ implies that $|\pi_1| \le |v| \le \ell$; that is,
   the intersection of $V(\pi')$ and $V(\tilde{\pi}')$ has happened in the first $\ell + 1$ vertices of $V(\pi')$.
   Similarly, if for some initial subword $v$ of $\gamma_2$ 
   and some initial subword $\pi_1$ of $\pi'$ we have
   \[
      \rho(\gamma_1 \pi' v) = \rho(\pi_1);
   \]
   taking the $v'$ to be the subword of $\gamma_2$ such that $\gamma_2 = vv'$ we get
   \[
      \rho(\pi') = \rho(\gamma_1 \pi' \gamma_2) = \rho(\pi_1) \rho(v'),
   \]
   and taking $\pi_2$ to be the subword of $\pi'$ such that $\pi' = \pi_1 \pi_2$ we get
   \[
      \rho(\pi_1) \rho(\pi_2) = \rho(\pi') = \rho(\pi_1) \rho(v') \Rightarrow
      \rho(\pi_2) = \rho(v')
   \]
   by cancellation, and
   so by geodesicity of $\pi'$  we conclude that $|\pi_2| \le |v'| \le \ell$; that is, the intersection of $V(\pi)'$ and $V(\tilde{\pi}')$
   has happened in the last $\ell + 1$ vertices.
   The only remaining case is that there exist initial subwords $\pi_1$ and $\pi_2$ of $\pi$ such that
   \[
      \rho(\gamma_1 \pi_1) = \rho(\pi_2),
   \]
   and hence
   \[
      f \rho(\pi_1) = 
      \rho(\pi_1) f^{\rho(\pi_1)} = \rho(\pi_2) \Rightarrow f^{\rho(\pi_1)}  = \rho(\pi_1)^{-1} \rho(\pi_2) \in F \setminus \{1\}.
   \]
   But since $\pi_1$ and $\pi_2$ are both initial subwords of $\pi'$, there exists some subword $\pi_3$ of $\pi$ such that
   $\rho(\pi_1)^{-1} \rho(\pi_2) = \rho(\pi_3)^{\pm}$, but then $\pi_3$ is a subword of $\pi'$ with $\rho(\pi_3) \in F \setminus \{1\}$,
   which contradicts our construction of $\pi'$, so this case cannot occur. 
   Thus we have proven that $V(\pi')$ and $V(\tilde{\pi}')$ only intersect in the first $\ell +1 $ or the last $\ell + 1$ vertices of $V(\pi')$,
   and so in particular $|\pi' \setminus \tilde{\pi}'| \ge |\pi'| - 2\ell$.
   
   Finally, if $\pi = \alpha \pi' \omega$, set $\tilde{\pi} = \alpha \tilde{\pi}' \omega$. We then have that $\rho(\pi) = \rho(\tilde{\pi})$,
   that is, $\pi$ and $\tilde{\pi}$ have the same endpoints, and
   \[
      |\tilde{\pi}| - |\pi| = |\tilde{\pi}'| - |\pi'| \le 2 \ell,
   \]
   while
   \[
      |\pi \setminus \tilde{\pi}| \ge |\pi' \setminus \tilde{\pi}'| \ge |\pi'| - 2 \ell \ge \frac{ |\pi| - \ell }{\ell + 1} - 2 \ell.
   \]
   Thus, given $\epsilon > 0$, if we choose $C$ large enough that $(2 \ell) \left(\frac{ C - \ell }{\ell + 1} - 2 \ell \right)^{-1} \le \epsilon$,
   then for any unique geodesic $\pi$ with $|\pi| \ge C$,
   we have that
   \[
      |\tilde{\pi}| - |\pi| \le \epsilon |\pi \setminus \tilde{\pi}|,
   \]
   that is, $\tilde{\pi}$ is an $\epsilon$-detour for $\pi$.
\end{proof}

For other classes of groups, we will need to treat two cases separately: either we can construct a detour for our unique geodesic
with a simple construction, or our unique geodesic has a very special form.
This is made precise in the following lemma, which we use several times below:
\begin{lemma} \label{lem:grouptodetour}
   Let $\Gamma$ be a finitely generated group and let $H \le \Gamma$ be a subgroup of finite index. Let $G$ be a Cayley graph
   associated to a generating set $S$ for $\Gamma$, $A$ be the relevant alphabet.
   Suppose that there exists some $z \in H \setminus \{1\}$ such that for all $h \in H$
   \[
      |z^h| := |h^{-1} z h| = |z|.
   \]
   Fix $w \in A^*$ a geodesic from $1$ to $z$, and for each $H$-conjugate $z^h$ of $z$
   fix $w^{h} \in A^*$ a geodesic from $1$ to $z^h$ (note the property of $z$
   ensures that $|w|=|w^h|$ for all $h \in H$). 
   Let $\pi \in A^*$ be a unique geodesic in $G$ such that $\rho(\pi) \in H$. Then:
   \begin{enumerate}
      \item 
      If there exist $\alpha, \omega \in A^*$, $h \in H$ such that $\rho(\alpha), \rho(\omega) \in H$ and $\pi = \alpha w^h \omega$,
      then $\pi$ is an initial subpath of $(w^{h'})^N$ for some $h' \in H$, $N < \infty$.
      If $\pi = \alpha (w^h)^{-1} \omega$ then $\pi$ is an initial subpath of $((w^{h'})^{-1})^N$ for some $h' \in H, N < \infty$.
      \item Suppose that the condition above fails, that is, for all choices of $\alpha, \omega \in A^*$ and $h \in H$ with
      $\rho(\alpha), \rho(\omega) \in H$, we have
      $\pi \ne \alpha w^h \omega$ and $\pi \ne \alpha (w^h)^{-1} \omega$. Then the path
      \[
         \pi' := w \pi (w^{\rho(\pi)})^{-1}
      \]
      has the same endpoints as $\pi$ and satisfies
      \[
         |V(\pi') \cap V(\pi) \cap H| \le 2(|w|+1),
      \]
      and hence
      \[
         | V(\pi) \setminus V(\pi') | \ge | (V(\pi) \setminus V(\pi')) \cap H | = |V(\pi) \cap H| - |V(\pi' \cap \pi \cap H)| \ge |V(\pi) \cap H| - 2(|w| +1)
      \]
      so that
      \[
         |\pi \setminus \pi'| \ge \frac{1}{2} |V(\pi) \setminus V(\pi')| \ge \frac{1}{2}|V(\pi) \cap H| - |w| - 1.
      \]
   \end{enumerate}
\end{lemma}

Before proving the lemma, we give context by proving the implications we need it for.
\begin{prop} \label{prop:Zdetour}
   Let $G$ be a Cayley graph of $\Z$. Either $A = \{a, a^{-1}\}$ for some $a \in A$, in which case
   $G$ is isomorphic to the standard Cayley graph of $\Z$, or $G$ admits detours.
\end{prop}
\begin{proof}
   If $A = \{a,a^{-1}\}$, then clearly $G$ is isomorphic to the standard Cayley graph of $\Z$.
   Assume that there exists $s \in A \setminus \{a,a^{-1}\}$ for some $a \in A$. Letting $H = \Gamma = \Z$, $w=a$,
   we see that the assumptions of the lemma are satisfied, since $\Z$ is abelian and hence the conjugation action is trivial
   (we also set $w^h = a$ for all $h \in H$).
   Therefore, if $\pi$ is any unique geodesic in $\Gamma$, either $\pi = a^N, \pi = (a^{-1})^N$, or $\pi$ does not contain $a$ or $a^{-1}$.
   By part (2) of the lemma, in the latter case, setting $\pi' := a \pi a^{-1}$ gives a path with the same endpoints such that\footnote{
   By looking more closely at the paths, one can see that the prefactor of $\frac{1}{2}$ is not necessary;
   the crude bound $|\pi \setminus \pi'| \ge \frac{1}{2} |V(\pi) \setminus V(\pi')|$ from which it comes could be avoided by a more
   careful argument,  but this is not necessary for our purposes.}
   \[
      |\pi \setminus \pi'| \ge \frac{1}{2}|\pi| - 2.
   \]
   Now consider the former case, i.e. $\pi = a^N$ or $\pi = (a^{-1})^N$. Note that taking $H = \Gamma = \Z$ and $w=s$
   also satisfies the hypotheses of the lemma, and since $s, s^{-1}$ do not appear in $\pi$, by part (2) of the lemma,
   taking $\pi' := s \pi s^{-1}$ gives a path with the same endpoints as $\pi$ and
   \[
      |\pi \setminus \pi'| \ge \frac{1}{2}|\pi| - 2.
   \]
   Thus, for any $\epsilon > 0$, if $C$ is sufficiently large that $2 \le \epsilon(\frac{C}{2} - 2)$ then we have that if $\pi$ is 
   a unique geodesic with $|\pi| \ge C$ then with the above construction of $\pi'$ we have
   \[
      |\pi'| - |\pi| = 2 \le \epsilon |\pi \setminus \pi'|,
   \]
   as desired.
\end{proof}
\begin{prop} \label{prop:dihedraldetour}
   Let $G$ be a Cayley graph of $\Gamma := \Z/2 * \Z/2$. Either $G$ is a reduced Cayley graph associated
   to a generating set $S$ such that $f(S)$ consists of exactly two elements of order 2, in which case $G$ is isomorphic to the standard
   Cayley graph of $\Z$, or $G$ admits detours.
\end{prop}
\begin{proof}
   If $G$ is a reduced Cayley graph with $A = f(S)$ consisting of exactly two elements of order 2, then $G$ is 
   a connected regular infinite graph of degree 2, and so it is isomorphic to the standard Cayley graph of $\Z$.
   We want to show that in any other case, $G$ admits detours.
   
   Recall that all nonidentity elements of $\Gamma = \Z/2 * \Z/2 \cong \Z \rtimes \Z/2$ have either infinite order or order two.
   Hence, the remaining cases to check are (1): $f(S)$ contains an element of infinite order; 
   (2): $f(S)$ contains only elements of order 2 but $G$ is an unreduced Cayley graph;
   (3): $f(S)$ contains only elements of order 2, $G$ is a reduced Cayley graph, and $|f(S)| \ge 3$.
   Also recall that any product of two distinct elements of order 2 is an element of infinite order,
   and that for any element $x \in \Gamma$ of infinite order, for any $g \in \Gamma$ we have $g^{-1}x g = x^{-1}$ if $g$ has order 2
   and $g^{-1} x g = x$ otherwise.
    
   So assume that for some $z \in A$, $\rho(z)$ has infinite order. Then for any $g \in \Gamma$ we have
   $\rho(z)^g = \rho(z)$ or $\rho(z)^g = \rho(z)^{-1} = \rho(z^{-1})$.
   Therefore if we take $H = \Gamma$ and $w = z$, the hypotheses of the lemma are satisfied, and for each $h \in H$ we can choose $w^h$
   to be either $z$ or $z^{-1}$.
   Let $\pi$ be a unique geodesic in $G$. By the lemma, either $\pi = z^N$, $\pi = (z^{-1})^N$ or $\pi$ does not contain $z$ or $z^{-1}$.
   In the latter case, by the lemma, 
   taking $\pi' := z \pi z^{-1}$ gives a path with the same endpoints and $|\pi \setminus \pi'| \ge \frac{1}{2}|\pi| - 2$.
   In the former cases, take $a \in A$ with $\rho(a)$ of order 2 (any generating set of $\Gamma$ must contain such an element).
   Taking $\pi' := a \pi^{-1} a$ then gives a path with the same endpoints as $\pi$, and $V(\pi) \cap V(\pi')$ is equal to the endpoints
   of $\pi$ and $\pi'$, since for any $0 \le n,m \le N$, $\rho(a z^{\pm n}) \ne \rho(z^{\pm m})$, since the two elements lie in distinct
   cosets of $\langle \rho(z) \rangle \le \Gamma$. Hence in fact $|\pi \setminus \pi'| = |\pi|$.
   
   Next, assume that $f(S)$ contains only elements order 2, but $G$ is an unreduced Cayley graph. In this case, every edge
   of $G$ is a double edge, so there are no unique geodesics, and so $G$ admits detours.
   
   Lastly, assume that $f(S)$ contains only elements of order 2, $G$ is a reduced Cayley graph, and $|f(S)| \ge 3$.
   Pick $a,b,c \in A$ such that $\rho(a),\rho(b),\rho(c)$ are all distinct. 
   Since $\rho(ab)$ has infinite order, again taking $w=ab$, $H = \Gamma$ satisfies
   the hypotheses of the lemma, and we can choose for each $h \in H$ either $w^h = ab$ or $w^h = ba$.
   Hence, letting $\pi$ be a unique geodesic in $G$, we have that either $\pi$ does not contain $ab$ or $ba$ as a subword,
   or that $\pi$ is a subword of $(ab)^N$ for some $N$. In the first case, by the lemma, we have that
   taking $\pi' := (ab) \pi (ba)^{(-1)^{|\pi|}}$ gives a path with the same endpoints as $\pi$ with $|\pi \setminus \pi'| \ge \frac{1}{2}|\pi| - 3$.
   In the second case, take $\pi' := (ac) \pi (ca)^{(-1)^{|\pi|}}$. Then $\pi$ and $\pi'$ have the same endpoints,
   and we claim that $|V(\pi') \cap V(\pi)|$ is contained in the union of the first two and last two vertices of $V(\pi')$,
   implying that $|\pi \setminus \pi'| \ge |\pi| - 3$. To see this, suppose to the contrary that for some
   initial subpaths $\pi_1, \pi_2$ of $\pi$ we had
   \[
      \rho( (ac) \pi_1 ) = \rho( \pi_2).
   \]
   We have $\rho( (ac) \pi_1 ) = \rho( \pi_1 (ac)^{(-1)^{|\pi_1|}} )$. Moreover, there is some subpath $\tilde{\pi}$ of $\pi$ such that
   either $\pi_2 \tilde{\pi} = \pi_1$ or $\pi_1 \tilde{\pi} = \pi_2$. 
   Then in the first case, by cancellation we have
   \[
      \rho( (ac)^{(-1)^{|\pi_1|}} ) = \rho( \tilde{\pi} ).
   \]
   As a subpath of $\pi$, $\tilde{\pi}$ is uniquely geodesic, but $(ac)$ and $(ca)$ are both geodesics, since $f(S)$ only contains elements
   of order $2$ and hence any path to an element of infinite order has length at least 2. Hence we have that $\tilde{\pi} = (ac)^{(-1)^{|\pi_1|}}$,
   contradicting our assumption that $\pi$ is a subpath of $(ab)^N$. 
   In the case that $\pi_1 \tilde{\pi} = \pi_2$, we similarly conclude that $\tilde{\pi} = (ac)^{(-1)^{|\pi_1|+1}}$, which is similarly a contradiction.
   
   Thus, in any of these three cases, given a unique geodesic $\pi$, we can produce a path $\pi'$ with the same endpoints such
   that $|\pi \setminus \pi'| \ge \frac{1}{2}|\pi| - 3$ and $|\pi'| \le |\pi| + 4$. Thus, given $\epsilon > 0$, if we choose $C$
   sufficiently large that $4 \le \epsilon(\frac{1}{2}C - 3)$, then for any self avoiding path $\pi$ of length at least $C$ have have $\pi'$ with
   \[
      |\pi'| - |\pi| \le 4 \le \epsilon |\pi \setminus \pi'|,
   \]
   as desired.
\end{proof}

\withcenter
\begin{proof}
   First, we show that we can assume without loss of generality that $H$ is not cyclic.
   Suppose that $H$ were cyclic; then $\Gamma$ would be virtually $\Z$, and therefore (see Lemma 11.4 on page 102
   of \cite{Hempel})
   there is a finite normal subgroup $F \unlhd \Gamma$ such that $\Gamma/F$ is isomorphic to either $\Z$
   or $\Z/2 * \Z/2$.
   If $F$ is not trivial, then Proposition \ref{prop:finitenormalsubgroup} tells us that any Cayey graph of $\Gamma$
   admits detours. If $F$ is trivial, then $\Gamma$ itself is isomorphic to either $\Z$ or $\Z/2 * \Z/2$, and these
   cases are excluded by assumption.
   
   So assume that $H$ is a non-cyclic finite index subgroup of $\Gamma$ with nontrivial center. Fix $z \ne 1$ a nontrivial central element of $H$
   of minimal distance to 1
   and consider a geodesic path $w \in A^*$ from $1$ to $z$. Since by definition $z^h = z$ for all $h \in H$, this choice of $w$ and $H$
   satisfy the conditions of the lemma; we also set $w^h = w$ for all $h \in H$.
   
   Take a unique geodesic $\tilde{\pi}$ in $G$. By the pigeon-hole principle, there is some $t \in \Gamma$ such that
   at least $\frac{1}{[\Gamma: H]} |V(\pi)|$ of the vertices in $V(\pi)$ lie in the coset $tH$.
   Denote by $\eta_i$ the subpath of $\tilde{\pi}$ starting at the $i^{th}$ such vertex and ending at the $(i+1)^{th}$ such vertex.
   Set $\pi := \eta_1 \cdots \eta_M$, the subpath of $\tilde{\pi}$ from the first such vertex to the last such.
   Note that each $\rho(\eta_i) \in H$, $\rho(\pi) \in H$, and that the $\eta_i$ are minimal in the sense that
   if $\alpha$ is a proper initial or final subword of some $\eta_i$, then $\rho(\alpha) \notin H$.
   
   Now, by the lemma, either $\pi$ is an initial subpath of $w^N$ or $(w^{-1})^N$ for some $N < \infty$ 
   or there are no $\alpha, \omega \in A^*$ with $\rho(\alpha), \rho(\omega) \in H$ and either $\pi = \alpha w \omega$ or 
   $\pi = \alpha w^{-1} \omega$. In the latter case, the lemma also tells us that taking $\pi' := w \pi w^{-1}$ gives a
   path with the same endpoints as $\pi$ and
   \[
      |\pi \setminus \pi'| \ge \frac{1}{2}|V(\pi) \cap H| - |w| - 1.
   \]
   
   So consider the case that $\pi$ is an initial subpath of $w^N$ (the case that $\pi$ is an initial subpath of $(w^{-1})^N$ is
   exactly analgous). First note that as long as $|\pi| > |w|$, this implies
   that all the $\eta_i$ are equal. For suppose that $\eta_1 \cdots \eta_l = w$; then we have
   \[
      \rho(\pi) = \rho(w \eta_{l+1} \cdots \eta_M) = \rho(\eta_{l+1} w \eta_{l+2} \cdots \eta_M) 
      = \rho( \eta_{l+1} \eta_1 \cdots \eta_{l} \eta_{l+2} \cdots \eta_M ),
   \]
   so unique geodesicity, together with the fact that the decomposition $w = \eta_1 \cdots \eta_M$ is uniquely specified by $w$
   and $H$, implies that $\eta_i = \eta_{i+1}$ for all $i=1,...,l-1$, so $w = \eta_1^{l}$, and hence $\pi = \eta_1^M$.
   Now take a geodesic path $\alpha$ from 1 to some element of $H \setminus \langle \rho(\eta_1) \rangle$ (which is possible since $H$
   is not cyclic) with minimal distance to 1. Let $0 \le r < l$ be minimal such that $M + r$ is a multiple of $l$. Then we set
   \[
      \pi' := \alpha \eta_1^{M+r} \alpha^{-1} \eta_1^{-r}.
   \]
   Since $\rho(\eta_1^{M+r})$ is a power of $\rho(\eta_1^l )= \rho(w) = z$, it is central in $H$, and hence we see that $\pi$ and $\pi'$ have the
   same endpoints. We further claim that 
   \[
      V(\pi') \cap V(\pi) \cap H \subset V(\alpha) \cup \rho(\alpha \eta_1^{M+r}) V(\alpha^{-1} \eta_1^{-r}),
   \]
   which implies that
   \[
      |V(\pi') \cap V(\pi) \cap H | \le 2|\alpha| + |w| + 2,
   \]
   hence $|V(\pi \setminus \pi')| \ge |V(\pi) \cap H| - (2|\alpha| + |w| + 2)$ and 
   $|\pi \setminus \pi'| \ge \frac{1}{2} |V(\pi) \setminus V(\pi')| \ge \frac{1}{2} |V(\pi) \cap H| - |\alpha| - |w| - 1$.
   
   To see this, suppose to the contrary that for some $0 \le i \le M+r, 0 \le j \le M$ we had
   \[
      \rho( \alpha \eta_1^i ) = \rho( \eta_1^j).
   \]
   Then cancellation gives us
   \[
      \rho( \alpha ) = \rho(\eta_1)^{j-i} \in \langle \rho(\eta_1) \rangle,
   \]
   contradicting our choice of $\alpha$.
   
   Thus, in either case, given $\pi$, we get $\pi'$ with the same endpoints satisfying
   \[
      |\pi'| \le |\pi| + 2(|\alpha| + |w|)
   \]
   and
   \[
      |\pi \setminus \pi'| \ge \frac{1}{2}|V(\pi) \cap H| - |\alpha| - |w| - 1.
   \]
   If $\tilde{\pi} = \pi_1 \pi \pi_2$, set $\tilde{\pi}' := \pi_1 \pi' \pi_2$.
   Then we again have
   \[
      |\tilde{\pi}'| - |\tilde{\pi}| = |\pi'| - |\pi| \le 2(|\alpha| + |w|)
   \]
   and
   \begin{align*}
      |\tilde{\pi} \setminus \tilde{\pi}'| 
      &= |\pi \setminus \pi'| \ge \frac{1}{2}|V(\pi) \cap H| - |\alpha| - |w| - 1 \\
      &= \frac{1}{2}|V(\tilde{\pi}) \cap tH| - |\alpha| - |w| - 1 \ge \frac{1}{2[\Gamma:H]} |V(\tilde{\pi})| - |\alpha| - |w| - 1.
   \end{align*}
   
   Note that $|w| = |z|$ is a constant independent of the path $\tilde{\pi}$. Moreover,
   \[
      |\alpha| \le \sup_{h \in H} \inf_{h' \in H \setminus \langle h \rangle} |h'| =: K < \infty,
   \]
   where $K$ is a constant also independent of the path $\tilde{\pi}$.
   To see that $K$ is finite, take a finite generating set $S'$ for $H$ ($H$ is finitely generated by Schreier's lemma, as a finite index subgroup
   of the finitely generated group $\Gamma$.) Since $H$ is not cyclic, for every $h \in H$, there is some $s \in S'$ such that
   $s \notin \langle h \rangle$, and so $K \le \sup_{s \in S'} |s| < \infty$, as desired.
   
   Thus, given $\epsilon > 0$, if we choose $C > |w|$ sufficiently large that 
   $2(K + |w|) \le \epsilon\left(\frac{1}{2[\Gamma : H]} C - K - |w| -1 \right)$, then for any unique geodesic $\tilde{\pi}$ in $G$
   of length at least $C$ we have $\tilde{\pi'}$ with the same endpoints such that
   \[
      |\tilde{\pi'}| - |\tilde{\pi}| \le \epsilon |\tilde{\pi} \setminus \tilde{\pi}'|,
   \]
   as desired.
\end{proof}
\begin{rmk}
   The preceding three propositions are in some sense as far as we can push the lemma. If there exists $z \ne 1$
   such $|z^h|=|z|$ for all $h \in H$, in particular the $H$-orbit of $z$ under the conjugation action is finite;
   hence by the orbit stabilizer theorem, $\mathrm{Stab}_H(z) \le H$ is a finite index subgroup of $H$
   with nontrivial center (the center contains $z$). Thus if $H$ is itself finite index in $\Gamma$, 
   $\Gamma$ has a finite index subgroup with nontrivial center, and one of the preceding propositions apply.
   The lemma still holds if $H$ is not finite index, but in that case it is not clear how to use it to construct 
   detours for $G$.
\end{rmk}

Now to prove the lemma.
\begin{proof}[Proof of Lemma \ref{lem:grouptodetour}]
   Let $\pi$ be a unique geodesic with $\rho(\pi) \in H$
   and suppose that for some $\alpha, \omega \in A^*, h \in H$, we have $\rho(\alpha), \rho(\omega) \in H$
   and $\pi = \alpha w^h \omega$. (The case that $\pi = \alpha (w^h)^{-1} \omega$ is exactly analogous).
   Note that since $\rho(\alpha), \rho(w^h), \rho(\omega) \in H$, we have decompositions
   \[
      \alpha = \iota_1 \cdots \iota_j,
   \]
   \[
      w^h = \eta_1 \cdots \eta_l
   \]
   \[
      \omega = \kappa_1 \cdots \kappa_m,
   \]
   where each $\rho(\iota_i), \rho(\eta_i), \rho(\kappa_i) \in H$ and if $\gamma$ is an initial or final proper 
   subpath of some $\iota_i, \eta_i,$ or $\kappa_i$,
   then $\rho(\gamma) \ne H$. Note that this property also uniquely specifies the decomposition.
   
   Now we will show that  
   $\alpha$ is a final subword of some $(w^h)^N$ and 
   that $\omega$ is an initial subword of some $(w^h)^N$.
   First, suppose that $j \le l$. We have that
   \[
      \rho(\alpha \eta_1 \cdots \eta_l) = \rho(\alpha w^h) = \rho(w^{h \rho(\alpha)^{-1}} \alpha) = \rho(w^{h \rho(\alpha)^{-1}} \iota_1 \cdots \iota_j)
   \]
   where $|w^{h \rho(\alpha)}| = |w^h|$; hence by unique geodesicity we have that the above paths are equal. By the uniqueness of the
   decomposition, we then have that $\iota_i = \eta_{i + l - j}$ for $i=1,...,j$, that is, $\alpha$ is a final subword 
   of $w$
   If $j > l$, we have
   \[
      \rho(\iota_1 \cdots \iota_j w^h) = \rho( \iota_1 \cdots \iota_{j-l} w^{h \rho(\iota_{j-l+1} \cdots \iota_j)^{-1}} \iota_{j-l+1} \cdots \iota_j),
   \]
   so that by unique geodesicity, $\iota_{j-l+1} \cdots \iota_j = w^h$, that is $\alpha = \alpha' w$ for some $\alpha'$ of shorter length.
   Thus, by induction, $\alpha$ is a final subword of $(w^h)^N$ for some $N$.
   The argument that $\omega$ is an initial subword of some $(w^h)^N$ is exactly analogous.
   
   Thus, $\pi = \alpha w \omega$ is a subword of some $(w^h)^N$; even more than that, $\alpha$ starts with $\eta_i$ for some $i$,
   so we see that $\pi$ is an initial subword of some $w'^N$, $w' := \eta_i \eta_{i+1} \cdots \eta_{l} \eta_1 \cdots \eta_{i-1}$.
   Note that $\rho(w') = \rho(w^h)^{\rho(\eta_1 \cdots \eta_{i-1})}
   =\rho(w^{h \rho(\eta_1 \cdots \eta_{i-1})})$, and $w'$ is a subword of $\pi$,
   so by unique geodesicity $w' = w^{h \rho(\eta_1 \cdots \eta_{i-1})}$, so we have proven (1).
   
   Now to prove (2). Let $\pi$ be a unique geodesic with $\rho(\pi) \in H$ and such that no decomposition of the form
   $\pi = \alpha (w^h)^{\pm 1} \omega$ holds for any $\rho(\alpha), \rho(\omega), h \in H$.
   Then take $\pi' := w \pi w^{-1}$. We claim that 
   \[
      V(\pi') \cap V(\pi) \cap H \subset V(w) \cup \rho(w \pi) V(w^{-1}),
   \]
   which then clearly implies that $|V(\pi') \cap V(\pi) \cap H| \le 2(|w|+1)$.
   
   To prove our claim, suppose that to the contrary there were two initial subpaths $\pi_1, \pi_2$ of $\pi$ with $\rho(\pi_1), \rho(\pi_2) \in H$
   such that
   \[
      \rho(w \pi_1) = \rho(\pi_2).
   \]
   Then we have
   \[
      \rho(w \pi_1) = \rho(\pi_1 w^{\rho(\pi_1)}) = \rho(\pi_2).
   \]
   There is some subpath $\tilde{\pi}$ of $\pi$ such that $\rho(\tilde{\pi}) \in H$ and either $\pi_1 = \pi_2 \tilde{\pi}$ or 
   $\pi_2 = \pi_1 \tilde{\pi}$. In the second case, cancellation and unique geodesicity give
   \[
      \rho(w^{\rho(\pi_1)}) = \rho(\tilde{\pi}) \Rightarrow w^{\rho(\pi_1)} = \tilde{\pi}
   \]
   which contradicts our assumption on $\pi$, since $\pi = \pi_1 \tilde{\pi} \omega$ for some $\rho(\omega) \in H$
   and the above equation says that $\tilde{\pi}$ is a geodesic from 1 to $z^{\rho(\pi_1)}$.
   The first case gives $\rho(\tilde{\pi}) = \rho(w^{\rho(\tilde{\pi})})^{-1}$,
   which is similarly a contradiction. So we are done.
   
   The final inequalities are consequences of straightforward algebraic manipulations, together with the inequality
   \[
      |\pi \setminus \pi'| \ge \frac{1}{2} |V(\pi) \setminus V(\pi')|,
   \]
   which follows from the fact that we can construct a map $V(\pi) \setminus V(\pi') \to \pi \setminus \pi'$ with fibers of size at most 2
   as follows: take each $v \in V(\pi) \setminus V(\pi')$ and associate to it an edge $e \in \pi$ such that $v$ is an endpoint of $e$;
   such an edge exists since $v \in V(\pi)$ and $e \notin \pi'$ because $v \notin V(\pi')$.
\end{proof}

\nilpotentdetours
\begin{proof}
   If $G$ is a Cayley graph of a group which is not isomorphic to $\Z$ or $\Z/2 * \Z/2$, this follows from 
   Proposition \ref{prop:withcenter}, since nilpotent groups have nontrivial center. 
   If $G$ is a Cayley graph of $\Z$, this follows from Proposition \ref{prop:Zdetour}. 
   If $G$ is a Cayley graph of $\Z/2 * \Z/2$, this follows from Proposition \ref{prop:dihedraldetour}.
\end{proof}

\bibliography{strictineqbib}
\bibliographystyle{plain}
\end{document}